%% file: 5_high_acc.tex
\documentclass[11pt]{article}

\usepackage[margin=1in]{geometry}
\usepackage[T1]{fontenc}
\usepackage{amsmath,amssymb,amsthm,mathtools,bm,mathrsfs,euscript}
\usepackage[Symbolsmallscale]{upgreek}
\usepackage{booktabs,array,enumitem,longtable,microtype}
\usepackage{aliascnt}
\usepackage{xcolor}
\usepackage{tikz}
\usepackage{placeins}
\usetikzlibrary{arrows.meta,positioning}
\usepackage[colorlinks=true,linkcolor=blue!55!black,citecolor=blue!55!black,
  urlcolor=blue!55!black]{hyperref}
\usepackage[nameinlink,capitalize,noabbrev]{cleveref}

\allowdisplaybreaks

\newtheorem{theorem}{Theorem}[section]

\newaliascnt{proposition}{theorem}
\newtheorem{proposition}[proposition]{Proposition}
\aliascntresetthe{proposition}

\newaliascnt{lemma}{theorem}
\newtheorem{lemma}[lemma]{Lemma}
\aliascntresetthe{lemma}

\newaliascnt{corollary}{theorem}
\newtheorem{corollary}[corollary]{Corollary}
\aliascntresetthe{corollary}

\newaliascnt{assumption}{theorem}

\aliascntresetthe{assumption}

\theoremstyle{definition}
\newaliascnt{definition}{theorem}

\aliascntresetthe{definition}

\newtheorem{gadget}{Gadget}

\newaliascnt{algorithm}{theorem}
\newtheorem{algorithm}[algorithm]{Algorithm}
\aliascntresetthe{algorithm}

\theoremstyle{remark}
\newaliascnt{remark}{theorem}
\newtheorem{remark}[remark]{Remark}
\aliascntresetthe{remark}

\crefname{theorem}{Theorem}{Theorems}
\Crefname{theorem}{Theorem}{Theorems}
\crefname{proposition}{Proposition}{Propositions}
\Crefname{proposition}{Proposition}{Propositions}
\crefname{lemma}{Lemma}{Lemmas}
\Crefname{lemma}{Lemma}{Lemmas}
\crefname{corollary}{Corollary}{Corollaries}
\Crefname{corollary}{Corollary}{Corollaries}
\crefname{assumption}{Assumption}{Assumptions}
\Crefname{assumption}{Assumption}{Assumptions}
\crefname{gadget}{Gadget}{Gadgets}
\Crefname{gadget}{Gadget}{Gadgets}
\crefname{algorithm}{Algorithm}{Algorithms}
\Crefname{algorithm}{Algorithm}{Algorithms}
\crefname{definition}{Definition}{Definitions}
\Crefname{definition}{Definition}{Definitions}
\crefname{remark}{Remark}{Remarks}
\Crefname{remark}{Remark}{Remarks}

\newcommand{\R}{\mathbb R}
\newcommand{\E}{\mathbb E}
\newcommand{\Pp}{\mathbb P}
\newcommand{\cN}{\mathsf N}
\newcommand{\Id}{I}
\newcommand{\T}{\mathsf T}
\newcommand{\ii}{\boldsymbol i}
\newcommand{\dd}{\mathrm d}
\newcommand{\e}{\mathrm e}
\newcommand{\eps}{\varepsilon}
\newcommand{\TV}{\mathsf{TV}}
\newcommand{\KL}{\mathsf{KL}}
\newcommand{\Law}{\operatorname{Law}}

\newcommand{\poly}{\operatorname{poly}}
\newcommand{\polylog}{\operatorname{polylog}}
\newcommand{\prox}{\operatorname{prox}}
\newcommand{\Lip}{\operatorname{Lip}}
\newcommand{\op}{\mathrm{op}}
\newcommand{\tr}{\operatorname{tr}}
\newcommand{\divg}{\operatorname{div}}
\newcommand{\wtO}{\widetilde O}
\newcommand{\wh}{\widehat}
\newcommand{\ol}{\overline}

\newcommand{\cF}{\mathcal F}
\newcommand{\comp}{\mathsf{c}}
\newcommand{\norm}[1]{\lVert #1\rVert}
\newcommand{\abs}[1]{\lvert #1\rvert}
\newcommand{\ip}[2]{\langle #1,#2\rangle}

\newcommand{\deq}{\coloneqq}
\newcommand{\clipB}{\mathtt{B}}
\newcommand{\cloud}{\mathscr C}

\newcommand{\HS}{\mathrm{HS}}
\newcommand{\mmid}{\mathbin{\|}}

\numberwithin{equation}{section}
\renewcommand{\le}{\leqslant}
\renewcommand{\leq}{\leqslant}
\renewcommand{\ge}{\geqslant}
\renewcommand{\geq}{\geqslant}

\title{High-accuracy simulation of Picard HMC, part I: \\ Gaussian cloud correction}
\author{
 Fan Chen\thanks{Department of Electrical Engineering and Computer Science,
 Massachusetts Institute of Technology.
 Email: \href{mailto:fanchen@mit.edu}{\texttt{fanchen@mit.edu}}.}
 \and
 Sinho Chewi\thanks{Department of Statistics and Data Science, Yale University.
 Email: \href{mailto:sinho.chewi@yale.edu}{\texttt{sinho.chewi@yale.edu}}.}
 \and
 Jianfeng Lu\thanks{Department of Mathematics, Duke University.
 Email: \href{mailto:jianfeng@math.duke.edu}{\texttt{jianfeng@math.duke.edu}}.}
 \and
 Matthew S. Zhang\thanks{Department of Mathematics, Massachusetts Institute of Technology.
 Email: \href{mailto:shuns436@mit.edu}{\texttt{shuns436@mit.edu}}.}
}
\date{September 2026}
\begin{document}
\maketitle
\begin{abstract}
We study the problem of sampling from a continuous density
$\pi\propto \exp(-V)$ on $\R^d$, where
$V\in C^2(\R^d)$ has a $\beta$-Lipschitz gradient and $\pi$ satisfies a
logarithmic Sobolev inequality with constant $\alpha^{-1}$, and write
$\kappa\deq\beta/\alpha$. We introduce the \emph{Gaussian cloud sampler}, which achieves total variation
accuracy $\varepsilon$ using $\widetilde O(\kappa d^{1/5}\polylog(1/\varepsilon))$ gradient
queries in expectation. The algorithm uses first-order rejection sampling (FORS) to correct the law of smoothed Picard HMC trajectories.
To do so, we represent the iterates of the ideal Picard iteration via \emph{Gaussian clouds},
whose centers are never evaluated, and we develop a suite of likelihood correction gadgets which only use samples from this indirect cloud representation.
\end{abstract}
\setcounter{tocdepth}{2}
\tableofcontents
\clearpage

\input{5_high_acc_sections/01_introduction}
\input{5_high_acc_sections/02_smoothing}
\input{5_high_acc_sections/03_cloud_sampler}
\input{5_high_acc_sections/04_gaussian_correction}
\input{5_high_acc_sections/05_marginal_corrections}
\input{5_high_acc_sections/06_recursive_assembly}

\newpage
\appendix
\input{5_high_acc_sections/appendix_a_rgo}
\input{5_high_acc_sections/appendix_b_gradient_realization}

\bibliographystyle{alpha}
\bibliography{ref}

\end{document}

%% file: 5_high_acc_sections/01_introduction.tex
\section{Introduction}

We study the problem of sampling from the probability measure
$\pi(\dd x)\propto\exp\{-V(x)\}\,\dd x$ on $\R^d$ using queries to
$\nabla V$. Our goal is \emph{high-accuracy sampling}: the number of
queries should depend only polylogarithmically on the inverse target
accuracy. The central question is how this cost grows with the geometry of the target $\pi$, and the ambient dimension $d$.
We introduce the \emph{Gaussian cloud sampler}, which achieves dimension dependence $d^{1/5}$ without higher-order smoothness assumptions.

For the main result, let $V\in C^2(\R^d)$ have a $1$-Lipschitz gradient,
and assume that $\pi$ satisfies a logarithmic Sobolev inequality (LSI)
with constant $\kappa\ge1$: for every probability measure $\mu$ with a
smooth density,
\begin{equation}\label{eq:LSI}
 \KL(\mu\mmid\pi)
 \leq\frac\kappa2\int_{\R^d}
 \Bigl\lVert\nabla\log\frac{\dd\mu}{\dd\pi}\Bigr\rVert^2\,\dd\mu\,.
\end{equation}
Also, suppose that we have access to a reference point $x_{\rm ref}$ such that
\begin{equation}\label{eq:reference-point}
 \norm{\nabla V(x_{\rm ref})}\leq\sqrt{d/\kappa}\,.
\end{equation}
In particular, we do not require a warm start initialization. Our main
result is as follows.

\begin{theorem}[Main result]\label{thm:fifth-root}
Suppose that $V\in C^2(\R^d)$ has a $1$-Lipschitz gradient, $\pi$
satisfies \eqref{eq:LSI}, and the supplied reference point satisfies
\eqref{eq:reference-point}. For every $0<\varepsilon\le1/4$, there exists
a sampler that returns a law $\widehat\pi$ with $\TV(\widehat\pi,\pi)\le\varepsilon$
using at most
\begin{align*}
    O\bigl(\kappa\,\{d+\log(\kappa d/\varepsilon)\}^{1/5}\log^{13/2}(\kappa d/\varepsilon)\bigr)
\end{align*}
gradient queries in expectation.
\end{theorem}

Since a logarithmic Sobolev inequality is implied by strong log-concavity,
the theorem applies,
in particular, to strongly log-concave targets satisfying
$0 \prec \alpha\Id\preceq\nabla^2V\preceq\beta\Id$ with condition number $\kappa = \beta/\alpha$, after a suitable rescaling.

\paragraph{Comparison with prior works.}
To place the dimension dependence in context, consider the strongly
log-concave setting. The Metropolis-adjusted Langevin algorithm
(MALA) achieves $\widetilde O(\kappa d^{1/2})$ complexity from a warm
start~\cite{Chewi+21MALA,WuSchChe22MALA}; furthermore, a warm start can be generated with the same complexity~\cite{AltChe24Warm}. The proximal sampler
also achieves $\widetilde O(\kappa d^{1/2})$ complexity through approximate
rejection sampling~\cite{FanYuaChe23ImprovedProx,Chen+26HighAccDiffusion},
with the latter work requiring only gradient queries. More recently,
exact diffusion simulation via first-order rejection sampling (FORS)
gave the improved bound $\widetilde O(\kappa^{2/3}d^{1/3})$~\cite{Chen+26ExactDiffusion}.
Then, the proximal bouncy particle sampler
(or proximal BPS) achieved $\widetilde O(\kappa^{1/2}d^{1/4})$
complexity from a warm start~\cite{ProximalBPS}.
Smoothed Picard HMC can generate the warm start, giving the overall
high-accuracy bound
$\widetilde O(\kappa^{7/6}d^{1/6}+\kappa^{1/2}d^{1/4})$~\cite{PicardHMC}.
In this discussion, we have omitted results that rely on higher-order smoothness assumptions on $V$.
See \Cref{tab:high-acc} for a summary.

\begin{table}[t]
\centering
\small
\setlength{\tabcolsep}{4pt}
\renewcommand{\arraystretch}{1.18}
\caption{High-accuracy sampling bounds for
$0\prec \alpha\Id\preceq\nabla^2V\preceq\beta\Id$, with
$\kappa=\beta/\alpha$, without higher-order smoothness assumptions.
The bounds give total variation error at most $\varepsilon$;
$\widetilde O$ suppresses logarithmic factors in $d$, $\kappa$, and
$1/\varepsilon$. Here $W$ denotes the query cost of generating a warm start;
the current best bound on $W$ is $\widetilde O(\kappa^{7/6} d^{1/6})$, achieved by smoothed Picard HMC~\cite{PicardHMC}.}
\label{tab:high-acc}
\vspace{0.5em}
\begin{tabular*}{0.75\textwidth}{@{}p{0.5\textwidth}@{\extracolsep{\fill}}>{\centering\arraybackslash}p{0.25\textwidth}@{}}
Method & Expected query bound \\
\midrule
\raggedright MALA~\cite{Chewi+21MALA,WuSchChe22MALA}
 & $\widetilde O(\kappa d^{1/2}+W)$ \\
\raggedright Proximal sampler with FORS~\cite{Chen+26HighAccDiffusion}
 & $\widetilde O(\kappa d^{1/2})$ \\
\raggedright Windowed thinning BPS~\cite{LuLuo26Zigzag}
 & $\widetilde O(\kappa^{1/2}d+W)$ \\
\raggedright Windowed thinning zigzag${}^\dagger$~\cite{LuLuo26Zigzag}
 & $\widetilde O(\kappa d^{5/4}+W)$ \\
\raggedright Exact underdamped Langevin simulation~\cite{Chen+26ExactDiffusion}
 & $\widetilde O(\kappa^{2/3}d^{1/3})$ \\
\raggedright Proximal BPS~\cite{ProximalBPS}
 & $\widetilde O(\kappa^{1/2}d^{1/4}+W)$ \\
\midrule
\raggedright \textbf{Gaussian cloud sampler (\cref{thm:fifth-root})}
 & $\widetilde O(\kappa d^{1/5})$ \\
\bottomrule
\end{tabular*}
\smallskip

\begin{minipage}{0.75\textwidth}
\footnotesize
${}^\dagger$The zigzag sampler uses partial derivative queries
$\partial_jV$, $j\in\{1,\ldots,d\}$.
\end{minipage}
\end{table}

The present result therefore improves the dimension dependence of high-accuracy sampling from $d^{1/4}$ to $d^{1/5}$.
However, our result is still incomparable to the proximal BPS, which exhibits the fully accelerated square root dependence on $\kappa$ from a warm start.

We expect that the dimension dependence can still be improved further; in fact, even the current Gaussian cloud sampler could be pushed slightly beyond $d^{1/5}$ (see the discussion in \cref{ssec:roadmap}).
Beyond the rate, however, we believe that the Gaussian cloud sampler should be of broad interest to the community for its innovative algorithmic components, which we highlight in \cref{sec:fifth-root}.

\paragraph{Smoothed Picard HMC\@.}
The algorithm in this paper builds upon the smoothed Picard Hamiltonian Monte Carlo framework developed in our prior work~\cite{PicardHMC}, and that work should be considered a prerequisite for this one.

Prior to that work, the key difficulty in log-concave sampling was \emph{discretization}: how can we design discrete-time Markov chains which closely track an ideal stochastic process, and how can we control the errors?
The main insight in~\cite{PicardHMC} is that instead of sampling from $\pi$ directly, we can instead target the smoothed distribution $\pi_\eta = \pi * \cN(0, \eta I)$.
The additional smoothness allows us to apply a high-order integrator; in this case, Picard iteration with Chebyshev--Lobatto quadrature.
Then, we can go from the smoothed distribution $\pi_\eta$ back to $\pi$ using a call to a restricted Gaussian oracle (RGO)~\cite{LeeSheTia21RGO}.
This framework makes the discretization error negligible, but shifts the difficulty to a different problem, namely that of \emph{estimation}: how can we efficiently approximate the score of the smoothed distribution, given only queries to the original distribution?

In~\cite{PicardHMC}, the final sampling error was controlled by bounding the bias and variance incurred by a suitable score estimator.
In this work, we instead show how to correct the smoothed Picard HMC trajectories, thereby achieving a high-accuracy sampler.
The basic correction mechanism is based on first-order rejection sampling (FORS), introduced in~\cite{Chen+26HighAccDiffusion}.
This allows us to sample from a distribution $P$ which is a tilt of a proposal $Q$, in the sense that $\frac{\dd P}{\dd Q} \propto \exp w$, given only an unbiased estimator for $w$ (rather than evaluations of $w$ itself).

\paragraph{Gaussian cloud sampler.}
The standard correction for HMC is based on the Metropolis--Hastings
filter~\cite{Duane+1987HMC,neal11hmc}.
This leads to the Metropolized Hamiltonian Monte Carlo (MHMC) algorithm.
For the standard leapfrog integrator, the familiar $d^{1/4}$ scaling
comes from controlling the acceptance probability, even for regular
product targets such as Gaussians~\cite{Bes+13HMC}. Rigorous mixing
bounds with this dimension dependence beyond Gaussian targets use
additional Hessian regularity~\cite{CheGatJia26MHMC,ZhaAltChe26HMC};
we review these results below.

Regardless of whether we consider the Metropolis--Hastings filter, or FORS, it is unclear that we can apply these correction mechanisms directly to Picard HMC\@.
Indeed, the correction mechanisms require evaluations of the likelihood ratio between the target transition and proposal, or unbiased estimators thereof, but the Picard map is non-linear, and the resulting likelihood ratio is complicated.
This is a main reason why MHMC relies on the symplectic leapfrog integrator, as its volume preservation property is key to producing a tractable Metropolis--Hastings correction.

The \textbf{main idea} in this work is to not track the ideal Picard HMC iterates themselves, but instead to track Gaussians centered at these ideal iterates, which we call \emph{Gaussian clouds} (see \cref{fig:fifth-clouds}).
Surprisingly, this can be done using samples from the clouds, without ever knowing their means.
More specifically, given samples from the current cloud, we make a stochastic Picard update---stochastic, because we use a stochastic estimator for the score of the smoothed distribution $\pi_\eta$---and add fresh Gaussian noise. We then develop correction mechanisms that tilt the law of this proposal to the law of the next cloud.

These correction mechanisms must draw samples from Gaussians with unknown means, given only stochastic approximations of those means.
The stochastic approximations are known non-linear functions of the randomness in the Picard update.
However, we cannot use these approximations to estimate the means themselves, as that would be too costly; it would not lead to a high-accuracy algorithm. (Even worse, it would mean that we need many cloud samples at one level to draw a cloud sample at the next level, leading to the number of cloud samples growing exponentially with the Picard depth.)
Instead, we combine these stochastic approximations with FORS to develop a suite of algorithmic \emph{gadgets}, which we isolate out in \cref{sec:reusable-correction}.
We believe that these gadgets are algorithmically interesting in their own right.

A detailed overview of the algorithm is given in \cref{sec:fifth-root}.

\paragraph{Related work.}
For an exposition of the complexity of log-concave sampling, see
\cite{Chewi26Book}.

Since our algorithm can be viewed as a more efficient correction mechanism for HMC\@, we give an overview of that literature.
HMC originates in lattice field theory~\cite{Duane+1987HMC} and has
become a standard tool in statistics~\cite{neal11hmc}. Its usual
implementation refreshes a Gaussian momentum, takes several leapfrog
steps, and accepts or rejects the resulting proposal according to its
Hamiltonian error. Reversibility and volume preservation make this
correction tractable and ensure exact invariance of the target; see
\cite{BouSan18HMC} for the principles underlying this
construction.

For sufficiently regular i.i.d.\ product targets at
stationarity, Beskos et al.~\cite{Bes+13HMC} show that a leapfrog step
size of order $d^{-1/4}$ maintains a non-vanishing acceptance
probability over a fixed integration time. This gives $d^{1/4}$
gradient evaluations per trajectory, but does not by itself establish
a mixing bound from a prescribed initialization. Qualitative
convergence theory, including irreducibility and geometric ergodicity,
was developed in~\cite{Liv+19HMC,DurMouSak20HMC}. Quantitative coupling
bounds were obtained in~\cite{BouEbeZim20HMC,BouObe24Metropolis}.

Non-asymptotic analyses have progressively narrowed the gap between
these predictions and implementable high-accuracy algorithms.
Chen, Dwivedi, Wainwright, and Yu~\cite{chenetal2020hmc} established
conductance-based bounds for multi-step leapfrog MHMC\@.
Then, Chen, Gatmiry, and Jiang~\cite{CheGatJia26MHMC}
obtained the $d^{1/4}$ dimension dependence for general targets from a
warm start under isoperimetry and a Frobenius-Lipschitz Hessian
assumption.

Initialization is a separate source of difficulty. For Gaussian targets,
Apers, Gribling, and Szil\'agyi~\cite{ApeGriSzi24HMCGaussian}
combine unadjusted HMC initialization with long, randomized
integration times to obtain an overall
$\widetilde O(\kappa^{1/2}d^{1/4})$ bound in total variation.
Zhang, Altschuler, and Chewi~\cite{ZhaAltChe26HMC}
extend the algorithmic warm start approach beyond Gaussians under
strong convexity, smoothness, and the Frobenius-Lipschitz
Hessian condition. A variant of unadjusted HMC produces a warm start in
R\'enyi divergence, after which MHMC achieves high accuracy. Both stages have $d^{1/4}$ dimension
dependence when the condition number and third derivative bound are
fixed.

High-order geometric integrators and polynomial approximation can further
improve dimension dependence under additional regularity or
structure~\cite{MS17, BouSan18HMC,LeeSonVem18ODE}.

Lee, Shen, and
Tian~\cite{leeshetia21malalower} proved lower bounds for the relaxation
time of MHMC over well-conditioned targets in terms of the number of leapfrog steps.

We extensively make use of the proximal sampler framework and the RGO developed
in~\cite{LeeSheTia21RGO}, and the convergence guarantees established under a logarithmic Sobolev inequality in~\cite{Chen+22ProxSampler}.
Our likelihood corrections build on
FORS~\cite{Chen+26HighAccDiffusion}, which implements rejection sampling
through unbiased estimators rather than exact potential evaluations.
Exact diffusion simulation~\cite{Chen+26ExactDiffusion} applies this
principle to path laws using Girsanov's theorem. We instead apply it to Gaussian
cloud laws, with separate corrections for fluctuations and means. This
allows us to propagate a high-order computation whose intermediate
values are accessible only through samples.

\paragraph{AI usage.}
The main ideas underlying the algorithm and its analysis emerged through extensive interactions with GPT-5.6 Sol.
Specifically, after obtaining the proximal BPS~\cite{ProximalBPS}, we prompted GPT to develop a high-accuracy sampler based on applying FORS correction to smoothed Picard HMC~\cite{PicardHMC}.
We subsequently simplified the algorithm, checked the arguments, and prepared the manuscript (with further assistance from GPT), and we take full responsibility for the contents of this work.

%% file: 5_high_acc_sections/02_smoothing.tex
\section{Preliminaries}\label{sec:preliminaries}

We recall the following algorithmic ingredients from~\cite{PicardHMC}.
Throughout, $c, C > 0$ denote universal constants that may change from line to line.
Except in \cref{sec:proximal-composition}, the construction and
analysis use the stronger assumption
\begin{equation}\label{eq:HVP-regularity}
 \kappa^{-1}\Id\preceq\nabla^2V(x)\preceq\Id
 \quad\text{for all }x\in\R^d\,.
\end{equation}

\subsection{Gaussian smoothing}

Write $\pi_\eta\deq\pi*\cN(0,\eta\Id)$. Let $V_\eta\deq-\log\pi_\eta$, up to an additive constant, and define the smoothed gradient $g_\eta\deq\nabla V_\eta$. The restricted Gaussian oracle (RGO) is the distribution
\begin{equation*}
 R_{\eta,y}(\dd x)\propto
 \exp\Bigl\{-V(x)-\frac{\norm{x-y}^2}{2\eta}\Bigr\}\,\dd x\,.
\end{equation*}
We also write $R_\eta$ for the Markov kernel with $R_\eta(y,\cdot) \deq R_{\eta,y}$.
Then, we have the identities $\pi_\eta R_\eta = \pi$ and $g_\eta(y) = \E_{X\sim R_{\eta,y}} \nabla V(X)$.
Moreover, the smoothed potential satisfies the bounds
\begin{align*}
 \frac1{\kappa+\eta}\,\Id&\preceq\nabla^2V_\eta
 \preceq\frac1{1+\eta}\,\Id\,.
\end{align*}
These facts are taken from \cite[Section 3.1 and Lemma 4.1]{PicardHMC}. We also use the moment estimate
\begin{equation}\label{eq:score-moment}
 \norm{g_\eta(Y)}_{L^p}\leq C\sqrt{d+p}\,,
 \qquad Y\sim\pi_\eta\,,\quad p\geq2\,,
\end{equation}
which is the $k=0$ case of \cite[Lemma 5.8]{PicardHMC}.

Since the smoothed score $g_\eta$ is not directly computable from gradient queries, we consider the stochastic estimator
\begin{equation*}
 \widehat g_\eta(y;G)\deq
 \nabla V\bigl(\prox_{\eta V}(y)+\sqrt\eta\,G\bigr)\,,
 \qquad G\sim\cN(0,\Id)\,.
\end{equation*}
By \cite[Lemma 4.2]{PicardHMC}, $\widehat g_\eta(y;G)$ is $1$-Lipschitz in $y$, $\sqrt\eta$-Lipschitz in $G$, and its mean $\overline g_\eta(y)\deq\E\widehat g_\eta(y;G)$ satisfies
\begin{equation*}
 \norm{g_\eta(y)-\overline g_\eta(y)}\leq\eta^{3/2}\sqrt d\,.
\end{equation*}

\subsection{Picard iteration and Chebyshev--Lobatto quadrature}

As in \cite{PicardHMC}, we use $J$ Chebyshev--Lobatto nodes on $[0,h]$, defined via
\begin{equation*}
 t_j\deq\frac h2\,\Bigl(1-\cos\frac{(j-1)\uppi}{J-1}\Bigr)\,,
 \qquad j\in[J]\,.
\end{equation*}
Thus, $t_1=0$ and $t_J=h$. Let $\ell_j$ be the cardinal polynomials,
such that $\ell_j(t_i)=\mathbf 1_{\{i=j\}}$, $i,j\in[J]$. Write
$\mathcal I_{J,h}$ for interpolation at these nodes, and let
\begin{equation*}
\Lambda_J\deq \sup_{t\in[0,h]}\sum_{j\in[J]}\abs{\ell_j(t)}\,, \qquad
 \omega_{i,j}\deq\int_0^{t_i}(t_i-s)\,\ell_j(s)\,\dd s\,,
 \qquad
 \omega_j\deq\int_0^h\ell_j(s)\,\dd s\,,\qquad i,j\in[J]\,.
\end{equation*}
By \cite[Proposition~B.1]{PicardHMC},
\begin{equation}\label{eq:Cheb-bounds}
 \omega_j\geq0\,,\qquad\sum_{j\in[J]}\omega_j=h\,,\qquad
 \Lambda_J\leq C\log J\,.
\end{equation}
We also use the square estimates of \cite[Lemma~B.2]{PicardHMC}:
\begin{equation}\label{eq:integrated-weights-l2}
 \max_{i\in[J]}\sum_{j\in[J]}\omega_{i,j}^2\leq\frac{5\uppi^2h^4}{128\,(J-1)}\,,
 \qquad
 \sum_{j\in[J]}\omega_j^2\leq\frac{\uppi^2h^2}{8\,(J-1)}\,.
\end{equation}
In particular it also implies the bound stated in \cite[Proposition~B.1]{PicardHMC}:
\begin{equation}\label{eq:integrated-weights-l1}
 \max_{i\in[J]}\sum_{j\in[J]}\abs{\omega_{i,j}}
 \leq h^2\,.
\end{equation}

In~\cite{PicardHMC}, we applied Picard iteration to the Hamiltonian flow associated with $\pi_\eta$, using Chebyshev{--}Lobatto quadrature to approximate the integration in time.
In that work, the Picard error was already not dominant at Picard depth $2$.
In the present work, we seek high-accuracy guarantees, which necessitates considering an arbitrary depth $K$.
The required estimates for deep Picard iteration are given in \cite[Theorem~E.2]{PicardHMC}.
In particular, we write $K_{\eta,h}^{[J,K]}$ for the Picard kernel with $J$ Chebyshev--Lobatto nodes and depth $K$, using the exact smoothed score $g_\eta$, and including a Gaussian momentum half-refresh before and after the Hamiltonian dynamics. Its reference phase-space law and associated contraction metric are
\begin{equation*}
 \Pi_\eta\deq\pi_\eta\otimes\cN(0,\Id)\,,\qquad
 \norm{(u,v)}_{M_\kappa}^2
 \deq\norm{v+u/2}^2+\Bigl(\frac1{2\kappa}+\frac14\Bigr)\,\norm u^2\,.
\end{equation*}
We denote by $W_{p,M_\kappa}$ the $p$-Wasserstein distance induced by this norm. It is uniformly equivalent to the usual Euclidean norm.

\begin{proposition}[Picard reference estimates]\label{prop:reference-Picard}
Let $0<\eta\leq1$ and let $J\geq2$, $K\geq1$ be integers. If $h^2\leq1/4$ and $h\leq c/\kappa$, then, for every $p\geq2$,
\begin{align}
 W_{p,M_\kappa}(\delta_zK_{\eta,h}^{[J,K]},\delta_{z'}K_{\eta,h}^{[J,K]})
 &\leq(1-ch/\kappa)\,\norm{z-z'}_{M_\kappa}\,,\notag
\\
 W_{p,M_\kappa}(\Pi_\eta K_{\eta,h}^{[J,K]},\Pi_\eta)
 &\leq C\sqrt{d+p}\,
 \Bigl\{B_{J,p}h\,\Bigl(\frac h{\sqrt\eta}\Bigr)^J
 +h^{2K+1}\Bigr\}\,,\label{eq:reference-defect}
\end{align}
where $B_{J,p}=C_0^{J+1}p^J\sqrt{J!}$ and $C_0$ is universal.
\end{proposition}
\begin{proof}
This is \cite[Theorem~E.2]{PicardHMC}.
\end{proof}

%% file: 5_high_acc_sections/03_cloud_sampler.tex
\section{The Gaussian cloud sampler}\label{sec:fifth-root}

The Picard kernel $K_{\eta,h}^{[J,K]}$ introduced in
\cref{sec:preliminaries} uses the exact smoothed gradient $g_\eta$.
Our task is then to construct a high-accuracy implementation of its Picard updates. Gaussian smoothing gives
the derivative bounds needed for accurate quadrature, and increasing the
Picard depth reduces the iteration error. The remaining difficulty is that
the smoothed gradient is an RGO expectation,
\begin{equation*}
 g_\eta(y)=\E_{X\sim R_{\eta,y}}\nabla V(X)\,.
\end{equation*}
Recall the stochastic estimator introduced in \cref{sec:preliminaries},
$\widehat g_\eta(y;G)\deq
 \nabla V\bigl(\prox_{\eta V}(y)+\sqrt\eta\,G\bigr)$ with $
 \qquad G\sim\cN(0,\Id)$.
It is generally biased for $g_\eta(y)$ and has random fluctuations.
The key algorithmic idea introduced here is to correct the law of the stochastic Picard updates
\emph{without}
explicitly computing the corresponding ideal Picard iterates.

\subsection{Main ideas}

\paragraph{Representing an iterate by a Gaussian cloud.}
The key idea is to associate each ideal Picard iterate with a Gaussian having
that iterate as its mean, and to maintain a procedure for sampling this law.
Fix a phase $n\in\{0,\ldots,N-1\}$, where $N$ is the total number
of phases. Its input is the position--momentum pair $(X_{nh},P_{nh})$ at
smoothing variance $\eta_n$. Position and momentum subscripts denote physical
time: phase $n$ runs from $nh$ to $(n+1)h$. Momentum refreshes
are included in the phase update and do not advance this time index.

In a phase, first draw $\zeta_0\sim\cN(0,\Id)$ and refresh the momentum:
\begin{equation*}
 P_{nh}^-\deq a_hP_{nh}+\sigma_h\zeta_0\,,\qquad
 a_h\deq\e^{-h/2}\,,\qquad \sigma_h\deq\sqrt{1-\e^{-h}}\,.
\end{equation*}
Throughout the description of a cloud, we condition on this phase input and
first refresh. For a total Picard depth $K\ge2$, define the ideal node
positions by
\begin{equation}\label{eq:fifth-Picard}
\begin{aligned}
 X_{nh+t_j}^{[0]}&\deq X_{nh}+t_jP_{nh}^-\,,&&j\in[J]\,,\\
 X_{nh+t_i}^{[\ell+1]}&\deq X_{nh+t_i}^{[0]}
 -\sum_{j\in[J]}\omega_{i,j} g_{\eta_n}(X_{nh+t_j}^{[\ell]})\,,
 &&i\in[J]\,,\quad 0\le\ell\le K-2\,.
\end{aligned}
\end{equation}
Write
$X^{[\ell]}\deq(X_{nh+t_j}^{[\ell]})_{j\in[J]}\in\R^{Jd}$ for the stacked
ideal node vector. These vectors are deterministic under the conditioning above.

Choose an auxiliary cloud variance $0<\theta_n<\eta_n/2$. A \emph{Gaussian cloud} is the law
of a random vector $\cloud^{[\ell]}\in\R^{Jd}$ defined by
\begin{equation*}
 \cloud^{[\ell]}\deq(\cloud_{nh+t_j}^{[\ell]})_{j\in[J]}
 \sim\cN\bigl(X^{[\ell]},\theta_n\Id_{Jd}\bigr)\,,
 \qquad 0\le\ell\le K-1\,.
\end{equation*}
Thus, $\cloud_{nh+t_j}^{[\ell]}\in\R^d$ is the cloud component at time $nh+t_j$;
conditional on $X^[\ell]$, these components are independent with laws
$\cN(X_{nh+t_j}^{[\ell]},\theta_n\Id)$.
The variance $\theta_n$ is distinct from the target smoothing variance
$\eta_n$. The operation we need is to draw fresh samples from
this law, conditional on the phase input, without evaluating its mean
$X^{[\ell]}$.
The initial node vector $X^{[0]}$ is explicit, so its cloud is easy to sample.
The task at each later level is to turn access to the current cloud into
access to the next one.

To see the purpose of the added Gaussian, consider one Picard update. Given
the inaccessible ideal iterate $X^{[\ell]}$, the desired next Picard iterate
has components
\begin{equation*}
 X_{nh+t_i}^{[0]}-\sum_{j\in[J]}\omega_{i,j}
 g_{\eta_n}(X_{nh+t_j}^{[\ell]})\,,
 \qquad i\in[J]\,.
\end{equation*}
Given a cloud sample $\cloud^{[\ell]}\sim
\cN(X^{[\ell]},\theta_n\Id_{Jd})$, draw
$G,G'\sim\cN(0,\Id_{Jd})$ independently and form
\begin{equation*}
 \widetilde\cloud_{nh+t_i}^{[\ell+1]}
 \deq X_{nh+t_i}^{[0]}
 -\sum_{j\in[J]}\omega_{i,j}
 \widehat g_{\eta_n}(\cloud_{nh+t_j}^{[\ell]};G_j)
 +\sqrt{\theta_n}\,G_i'\,,
 \qquad i\in[J]\,.
\end{equation*}
This proposal has an incorrect mean and extra fluctuations, but it is a
small random perturbation of a Gaussian. That
structure allows a likelihood correction to target the entire desired
Gaussian law. Thus an inaccessible deterministic update is replaced by a
transition between cloud laws. Repeating this construction gives the
clouds depicted in \cref{fig:fifth-clouds}.

We next describe the corrections.

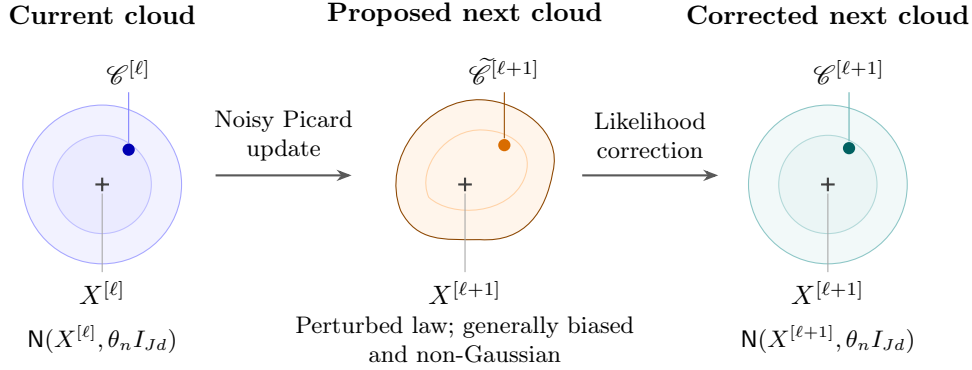
\begin{figure}[!ht]
\centering
\begin{tikzpicture}[font=\small,
  flow/.style={-{Stealth[length=2mm]},draw=black!65,thick},
  hidden/.style={black!75,thick},
  stage/.style={align=center,font=\small\bfseries}]
\node[stage] at (0,2.25) {Current cloud};
\node[stage] at (4.8,2.25) {Proposed next cloud};
\node[stage] at (9.6,2.25) {Corrected next cloud};

\fill[blue!5] (0,0) circle (1.05);
\fill[blue!8] (0,0) circle (.65);
\draw[blue!35] (0,0) circle (1.05);
\draw[blue!25] (0,0) circle (.65);
\draw[hidden] (-.09,0)--(.09,0) (0,-.09)--(0,.09);
\draw[black!35,thin] (0,-.12)--(0,-1.16);
\node[below,inner sep=3pt] at (0,-1.16) {$X^{[\ell]}$};
\fill[blue!70!black] (.35,.46) circle (2.3pt);
\draw[blue!50,thin] (.35,.55)--(.35,1.22);
\node[above,inner sep=3pt] at (.35,1.22) {$\cloud^{[\ell]}$};

\begin{scope}[shift={(4.8,0)}]
\path[fill=orange!8,draw=orange!55!black]
 (-.85,-.35) .. controls (-1.12,.35) and (-.5,1.15) .. (.22,1.12)
 .. controls (.72,1.1) and (1.3,.7) .. (1.16,.16)
 .. controls (1.07,-.35) and (.8,-.77) .. (.2,-.73)
 .. controls (-.3,-.73) and (-.63,-.76) .. cycle;
\path[draw=orange!35]
 (-.48,-.17) .. controls (-.65,.22) and (-.24,.75) .. (.24,.72)
 .. controls (.68,.7) and (.93,.42) .. (.78,.07)
 .. controls (.58,-.4) and (-.17,-.45) .. cycle;
\draw[hidden] (-.09,0)--(.09,0) (0,-.09)--(0,.09);
\draw[black!35,thin] (0,-.12)--(0,-1.16);
\node[below,inner sep=3pt] at (0,-1.16) {$X^{[\ell+1]}$};
\fill[orange!85!black] (.52,.52) circle (2.3pt);
\draw[orange!60!black,thin] (.52,.61)--(.52,1.22);
\node[above,inner sep=3pt] at (.52,1.22) {$\widetilde\cloud^{[\ell+1]}$};
\end{scope}

\begin{scope}[shift={(9.6,0)}]
\fill[teal!5] (0,0) circle (1.05);
\fill[teal!8] (0,0) circle (.65);
\draw[teal!45] (0,0) circle (1.05);
\draw[teal!30] (0,0) circle (.65);
\draw[hidden] (-.09,0)--(.09,0) (0,-.09)--(0,.09);
\draw[black!35,thin] (0,-.12)--(0,-1.16);
\node[below,inner sep=3pt] at (0,-1.16) {$X^{[\ell+1]}$};
\fill[teal!75!black] (.28,.48) circle (2.3pt);
\draw[teal!60,thin] (.28,.57)--(.28,1.22);
\node[above,inner sep=3pt] at (.28,1.22) {$\cloud^{[\ell+1]}$};
\end{scope}

\draw[flow] (1.5,.12)--(3.3,.12);
\node[align=center,font=\footnotesize] at (2.4,.65)
 {Noisy Picard\\update};
\draw[flow] (6.35,.12)--(8.15,.12);
\node[align=center,font=\footnotesize] at (7.25,.65)
 {Likelihood\\correction};

\node[align=center,font=\footnotesize] at (0,-2.05)
 {$\cN(X^{[\ell]},\theta_n\Id_{Jd})$};
\node[align=center,font=\footnotesize] at (4.8,-2.05)
 {Perturbed law; generally biased\\and non-Gaussian};
\node[align=center,font=\footnotesize] at (9.6,-2.05)
 {$\cN(X^{[\ell+1]},\theta_n\Id_{Jd})$};
\end{tikzpicture}
\caption{One internal cloud transition. Each dot is one sampled vector; contours depict laws, and
crosses indicate hidden ideal iterates. The noisy Picard update turns
$\cloud^{[\ell]}$ into a proposal $\widetilde\cloud^{[\ell+1]}$. Likelihood
correction yields a draw $\cloud^{[\ell+1]}$ from the next Gaussian cloud.
Correction uses additional independent cloud calls and may retry proposals;
the arrow denotes a correction of the law, not a deterministic motion of
the dot.}
\label{fig:fifth-clouds}
\end{figure}

\paragraph{Correcting fluctuations and correcting the mean.}
As before, condition on $(X_{nh}, P_{nh}^-)$. The current ideal node positions are
$X_{nh+t_j}^{[\ell]}$, $j\in[J]$. The desired output is a Gaussian with mean
$X^{[\ell+1]}$. The naive proposal $\widetilde\cloud^{[\ell+1]}$ instead has a random mean:
\begin{equation*}
 \widetilde\cloud^{[\ell+1]}=M+\sqrt{\theta_n}\,G'\,,\qquad
 M_i\deq X_{nh+t_i}^{[0]}
 -\sum_{j\in[J]}\omega_{i,j}
       \widehat g_{\eta_n}(\cloud_{nh+t_j}^{[\ell]};G_j)\,,\quad i\in[J]\,.
\end{equation*}
There are two sources of error: $M$ fluctuates from one proposal to the next, and
its average $\E M$ need not equal $X^{[\ell+1]}$.
We correct these errors via first-order rejection sampling (FORS)~\cite{Chen+26HighAccDiffusion}; the
constructions below supply the requisite computable tilt
estimators.
We extract these two correction operations as reusable gadgets in
\cref{sec:reusable-correction}; here we explain how they act on one
cloud transition.

The first correction removes the extra randomness due to $M$, while
retaining the Gaussian noise of variance $\theta_n$ that defines the
cloud; this is \cref{gadget:fluctuation} (hidden fluctuation correction). It uses independent cloud draws $\cloud_{{\rm fluc},r}^{[\ell]}$ and yields the output law
\begin{equation*}
 \cN\bigl(\E M,\theta_n\Id_{Jd}\bigr)\,.
\end{equation*}
This is a correction of the entire distribution: its ideal output is
Gaussian, not just a law with the right covariance. It is implemented
through a likelihood correction using $\widetilde O(1)$ evaluations of $M$, without evaluating the
unknown vector $\E M$.
Note that direct estimation of $\E M$ would require $\poly(1/\delta)$ evaluations of $M$ to achieve $\delta$ error, which would not yield a high-accuracy sampler.

The second correction moves the Gaussian mean from $\E M$ to
$X^{[\ell+1]}$. The required displacement in component $i\in[J]$ is
\begin{equation*}
 X_{nh+t_i}^{[\ell+1]}-\E M_i
 =-\sum_{j\in[J]}\omega_{i,j}\,
 \bigl\{g_{\eta_n}(X_{nh+t_j}^{[\ell]})
       -\E\widehat g_{\eta_n}(\cloud_{{\rm mean},nh+t_j}^{[\ell]};G_j)\bigr\}\,.
\end{equation*}
Here, $\cloud_{\rm mean}^{[\ell]}$ is another fresh independent copy of the current cloud.
Then, \cref{gadget:mean} (hidden mean correction) shows how to perform the mean correction using only an unbiased estimator of the difference in braces. Combining the two corrections gives the desired cloud law
$\cN(X^{[\ell+1]},\theta_n\Id_{Jd})$.

Why can this gradient difference be estimated cheaply when
$X_{nh+t_j}^{[\ell]}$ is unavailable?
In \cref{thm:marginal-cloud}, we develop a Gaussian Stein identity that furnishes the requisite unbiased estimator.
This representation, which uses proximal and Hessian queries, is given for the sake of exposition;
\cref{sec:gradient-marginal} provides a gradient-only implementation.

That identity also requires an independent RGO draw from
$R_{\eta_n}(X_{nh+t_j}^{[\ell]},\cdot)$. Since $X_{nh+t_j}^{[\ell]}$ is
unavailable, we cannot call an RGO at that iterate directly. The realization
of \cref{gadget:hidden-RGO} (hidden RGO sampler), analyzed in \cref{lem:hidden-RGO},
uses another independent cloud draw $\cloud_{\rm RGO}^{[\ell]}$ to sample from
this RGO without knowing the hidden center $X_{nh+t_j}^{[\ell]}$.

\paragraph{Do we need exponentially many cloud samples?}
Starting with the cloud sample $\cloud^{[\ell]}$ at level $\ell$, the correction mechanisms require additional cloud samples $\cloud_{{\rm fluc},r}^{[\ell]}$, $\cloud_{\rm mean}^{[\ell]}$, and $\cloud_{\rm RGO}^{[\ell]}$.
Na\"{\i}vely, generating a fixed number of independent predecessor clouds at every
level would produce an exponentially large recursion tree. This issue is solved via careful accounting.
Each level requires just one primary predecessor sample $\cloud^{[\ell]}$; auxiliary
samples are generated only when a tilt estimator or a rejection
requests them. For example, a second cloud copy used to correct
the mean in \cref{gadget:mean} is drawn only when the corresponding Poisson count is non-zero.
For a clipping parameter $\clipB$, this event and rejection each have
probability $O(\clipB)$. Our eventual parameter choices make these probabilities small
enough to offset both the number of levels and the number of auxiliary
calls that one tilt estimator can request. We show that this leads to
polylogarithmic overhead in the expected cloud count, rather than
an exponential dependence on depth; see
\cref{lem:adaptive-tree,sec:work}.

\paragraph{From clouds to a phase-space update.}
The position and momentum outputs of the Picard iteration also depend on the same gradient evaluations.
We therefore use the cloud at the last level and correct their joint law.
Write the uncorrected endpoint proposals as
$(\widetilde X_{(n+1)h}^{[K]},\widetilde P_{(n+1)h}^{[K]})$.

To apply the Gaussian correction gadgets, we need independent Gaussian
noise with non-zero variance in every direction in phase space. The final OU half-refresh supplies this noise in momentum, but
not for the position component. We
therefore introduce a variance $\tau>0$ and add an independent
$\cN(0,\tau\Id)$ draw to the proposed position. Denote the resulting
noisy position and refreshed momentum by
$(X_{(n+1)h}^{\circ},P_{(n+1)h}^{\circ})$ (see 
\eqref{eq:fifth-noisy-endpoint}). The circles distinguish these proposals from the accepted state
$(X_{(n+1)h},P_{(n+1)h})$. Their independent Gaussian noises have
variances $\tau$ and $\sigma_h^2$, respectively. Rescaling by their
standard deviations gives a full-rank standard Gaussian perturbation,
so the hidden fluctuation and mean corrections, \cref{gadget:fluctuation,gadget:mean}, apply to the entire pair. The
accepted pair becomes the next phase-space state, including the added
position noise.

\paragraph{Retaining the smoothing until the terminal step.}
Adding
an independent Gaussian increases the smoothing variance, while a
RGO draw with the matching variance removes the smoothing entirely.
For exact input laws, these identities read
\begin{equation*}
 \pi_\eta*\cN(0,\tau\Id)=\pi_{\eta+\tau}\,,
 \qquad
 \pi_\eta R_\eta=\pi\,.
\end{equation*}

RGO calls are expensive, so we retain the position noise added
at the endpoint and defer its removal. The reference smoothing level
then advances as $\eta_{n+1}=\eta_n+\tau$, without an RGO call
after each phase.
We shall choose the constants so that $\eta_n \in [\eta_0,5\eta_0/4]$ for the entire
run; the smoothing is removed only at the terminal step.

Finally, the chain is close to a smoothed target in Wasserstein distance,
whereas the requested output accuracy is in total variation. At the final
smoothing variance $\eta_{\rm f}$, one more Gaussian convolution converts this Wasserstein
control into KL control. Specifically, joint convexity of the KL divergence gives,
for probability laws $P,Q$ with finite second moments,
\begin{equation}\label{eq:W2-KL}
 \KL\bigl(P*\cN(0,\tau\Id)\mathbin\Vert Q*\cN(0,\tau\Id)\bigr)
 \leq\frac{W_2^2(P,Q)}{2\tau}\,.
\end{equation}
Applying this estimate with the final smoothing level $\tau=\eta_{\rm f}$ produces the reference location law
$\pi_{2\eta_{\rm f}}$, which explains the variance $2\eta_{\rm f}$ in the terminal
RGO draw. RGO disintegration then returns the original target
$\pi$.

\Cref{fig:fifth-sampler} summarizes the phase and the terminal RGO draw
in \cref{alg:fifth-phase}. In the terminal box, $Y_{\rm f}$ is the Gaussian-smoothed final
position and $X_{\rm out}$ the returned sample. The vector $G_{\rm f}$
is a fresh independent standard Gaussian.

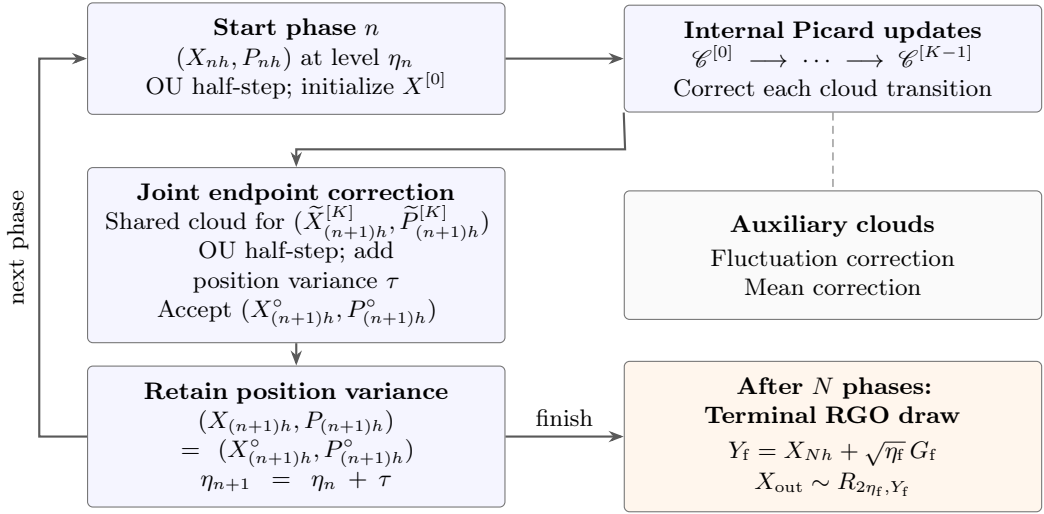
\begin{figure}[!ht]
\centering
\begin{tikzpicture}[
  font=\footnotesize,
  box/.style={draw=black!55,rounded corners=2pt,align=center,
    inner sep=6pt,text width=5.1cm,minimum height=1.25cm},
  phase/.style={box,fill=blue!4},
  RGO/.style={box,fill=orange!7},
  flow/.style={-{Stealth[length=2mm]},thick,draw=black!65}]
\node[phase] (start) at (0,0)
  {\textbf{Start phase $n$}\\
   $(X_{nh},P_{nh})$ at level $\eta_n$\\
   OU half-step; initialize $X^{[0]}$};
\node[phase] (cloud) at (7.1,0)
  {\textbf{Internal Picard updates}\\
   $\cloud^{[0]}\longrightarrow\cdots\longrightarrow \cloud^{[K-1]}$\\
   Correct each cloud transition};
\node[box,fill=black!2,text width=5.1cm,minimum height=1.7cm]
  (detail) at (7.1,-2.6)
  {\textbf{Auxiliary clouds}\\[2pt]
   Fluctuation correction\\
   Mean correction};
\node[phase,minimum height=1.7cm] (endpoint) at (0,-2.6)
  {\textbf{Joint endpoint correction}\\
   Shared cloud for $(\widetilde X_{(n+1)h}^{[K]},\widetilde P_{(n+1)h}^{[K]})$\\
   OU half-step; add position variance $\tau$\\
   Accept $(X_{(n+1)h}^{\circ},P_{(n+1)h}^{\circ})$};
\node[phase] (retain) at (0,-5.0)
  {\textbf{Retain position variance}\\
   $(X_{(n+1)h},P_{(n+1)h})$\\
   $=(X_{(n+1)h}^{\circ},P_{(n+1)h}^{\circ})$\\
   $\eta_{n+1}=\eta_n+\tau$};
\node[RGO,minimum height=1.25cm]
  (terminal) at (7.1,-5.0)
  {\textbf{After $N$ phases:}\\
   \textbf{Terminal RGO draw}\\[2pt]
   \mbox{$Y_{\rm f}=X_{Nh}+\sqrt{\eta_{\rm f}}\,G_{\rm f}$}\\[1pt]
   \mbox{$X_{\rm out}\sim R_{2\eta_{\rm f},Y_{\rm f}}$}};
\draw[flow] (start) -- (cloud);
\draw[densely dashed,draw=black!55] (cloud) -- (detail);
\draw[flow] (cloud.south west) -- (4.35,-1.2)
  -- (0,-1.2) -- (endpoint.north);
\draw[flow] (endpoint) -- (retain);
\draw[flow] (retain) -- node[above,font=\footnotesize] {finish} (terminal);
\draw[flow] (retain.west) -- (-3.4,-5.0)
  -- node[above,sloped,font=\footnotesize] {next phase} (-3.4,0)
  -- (start.west);
\end{tikzpicture}
\caption{The Gaussian cloud sampler. Blue boxes describe one phase; the orange
box is the terminal RGO draw. The Gaussian means are the ideal Picard iterates: the implementation
samples corrected clouds without evaluating those hidden iterates or their exact
smoothed gradients.}
\label{fig:fifth-sampler}
\end{figure}

\subsection{Algorithm description}

\FloatBarrier

The following algorithm assembles these operations. Its displayed exact
cloud and RGO laws specify the targets of the correction routines;
their implementations are given in
\cref{thm:robust-fluctuation,prop:marginal-score,sec:recursive}.

\begin{algorithm}[Gaussian cloud sampler]\label{alg:fifth-phase}
For prescribed sampler parameters, including the initial smoothing variance
$\eta_0$, initialize $X_0\deq x_{\rm ref}$ and
$P_0\sim\cN(0,\Id)$. For $n=0,\ldots,N-1$, given $(X_{nh},P_{nh})$ at smoothing level $\eta_n=\eta_0+n\tau$, do the
following.
\begin{enumerate}[label=\textup{(\roman*)},leftmargin=2.4em]
\item \emph{Refresh momentum and initialize the node positions.}
Put $a_h\deq \e^{-h/2}$ and $\sigma_h\deq \sqrt{1-\e^{-h}}$.  Draw
$\zeta_0\sim\cN(0,\Id)$, set $P_{nh}^-\deq a_hP_{nh}+\sigma_h\zeta_0$, and put
\[
 X_{nh+t_j}^{[0]}=X_{nh}+t_jP_{nh}^-\,,\qquad j\in[J]\,.
\]
\item \emph{Propagate the clouds through the internal Picard updates.}
For $\ell=0,\ldots,K-2$, use samples from the level-$\ell$ cloud
$\cN(X^{[\ell]},\theta_n \Id_{Jd})$ to target the next cloud
$\cN(X^{[\ell+1]},\theta_n \Id_{Jd})$, with the ideal iterates defined by
\eqref{eq:fifth-Picard}. The implementation samples these laws without
evaluating the iterates. The proposal uses independent
Gaussian blocks
\[
 \cloud^{[\ell]}\sim\cN(X^{[\ell]},\theta_n \Id_{Jd})\,,
 \qquad
 G\sim\cN(0,\Id_{Jd})\,,
 \qquad
 G'\sim\cN(0,\Id_{Jd})\,,
\]
whose $d$-dimensional components are denoted by $\cloud_{nh+t_j}^{[\ell]}$,
$G_j$, and $G_j'$, $j\in[J]$. For each $i\in[J]$, propose
\[
 \widetilde\cloud_{nh+t_i}^{[\ell+1]}
 \deq X_{nh+t_i}^{[0]}-\sum_{j\in[J]} \omega_{i,j}
   \wh g_{\eta_n}(\cloud_{nh+t_j}^{[\ell]};G_j)
   +\sqrt{\theta_n}\,G_i'\,.
\]
The fluctuation correction treats $(\cloud^{[\ell]},G)$ as one Gaussian input.
If the tilt estimators in one fluctuation FORS attempt request $k>0$ additional cloud
inputs in total, generate lazily
\[
 \cloud_{{\rm fluc},r}^{[\ell]}
 \sim\cN(X^{[\ell]},\theta_n\Id_{Jd})\,,
 \qquad r\in[k]\,,
\]
independently of each other and of the observed $\cloud^{[\ell]}$, conditional on
$X^{[\ell]}$. Use these samples to jointly generate the derivative inputs for the entire attempt, as described in \cref{sec:joint-derivatives}. For every component $j\in[J]$, the independent mean
estimator uses a fresh mean cloud
$\cloud_{{\rm mean},nh+t_j}^{[\ell]}\sim
\cN(X_{nh+t_j}^{[\ell]},\theta_n\Id)$, and its RGO sample is generated
from another cloud
$\cloud_{{\rm RGO},nh+t_j}^{[\ell]}\sim
\cN(X_{nh+t_j}^{[\ell]},\theta_n\Id)$. Conditional on $X^{[\ell]}$, the
current, fluctuation, mean, and RGO cloud draws are mutually independent.
\item \emph{Correct the joint endpoint and retain its position noise.}
Draw fresh independent Gaussian labels $G_1,\ldots,G_J\sim\cN(0,\Id)$ and, using the same
level-$(K-1)$ cloud in both components, form the endpoint
\begin{align}
 \widetilde X_{(n+1)h}^{[K]}
 &\deq X_{nh}+hP_{nh}^-
 -\sum_{j\in[J]} \omega_{J,j}
   \wh g_{\eta_n}(\cloud_{nh+t_j}^{[K-1]};G_j)\,,\label{eq:fifth-endpoint-Z}\\
 \widetilde P_{(n+1)h}^{[K]}
 &\deq P_{nh}^-
 -\sum_{j\in[J]}\omega_j
   \wh g_{\eta_n}(\cloud_{nh+t_j}^{[K-1]};G_j)\,.\notag
\end{align}
Draw independent $\zeta_1,G_x\sim\cN(0,\Id)$ and put
\begin{equation}\label{eq:fifth-noisy-endpoint}
\begin{aligned}
 X_{(n+1)h}^{\circ}&\deq \widetilde X_{(n+1)h}^{[K]}+\sqrt{\tau}\,G_x\,,
 \qquad P_{(n+1)h}^{\circ}\deq a_h\widetilde P_{(n+1)h}^{[K]}+\sigma_h\zeta_1\,.
\end{aligned}
\end{equation}
At the last internal level, apply \cref{cor:block-correction} jointly to the full-rank
whitened pair
\begin{equation}\label{eq:fifth-whitened-endpoint}
 \begin{bmatrix}
     (X_{(n+1)h}^{\circ}-(X_{nh} + hP_{nh}^-))/\sqrt{\tau}\\[1mm]
     (P_{(n+1)h}^{\circ}-a_h P_{nh}^-)/\sigma_h
 \end{bmatrix}
\end{equation}
and target \eqref{eq:fifth-endpoint-Z}--\eqref{eq:fifth-noisy-endpoint} with every occurrence of
$\widehat g_{\eta_n}$ replaced by $g_{\eta_n}(X_{nh+t_j}^{[K-1]})$. Only after this joint proposal is accepted, set
\begin{align*}
 (X_{(n+1)h},P_{(n+1)h})&\deq(X_{(n+1)h}^{\circ},P_{(n+1)h}^{\circ})\,,
 \qquad \eta_{n+1}\deq \eta_n+\tau\,.
\end{align*}
Thus the position Gaussian is retained and there is no RGO update at this step.
\end{enumerate}
\emph{Convert to the original target.}
After $N$ phases, write $\eta_{\rm f}\deq\eta_N=\eta_0+N\tau$ for the final smoothing variance.
Draw a fresh independent $G_{\rm f}\sim\cN(0,\Id)$ and then set
\[
 Y_{\rm f}\deq X_{Nh}+\sqrt{\eta_{\rm f}}\,G_{\rm f}\,,
 \qquad
 X_{\rm out}\sim R_{2\eta_{\rm f},Y_{\rm f}}\,.
\]
\end{algorithm}

The proof of \cref{thm:core} specifies the sampler parameters and the
accuracy budgets of its implemented RGO calls.

\paragraph{Proximal sampler wrapper.}
The construction just described is the core sampler, which we apply to targets with
bounded condition number. To obtain the main
bound, we compose it with the proximal sampler, which handles arbitrary
$\kappa$ by calling the Gaussian cloud sampler only on targets of condition number at most three.

\FloatBarrier

\subsection{Algorithmic components}\label{sec:reusable-correction}
We isolate out reusable algorithmic components in the form of gadgets.
First, we record the first-order rejection sampling (FORS) routine, which is the basic correction mechanism.

\par\medskip
\begingroup
\setlength{\fboxsep}{7pt}
\noindent\fbox{\begin{minipage}{\dimexpr\linewidth-2\fboxsep-2\fboxrule\relax}
        \begin{gadget}[First-order rejection sampling (FORS)]\label{gadget:FORS}
\leavevmode\par\noindent
\textbf{Access.} A sampler for a proposal law $Q$ and, for each
proposal $\omega$, a sampler for an estimator $W_\omega\in[-\clipB,\clipB]$,
where $\clipB>0$ is known.
\par\smallskip
\textbf{Output.} An exact draw from the probability law
\begin{equation*}
 Q_w(\dd\omega)\deq
 \frac{\exp\{w(\omega)\}\,Q(\dd\omega)}
 {\E_{\omega'\sim Q}\exp\{w(\omega')\}}\,,\qquad
 w(\omega)\deq\E[W_\omega\mid\omega]\,.
\end{equation*}
The algorithm cannot evaluate $w$ and its normalizing constant.
\end{gadget}
\end{minipage}}\par
\endgroup
\medskip

\begin{algorithm}[FORS]\label{alg:reusable-FORS}
To realize \cref{gadget:FORS}, repeat the following steps until acceptance.
\begin{enumerate}[label=\textup{(\roman*)},leftmargin=2.4em]
\item Draw $\omega\sim Q$ and independently
$M\sim\mathsf{Poisson}(2\clipB)$.
\item Conditional on $\omega$, draw independent
$W_\omega^{(k)}$, $k\in\{1,\ldots,M\}$.
\item Accept and return $\omega$ with probability $\prod_{k=1}^M\frac{\clipB+W_\omega^{(k)}}{2\clipB}$.
\end{enumerate}
\end{algorithm}

\begin{lemma}[FORS call counts]\label{lem:reusable-FORS}
\Cref{alg:reusable-FORS} realizes \cref{gadget:FORS} and terminates
almost surely. Its expected number of calls to the proposal and estimator samplers
 is at most $\e^{2\clipB}$ and
$2\clipB\e^{2\clipB}$, respectively. In particular, when $\clipB\le1$, these
bounds are $O(1)$ and $O(\clipB)$.
\end{lemma}
\begin{proof}
See~\cite{Chen+26HighAccDiffusion}.
\end{proof}

\par\medskip
\begingroup
\setlength{\fboxsep}{7pt}
\noindent\fbox{\begin{minipage}{\dimexpr\linewidth-2\fboxsep-2\fboxrule\relax}
\begin{gadget}[Hidden fluctuation correction]\label{gadget:fluctuation}
\leavevmode\par\noindent
\textbf{Access.} A globally $L$-Lipschitz map $\Phi:\R^m\to\R^n$,
with $m,n\ge1$, through evaluations $\Phi(g)$ and Jacobian-vector products
$D\Phi(g)v$, $D\Phi(g)^\T w$, for $g,v\in\R^m$ and $w\in\R^n$.
The Jacobian-vector products at points of non-differentiability may be defined arbitrarily.
\par\smallskip
\textbf{Output.} A draw with law $\widehat Q$ satisfying
\begin{equation*}
 \TV\bigl(\widehat Q,\cN(\E\Phi(G),\Id_n)\bigr)\le\delta\,,
 \qquad G\sim\cN(0,\Id_m)\,.
\end{equation*}
The expectation $\E\Phi(G)$ is not supplied.
\end{gadget}
\end{minipage}}\par
\endgroup
\medskip

We remark that simply estimating $\E \Phi(G)$ itself with accuracy $\delta$ would require $\poly(1/\delta)$ evaluations, which would not be compatible with our high-accuracy goal. We refer to this gadget and the following ones as ``hidden'' because the mean (or center, for the hidden RGO below) is not available.

The raw proposal $\Phi(G)+Z$, for independent standard Gaussians
$G\in\R^m$ and $Z\in\R^n$, has the desired mean but its law is not Gaussian, and it has excess fluctuations. We first define the covariance that the correction uses.
Let $(B_t)_{0\le t\le1}$ be standard Brownian motion in $\R^m$, with
$B_0=0$ and $B_1=G$, and let $\mathcal F_t$ denote the $\sigma$-algebra generated by 
$(B_s)_{0\le s\le t}$. The Clark--Ocone formula
\cite[Proposition~1.3.14]{Nua06Malliavin} gives
\begin{equation*}
    \Phi(G) - \E \Phi(G) = \int_0^1 A_t\,\dd B_t\,, \qquad
 A_t\deq\E[D\Phi(B_1)\mid\mathcal F_t]
 =\E_{G'} D\Phi(B_t+\sqrt{1-t}\,G')\,,
\end{equation*}
where $G' \sim \cN(0, \Id_m)$ is independent.
Define the random covariance matrix
\begin{equation*}
 \Sigma\deq\int_0^1A_tA_t^\T\,\dd t\,.
\end{equation*}
The following algorithm
constructs unbiased estimators of powers of $\Sigma$.

\begin{algorithm}[Hidden fluctuation correction]
\label{alg:reusable-fluctuation}
To realize \cref{gadget:fluctuation}, choose an integer $K\ge1$
and $\clipB_{\rm fluc}\in(0,1]$, and invoke
\cref{gadget:FORS} with bound $\clipB_{\rm fluc}$ and the following
samplers.
\begin{enumerate}[label=\textup{(\roman*)},leftmargin=2.4em]
\item \emph{Proposal.} Draw independent standard Gaussians
$G\in\R^m$ and $Z\in\R^n$, and set $Y\deq \Phi(G)+Z$.
Include in the proposal a Brownian path ${(B_t)}_{0\le t \le 1}$ conditioned on $B_1=G$, sampling its values as needed.
Specifically, between adjacent revealed times
$t_-<t<t_+$, draw
\begin{equation*}
 B_t\sim\cN\Bigl(
 \frac{t_+-t}{t_+-t_-}B_{t_-}
 +\frac{t-t_-}{t_+-t_-}B_{t_+},\,
 \frac{(t-t_-)\,(t_+-t)}{t_+-t_-}\Id_m\Bigr)\,.
\end{equation*}
Initially the revealed times are $0$ and $1$.
\item \emph{Tilt estimator.} Draw a fresh independent
$Z'\sim\cN(0,\Id_n)$. Generate
$K$ random matrices $\widehat\Sigma_r$, $r\in\{1,\ldots,K\}$,
independently conditional on the Brownian path. Each one is constructed as follows. Draw
$T\sim\mathsf{Unif}[0,1]$ and independent
$G_1,G_2\sim\cN(0,\Id_m)$, reveal $B_T$,
and set
\begin{equation*}
 \widehat\Sigma_r\deq
 D\Phi(B_T+\sqrt{1-T}\,G_1)\,
 D\Phi(B_T+\sqrt{1-T}\,G_2)^\T \,.
\end{equation*}
The variables $G_1$, $G_2$, $T$ are freshly drawn each time. Set
$\widehat{\Sigma^0}\deq\Id_n$ and
$\widehat{\Sigma^k}\deq\widehat\Sigma_k\widehat{\Sigma^{k-1}}$
for $k\in\{1,\ldots,K\}$; thus $\widehat{\Sigma^k}$ is an unbiased estimator of
$\Sigma^k$ conditional on the full Brownian path $B=(B_t)_{0\le t\le1}$.
Compute their actions on $Z,Z'$ successively,
reusing the two vectors from the previous level. Return
$\operatorname{clip}_{\clipB_{\rm fluc}}(W) \deq (-\clipB_{\rm fluc}) \vee W \wedge \clipB_{\rm fluc}$, where
\begin{equation*}
 W\deq\frac12\sum_{k=1}^K
 \Bigl\{\frac1k\,\langle Z',\widehat{\Sigma^k}Z'\rangle
       -\langle Z,\widehat{\Sigma^k}Z\rangle\Bigr\}\,.
\end{equation*}
\end{enumerate}
Return the accepted FORS proposal $Y$.
\end{algorithm}

The calculation has two steps. Conditional on the Brownian path,
independence gives for $k \geq 1$
\begin{equation*}
 \E[\widehat\Sigma_r\mid B]=\Sigma\,,\qquad
 \E[\widehat{\Sigma^k}\mid B]=\Sigma^k\,.
\end{equation*}
Averaging also over $Z'$ therefore gives
\begin{equation*}
 \E[W\mid B,Z]
 =\frac12\sum_{k=1}^K
 \Bigl\{\frac1k\tr(\Sigma^k)
        -Z^\T \Sigma^kZ\Bigr\}\,.
\end{equation*}
As $K\to\infty$, this is the log-density ratio
of $\cN(0,\Id_n-\Sigma)$ to $\cN(0,\Id_n)$ at $Z$, and FORS produces this tilt up to the errors incurred by clipping and truncation.
To see why the output is Gaussian, note that under the ideal tilt, for $u\in\R^n$, $\ii \deq \sqrt{-1}$, and for the FORS output $Y$,
\begin{align*}
    \E\exp(\ii\,\langle Y -\E\Phi(G), u\rangle)
    &= \E_B \E_{Z \sim \cN(0, \Id_n-\Sigma)}\exp(\ii\,\langle \Phi(G) - \E \Phi(G) + Z, u \rangle) \\[0.25em]
    &= \E_B\exp\Bigl(\ii\,\Bigl\langle \int_0^1 A_t\,\dd B_t, u \Bigr\rangle- \frac{1}{2}\,\langle u, (I_n-\Sigma)\,u\rangle\Bigr) \\[0.25em]
    &= \exp\bigl( - \frac{1}{2}\,\norm u^2\bigr)\,,
\end{align*}
where the last line uses the exponential martingale identity, see \cite[Proposition~5.11]{LeG16StocCalc}.
Thus, the ideal
output is indeed $\cN(\E \Phi(G),\Id_n)$.

The next lemma controls the clipping and truncation errors.
In \cref{sec:fixed-node}, we prove a significantly strengthened version
of this fluctuation gadget, allowing for approximate value and derivative evaluations. Then, in
\cref{sec:gradient-only}, we give a construction which removes the need for Jacobian-vector products altogether.

\begin{lemma}[Hidden fluctuation correction]
\label{thm:reusable-correction}
There are universal constants $c,C>0$ with the following property.
Let $\Phi:\R^m\to\R^n$ be a globally $L$-Lipschitz map, where
$m,n\ge1$ and $0<L^2\le1/8$. Suppose that the following operations are
available:
\begin{enumerate}[label=\textup{(\alph*)},leftmargin=2.4em]
\item Evaluate $\Phi(g)$ at any supplied $g\in\R^m$.
\item At any differentiability point $g$, evaluate either
$D\Phi(g)v$, $v\in\R^m$, or
$D\Phi(g)^\T w$, $w\in\R^n$.
\end{enumerate}
Fix $0<\delta\le1/2$, and define $\mathfrak L \deq \log(1/\delta)$.
Then, for $\clipB_{\rm fluc} \in (0,1]$, if
\begin{equation*}
    L^2\,\bigl(\sqrt{n\mathfrak L}+\mathfrak L\bigr)+L^4 n \le c\clipB_{\rm fluc}\,,
\end{equation*}
then \cref{alg:reusable-fluctuation}
with $K=\lceil C\mathfrak L\rceil$ realizes \cref{gadget:fluctuation}.
It uses at most $C$ evaluations of $\Phi$ and
$C \clipB_{\rm fluc}\mathfrak L$ Jacobian-vector products in expectation, without evaluating $\E \Phi(G)$.
\end{lemma}

\par\medskip
\begingroup
\setlength{\fboxsep}{7pt}
\noindent\fbox{\begin{minipage}{\dimexpr\linewidth-2\fboxsep-2\fboxrule\relax}
\begin{gadget}[Hidden mean correction]\label{gadget:mean}
\leavevmode\par\noindent
\textbf{Access.} Independent draws from a law $\widehat Q$ with
$\TV(\widehat Q,\cN(\mu,\Id_m))\le\varepsilon$, where $\mu\in\R^m$
is unknown, and an independent sampler for a random vector
$\Delta\in\R^m$.
\par\smallskip
\textbf{Output.} A draw with law $\widehat P$ satisfying
\begin{equation*}
 \TV\bigl(\widehat P,\cN(\mu+\E\Delta,\Id_m)\bigr)
 \le\delta+C\varepsilon\,,
\end{equation*}
where $C$ is a universal constant. Neither $\mu$ nor $\E\Delta$ is supplied.
\end{gadget}
\end{minipage}}\par
\endgroup
\medskip

Again, we want to avoid estimation of $\E \Delta$ as this requires $\poly(1/\delta)$ draws.

\begin{algorithm}[Hidden mean correction]
\label{alg:reusable-mean}
To realize \cref{gadget:mean}, choose $\clipB_{\rm mean}\in(0,1]$ and
invoke \cref{gadget:FORS} with bound $\clipB_{\rm mean}$ and the following
samplers.
\begin{enumerate}[label=\textup{(\roman*)},leftmargin=2.4em]
\item \emph{Proposal.} Draw $Y \sim \widehat Q$.
\item \emph{Tilt estimator.} Independently draw
$Y'\sim\widehat Q$ and a copy of $\Delta$, and return $\operatorname{clip}_{\clipB_{\rm mean}}
 \langle \Delta,Y-Y'\rangle$.
\end{enumerate}
Return the accepted FORS proposal $Y$.
\end{algorithm}

To get some intuition, suppose that $\widehat Q = \cN(\mu, I_m)$, i.e., $\varepsilon = 0$.
Since $\E[\langle \Delta,Y-Y'\rangle\mid Y] =\langle \E\Delta,Y-\mu\rangle$,
completing squares in the Gaussian density shows that tilting the law of
$Y$ by $\exp{\langle \E\Delta,Y-\mu\rangle}$ gives
\begin{equation*}
 Y\sim\cN(\mu+\E\Delta,\Id_m)\,.
\end{equation*}
The next lemma
accounts for clipping and the error in $\widehat Q$.

\begin{lemma}[Hidden mean correction]
\label{prop:reusable-mean}
There are universal constants $c,C>0$ with the following property. Let
$\mu\in\R^m$ be unknown. Suppose that a sampler returns independent draws
from a law $\widehat Q$ satisfying $\TV(\widehat Q,\cN(\mu,\Id_m))\le\varepsilon$,
and suppose that an independent sampler returns a random vector $\Delta\in\R^m$.
Fix $0<\delta\le1/2$, put $p\deq\lceil C\log(2/\delta)\rceil$, and
assume that $\norm{\Delta}_{L^p} \le b$.
Provided that
$\clipB_{\rm mean}\in(0,1]$ satisfies
\begin{equation*}
 \sqrt p\,b+b^2\le c\clipB_{\rm mean}\,,
\end{equation*}
\Cref{alg:reusable-mean} realizes \cref{gadget:mean}.
It uses at most $1+C\clipB_{\rm mean}$ calls to $\widehat Q$ and $C \clipB_{\rm mean}$ calls to the sampler for $\Delta$ in expectation.
\end{lemma}

\par\medskip
\begingroup
\setlength{\fboxsep}{7pt}
\noindent\fbox{\begin{minipage}{\dimexpr\linewidth-2\fboxsep-2\fboxrule\relax}
\begin{gadget}[Hidden RGO sampler]\label{gadget:hidden-RGO}
\leavevmode\par\noindent
\textbf{Access.} Independent draws from $\cN(y,\theta\Id_d)$,
where $y\in\R^d$ is unknown; a supplied anchor $z\in\R^d$;
gradient evaluations of a convex potential $V:\R^d\to\R$ with
$1$-Lipschitz gradient; and exact draws, at any explicit center $u$,
from the RGO $R_{\eta-2\theta,u}$.

The variances $\eta,\theta$ with $0<2\theta<\eta$ and an accuracy $0<\delta\le1/2$ are supplied.
\par\smallskip
\textbf{Output.} A draw with law $\widehat R$ satisfying
$\TV(\widehat R,R_{\eta,y})\le\delta$, without access to the center
$y$ itself.
\end{gadget}
\end{minipage}}\par
\endgroup
\medskip

\begin{algorithm}[Hidden RGO sampler]\label{alg:hidden-RGO}
To realize \cref{gadget:hidden-RGO}, take its inputs $\eta,\theta,z$
and choose $\clipB_{\rm RGO}\in(0,1]$.
Draw $X_0\sim R_{\eta-2\theta,z}$, set $g\deq\nabla V(X_0)$, and keep $g$
fixed throughout the following invocation of \cref{gadget:FORS},
with bound $\clipB_{\rm RGO}$.
\begin{enumerate}[label=\textup{(\roman*)},leftmargin=2.4em]
\item \emph{Proposal.} Draw a fresh $Y\sim\cN(y,\theta\Id)$ and
independent $G\sim\cN(0,\Id)$, and return the center
$U\deq Y-2\theta g+\sqrt\theta\,G$.
\item \emph{Tilt estimator.} Independently draw
$Y'\sim\cN(y,\theta\Id)$, $G'\sim\cN(0,\Id)$, and
$S\sim\mathsf{Unif}[-1,1]$. Put
\begin{equation*}
 U'\deq Y'-2\theta g+\sqrt\theta\,G'\,,\qquad
 M_U\deq\frac{U+U'}2\,,\qquad
 \Delta_U\deq\frac{U'-U}2\,.
\end{equation*}
Draw conditionally independent
$X_\pm\sim R_{\eta-2\theta,M_U\pm S\Delta_U}$, and
return $\operatorname{clip}_{\clipB_{\rm RGO}}(W)$, where
\begin{equation*}
 W\deq\langle\nabla V(X_+)+\nabla V(X_-)-2g,\Delta_U\rangle\,.
\end{equation*}
\end{enumerate}
At the center $U$ accepted by FORS, draw and return
$X\sim R_{\eta-2\theta,U}$.
\end{algorithm}

Let us explain why the center correction gives the desired RGO. Write
$W_g\deq V_{\eta-2\theta}-\langle g,\cdot\rangle$. The identity
$\nabla V_{\eta-2\theta}(u)
=\E_{X\sim R_{\eta-2\theta,u}}\nabla V(X)$ and integration along
the segment from $U$ to $U'$ give
\begin{equation}\label{eq:hidden-RGO-centered-likelihood}
 \E[W\mid U,U',g]
 =\int_{-1}^1
 \langle\nabla W_g(M_U+s\Delta_U),\Delta_U\rangle\,\dd s
 =W_g(U')-W_g(U)\,.
\end{equation}
The reference $U'$ is independent of $U$ given $g$, so we can average out $U'$ as well. The proposal center law
is $\cN(y-2\theta g,2\theta\Id)$. Tilting it by $\exp\{-W_g(U)\}$
cancels the linear term involving $g$, giving the center density
\begin{equation}\label{eq:hidden-RGO-center-law}
 q_y(\dd u)\propto\exp\Bigl\{-V_{\eta-2\theta}(u)
 -\frac{\norm{u-y}^2}{4\theta}\Bigr\}\,\dd u\,.
\end{equation}
Combining this density with the final RGO $R_{\eta-2\theta,u}$ cancels the
normalizer $\exp\{-V_{\eta-2\theta}\}$ up to a constant.
Integration yields
\begin{equation*}
 \int_{\R^d}\exp\Bigl\{-V(x)-\frac{\norm{x-u}^2}{2\,(\eta-2\theta)}
                 -\frac{\norm{u-y}^2}{4\theta}\Bigr\}\,\dd u
 \propto\exp\Bigl\{-V(x)-\frac{\norm{x-y}^2}{2\eta}\Bigr\}\,.
\end{equation*}
This is $R_{\eta,y}(x)$. The next lemma bounds the effect of clipping and
approximate RGO calls.

\begin{lemma}[Hidden RGO sampler]\label{lem:hidden-RGO}
Let $V$ be convex with $1$-Lipschitz gradient, and let
$0<\eta-2\theta\le1$ and $0<\theta\le(\eta-2\theta)/4$.
Fix $y,z\in\R^d$, with $y$ unknown and $z$ supplied. Assume access to
independent cloud draws and exact RGO draws as in \cref{gadget:hidden-RGO}.
Suppose that supplied bounds $\varepsilon_0,L\ge0$ satisfy
\begin{equation*}
 \norm{y-z}\le\varepsilon_0\,,\qquad
 \norm{g_{\eta-2\theta}(y)}+\norm{g_{\eta-2\theta}(z)}
 \le L\,.
\end{equation*}
For $0<\delta\le1/2$, set $p\deq \lceil C\log(2/\delta)\rceil$.
There are universal constants $c,C>0$ such that, if
\begin{equation*}
 \sqrt{\theta p}\,(\varepsilon_0+\theta L)
 +\sqrt{\eta\theta}\,(\sqrt{dp}+p) \le c\clipB_{\rm RGO}
\end{equation*}
for $\clipB_{\rm RGO} \in (0,1]$, then
\cref{alg:hidden-RGO} realizes \cref{gadget:hidden-RGO}.
Its expected numbers of cloud sampler calls, RGO
calls, and gradient evaluations are at most $C$, $2+C\clipB_{\rm RGO}$,
and $1+C\clipB_{\rm RGO}$, respectively.

If the RGO calls instead have
uniform TV error at most $\varepsilon$ at the requested explicit
centers, the output bound becomes $\delta+C\varepsilon$, with the
same expected call counts.
\end{lemma}

\subsection{Proof roadmap and complexity}\label{ssec:roadmap}
The rest of the paper separates the reusable corrections from their
application to Picard HMC and the analysis of the resulting sampler.
\begin{enumerate}[leftmargin=2em]
\item \Cref{sec:fixed-correction} proves the fluctuation, mean, and hidden
RGO gadget guarantees, including their robustness to approximate inputs
and their composition across Gaussian blocks.
\item \Cref{sec:marginal} develops the application of the gadgets to Picard HMC, including a Gaussian Stein estimator for mean correction.
\item \Cref{sec:recursive} assembles the recursive cloud computation and
bounds its calls and errors, 
proves the main guarantee for the Gaussian cloud sampler (\cref{thm:core}), and applies the proximal sampler wrapper to
obtain \cref{thm:fifth-root}.
\item \Cref{sec:rgo} uses the Gaussian cloud machinery to construct an improved RGO sampler. This is used to implement the hidden RGO sampler.
\item \Cref{sec:gradient-only} develops the implementation using gradient queries alone.
\end{enumerate}

\paragraph{Complexity.}
We now explain the dimension dependence, suppressing all logarithmic
factors and all dependence on $\kappa$ for clarity. Write $\eta$ and $\theta$ for the smoothing
and cloud variances within a phase, which remain comparable to their
baseline values throughout the run. 
The quadrature error in \cref{prop:reference-Picard} involves powers of $h/\sqrt\eta$; thus, we take the natural scale $\eta\asymp h^2$.
We focus on the interior Picard updates (i.e., depth $< K$), since they dominate the complexity.
The step size choice is determined by the following conditions.
\begin{itemize}
\item \textbf{Fluctuation correction.}
    This determines the cloud variance. As a function of its standard
    Gaussian inputs, the gradient estimator has Lipschitz
    coefficient $O(\sqrt{\eta+\theta})$ by
    \cref{lem:joint-estimator-derivatives}. An interior Picard update multiplies the
    gradient by weights of scale $h^2$, and whitening the output cloud
    divides by $\sqrt\theta$. In the regime $\theta\le\eta$, the 
    Lipschitz coefficient of the resulting map is therefore
    $L_{\rm fluc} \lesssim \sqrt{\eta/\theta}\, h^2 \asymp h^3/\sqrt \theta$. By \cref{thm:reusable-correction}, the fluctuation correction succeeds provided $L_{\rm fluc} d^{1/4} \ll 1$.
    This leads to the cloud variance $\theta \asymp \sqrt d\,h^6$.
\item \textbf{Mean correction.}
    The mean estimator of \cref{prop:marginal-score} is of size
    $(\eta^{3/2}+\theta/\sqrt\eta)\sqrt d \asymp (h^3 + \theta/h)\sqrt d$. As before, we multiply by $h^2$
    and divide by $\sqrt\theta$.
    By \cref{prop:reusable-mean}, the mean correction succeeds provided $(h^2/\sqrt \theta)\,(h^3 + \theta/h)\,\sqrt d \ll \clipB_{\rm mean} \le 1$.
    Substituting in $\theta \asymp \sqrt d\,h^6$, this leads to $\clipB_{\rm mean} \asymp d^{1/4} h^2 + d^{3/4} h^4$.
    For $h \gtrsim d^{-1/4}$, the second term dominates, $\clipB_{\rm mean} \asymp d^{3/4} h^4$,
    which leads to the condition $h \ll d^{-3/16}$.
\item \textbf{Explicit RGO calls.}
    The hidden RGO sampler makes a constant number of RGO calls
    at explicit centers, where each such call has variance comparable to
    $\eta$. Then, \cref{thm:fourth-RGO} produces an RGO sampler with query complexity $Q_{\rm RGO} \lesssim 1+(\eta^2d)^{1/4} \asymp d^{1/4} h$.
    This is expensive since $h \gtrsim d^{-1/4}$, but the key is that it does not happen too frequently.
    If we take into account the request frequency from mean correction, this
    costs $O(\clipB_{\rm mean}Q_{\rm RGO})$ expected queries per interior
    update.

\item \textbf{Retained smoothing and the terminal step.}
    The choice $\tau\asymp h^3\asymp\eta h$ also ensures that the
    total added position variance over $N\asymp h^{-1}$ phases
    is $N\tau\asymp h^2\asymp\eta$.
    The constants keep this increase below $\eta_0/4$.
    The single terminal RGO call contributes
    $O(Q_{\rm RGO})$ expected queries.
\end{itemize}
Combining the contributions over the $N\asymp h^{-1}$ phases
leads to the query complexity
\begin{equation*}
    N\,(1+\clipB_{\rm mean}Q_{\rm RGO}) + Q_{\rm RGO}
 \lesssim
 h^{-1}
 + (1 + d^{3/4} h^3)\,d^{1/4} h
 \asymp h^{-1} + dh^4\,.
\end{equation*}
Balancing these terms gives $h\asymp d^{-1/5}$ and $N\asymp d^{1/5}$.

In this balance, part of the dominant cost arises from $Q_{\rm RGO}$.
Generally speaking, if one has a high-accuracy log-concave sampler with complexity $d^p$, then one can hope that by adapting the sampler to the RGO structure---namely, by explicitly making use of the quadratic part of the RGO potential---it leads to an RGO sampler with complexity $1 + (\eta^2 d)^p$.
At our scale, this has complexity $d^p h^{4p}$.
For example, since our Gaussian cloud sampler achieves $d^{1/5}$ complexity, one could expect that this would lead to a better RGO sampler, with $p = 1/5$.
For a general $p$, the cost balance above becomes $h^{-1} + d^{p+3/4} h^{4p+3}$.
After balancing terms, this leads to complexity $d^{(p+3/4)/(4p+4)}$.
Even with an arbitrarily cheap RGO, $p \searrow 0$, this yields complexity no better than $d^{3/16}$, which is also the bottleneck produced by the mean correction.

Therefore, we anticipate that the Gaussian cloud sampler could be pushed slightly beyond $d^{1/5}$, but with a natural bottleneck at $d^{3/16}$.
However, we do not develop this further here.

%% file: 5_high_acc_sections/04_gaussian_correction.tex
\section{Gaussian cloud corrections}\label{sec:fixed-correction}

In this section, we develop (strengthened versions of) \cref{gadget:fluctuation}, \cref{gadget:mean}, and \cref{gadget:hidden-RGO}.

\subsection{A clipping lemma}

We use the notation $\operatorname{clip}_{\clipB}(u)\deq(-\clipB)\vee u\wedge\clipB$.
The following clipping lemma allows the likelihood estimator to be biased.

\begin{lemma}[Clipping with a biased likelihood estimator]\label{lem:clipping}
Let $Q$ be a probability law, let $w$ be a measurable function, and put
\[
 Z\deq\E_Q \exp w\,,\qquad
 Q_w(\dd\omega)\deq Z^{-1}\exp(w(\omega))\,Q(\dd\omega)\,.
\]
Assume $1/2\leq Z<\infty$. Given $\omega$, let $W_\omega$ be a
random estimator and define
\[
 b(\omega)\deq\E[W_\omega\mid\omega]-w(\omega)\,,\qquad
 \widetilde w(\omega)\deq
 \E[\operatorname{clip}_{\clipB}(W_\omega)\mid\omega]\,.
\]
Write $\ol Q$ and $\ol Q_w$ for the joint laws of $(\omega, W_\omega)$ under $\omega \sim Q$ and $\omega \sim Q_w$ respectively.
If $0<\clipB\leq1$, $p\geq2$, and
\[
 \E_Q\abs b+\E_{Q_w}\abs b\leq\eps_{\rm bias}\,,\qquad
 \norm{W_\omega}_{L^p(\ol Q)}
 +\norm{W_\omega}_{L^p(\ol Q_w)}\leq C_0\,,
\]
then the FORS output $Q_{\widetilde w}$ satisfies
\begin{equation*}
 \TV(Q_w,Q_{\widetilde w})
 \leq C\,\Bigl\{\eps_{\rm bias}
       +\clipB\,\Bigl(\frac{C_0}{\clipB}\Bigr)^p\Bigr\}\,.
\end{equation*}
\end{lemma}
\begin{proof}
For $P\in\{Q,Q_w\}$ and its corresponding joint law $\ol P$,
conditional Jensen gives
\[
 \E_P\abs{w-\widetilde w}
 \leq\E_P\abs b+
 \E_{\ol P}\abs{W_\omega-\operatorname{clip}_{\clipB}(W_\omega)}
 \leq\E_P\abs b+\clipB\,\Bigl(\frac{C_0}{\clipB}\Bigr)^p\,.
\]
Since $\abs{\widetilde w}\leq\clipB$, the inequality
$\abs{\exp u-\exp v}\leq\abs{u-v}\,(\exp u+\exp v)$ yields
\[
 \Delta\deq\int\abs{\exp w-\exp\widetilde w}\,\dd Q
 \leq Z\,\E_{Q_w}\abs{w-\widetilde w}
      +\e^{\clipB}\,\E_Q\abs{w-\widetilde w}\,.
\]
Moreover, for $\widetilde Z \deq \E_Q \exp\widetilde w$,
\[
 \abs{Z-\widetilde Z}
 =\Bigl\lvert\int(\exp w-\exp\widetilde w)\,\dd Q\Bigr\rvert
 \leq \Delta\,.
\]
Comparing the normalized densities with respect to $Q$ therefore gives
\begin{align*}
 \TV(Q_w,Q_{\widetilde w})
 &=\frac12\int
   \Bigl\lvert \frac{\exp w}{Z}-\frac{\exp\widetilde w}{\widetilde Z}\Bigr\rvert\,\dd Q
 \leq\frac{\Delta}{2Z}
 +\frac12\int (\exp\widetilde w)\,
       \bigl\lvert \frac1Z-\frac1{\widetilde Z}\bigr\rvert\,\dd Q
 \leq\frac \Delta Z\,.
\end{align*}
Finally, the bound on $\E_P\abs{w-\widetilde w}$ yields
\begin{align*}
 \frac \Delta Z
 &\leq \E_{Q_w}\abs b+\frac{\e^{\clipB}}{Z}\,\E_Q\abs b
 +\Bigl(1+\frac{\e^{\clipB}}{Z}\Bigr)\,
 \clipB\,\Bigl(\frac{C_0}{\clipB}\Bigr)^p
 \leq (1+2\e)\,\Bigl\{\eps_{\rm bias}
 +\clipB\,\Bigl(\frac{C_0}{\clipB}\Bigr)^p\Bigr\}\,,
\end{align*}
where the last inequality uses $\clipB\leq1$ and $Z\geq1/2$.
This proves the desired bound.
\end{proof}

\subsection{Gadget~\ref{gadget:fluctuation}: Fluctuation correction}
\label{sec:fixed-node}

The target of \cref{gadget:fluctuation} is
$\cN(\E\Phi(G),\Id_n)$, for a standard Gaussian $G\in\R^m$.
Here, we describe a significant extension to facilitate implementation.
Toward that end, let $\Phi^\circ$ be a comparison map
with $\Lip(\Phi^\circ)\le L$, and let $\widehat\Phi$ be the actual implemented map.
The comparison map is only for analysis; it is not evaluated. Define the two value
errors
\begin{equation}\label{eq:robust-value-errors}
 e_{\rm val}\deq\E\norm{\widehat\Phi(G)-\Phi^\circ(G)}\,,\qquad
 e_{\rm mean}\deq\norm{\E\Phi^\circ(G)-\E\Phi(G)}\,.
\end{equation}
All three maps are assumed integrable. Taking
$\widehat\Phi=\Phi^\circ=\Phi$ recovers the exact case.
Let $B$ be a standard Brownian
path with $B_1=G$, and put
\begin{align*}
 A_t&\deq\E[D\Phi^\circ(B_1)\mid \mathcal F_t]\,,
 \qquad \Sigma\deq\int_0^1 A_tA_t^\T\,\dd t\,, \qquad
 0\preceq\Sigma\preceq L^2\Id_n\,.
\end{align*}

The algorithm requires a
randomized map
$\widehat\Sigma:\R^n\to\R^n$ with the following
properties, conditional on the Brownian path, for some
$\varepsilon\ge0$:
\begin{enumerate}[label=\textup{(\alph*)},leftmargin=2.4em]
\item Almost surely, $\widehat\Sigma$ is odd and
$L^2$-Lipschitz.
\item For each supplied direction $v\in\R^n$,
\begin{equation}\label{eq:robust-action-bias}
 \norm{\E[\widehat\Sigma(v)\mid B]-\Sigma v}
 \le\varepsilon L^2\,\norm v\,.
\end{equation}
\end{enumerate}
Here, we allow $\widehat \Sigma(\cdot)$ to be a non-linear function (despite the fact that it approximates a linear action).
\Cref{lem:approximate-derivatives} below gives a construction of this
routine from Jacobian-vector products.

Use \cref{alg:reusable-fluctuation}, drawing for each tilt estimator $K$ fresh copies
$\widehat\Sigma_1,\ldots,\widehat\Sigma_K$ of $\widehat\Sigma$,
independently conditional on $B$. Define the successive compositions
\[
 \widehat{\Sigma^0}\deq\Id_n\,,\qquad
 \widehat{\Sigma^k}\deq
 \widehat\Sigma_k\circ\widehat{\Sigma^{k-1}}\,,
 \qquad k\in\{1,\ldots,K\}\,.
\]
For independent standard Gaussians $Z,Z'\in\R^n$, we set
\begin{equation}\label{eq:robust-fluctuation-tilt-estimator}
 \widetilde W_K
 \deq\frac12\sum_{k=1}^K
 \Bigl\{\frac1k\,\ip{Z'}{\widehat{\Sigma^k}(Z')}
 -\ip Z{\widehat{\Sigma^k}(Z)}\Bigr\}\,.
\end{equation}
Conditional on $B$, all of the randomness used in the $\widehat \Sigma$ map is independent of $Z,Z'$, and we use the same realization of $\widehat{\Sigma^k}$ in both terms.
Evaluate these compositions successively on $Z$ and $Z'$, reusing the
two vectors from level $k-1$ at level $k$. Thus, the entire computation of $\widetilde W_K$ uses
$2K$ evaluations of $\widehat\Sigma(\cdot)$.
After FORS acceptance, return
\begin{equation}\label{eq:gradient-unused-guard}
 \widehat Y\deq\widehat\Phi(G)+Z\,.
\end{equation}

\begin{theorem}[Robust fluctuation correction]\label{thm:robust-fluctuation}
Assume the preceding oracle access models and bounds, and
$K\ge1$. For $p\ge2$, put
\[
 e_p\deq C\,\bigl\{
 L^2\,(\sqrt{np}+p)
 +L^4 n\bigr\}\,.
\]
Suppose that
\[
    L^2 \le \frac{1}{4} \qquad\text{and}\qquad
 CL^{2K+2}n \le c_{\rm norm}\,,
\]
where
$c_{\rm norm}>0$ is a sufficiently small universal constant.
For $0<\clipB_{\rm fluc}\le1$, the output law $\widehat\nu$ satisfies
\begin{align}\label{eq:robust-fluctuation-error}
 \TV\bigl(\widehat\nu,\cN(\E\Phi(G),\Id_n)\bigr)
 \le{}&C\,\Bigl\{
 L^{2K+2} n
 +L^2 n\varepsilon
 +\clipB_{\rm fluc}\,\Bigl(\frac{e_p}{\clipB_{\rm fluc}}\Bigr)^p
 +e_{\rm val}\Bigr\}+\frac{e_{\rm mean}}2\,.
\end{align}
The expected numbers
of proposal evaluations and $\widehat \Sigma(\cdot)$ evaluations are at most
$1+C\clipB_{\rm fluc}$ and $C\clipB_{\rm fluc}K$, respectively.
\end{theorem}

We first record the two estimates used in the proof.

\begin{lemma}[Cancellation for odd Lipschitz maps]\label{lem:gradient-odd}
Let $T:\R^n\to\R^n$ be odd and Lipschitz, and let $Z',Z$ be
independent standard Gaussians. For $p\ge2$,
\begin{equation*}
 \bigl\|\ip{Z'}{T(Z')}-\ip Z{T(Z)}\bigr\|_{L^p}
 \le C\,\norm{T}_{\rm Lip}\,(\sqrt{np}+p)\,.
\end{equation*}
For a fixed $u\in\R^n$,
$\|\ip u{T(Z)}\|_{L^p}
\le C\sqrt p\,\norm{T}_{\rm Lip}\,\norm u$.
\end{lemma}
\begin{proof}
The weak gradient of $x\mapsto\ip x{T(x)}$ has norm at most
$2\,\norm{T}_{\rm Lip}\,\norm x$. The Gaussian $L^p$ Poincar\'e inequality bounds its
centered $L^p$ norm by
$C\sqrt p\,\norm{T}_{\rm Lip}\,(\sqrt n+\sqrt p)$, proving the first assertion.
Also, $x\mapsto\ip u{T(x)}$ is odd and
$\norm{T}_{\rm Lip}\,\norm u$-Lipschitz, so Gaussian concentration proves the second.
\end{proof}

\begin{lemma}[Normalization constant for the truncated series]
\label{lem:fixed-truncated-normalizer}
Let $0\preceq\Sigma_0\preceq\Id_n/4$, and put
\begin{equation}\label{eq:Gaussian-likelihood-series}
 w_K(Z)\deq\frac12\sum_{k=1}^K
 \Bigl\{\frac1k\tr(\Sigma_0^k)-Z^\T\Sigma_0^kZ\Bigr\}\,.
\end{equation}
For $Z\sim\cN(0,\Id_n)$, let $\mathcal Z_K\deq\E\exp w_K(Z)$.
Then,
\begin{equation}\label{eq:fixed-normalizer-log-bound}
 \abs{\log\mathcal Z_K}\le Cn\,\norm{\Sigma_0}_{\op}^{K+1}\,.
\end{equation}
The tilted law is $\cN(0,C_K)$, where
\begin{equation*}
 C_K\deq\Bigl(\Id_n+\sum_{k=1}^K\Sigma_0^k\Bigr)^{-1}\,,\qquad
 \frac34\,\Id_n\preceq C_K\preceq\Id_n\,,\qquad
 \norm{C_K-(\Id_n-\Sigma_0)}_{\op}
 \le C\,\norm{\Sigma_0}_{\op}^{K+1}\,.
\end{equation*}
Consequently,
\begin{equation}\label{eq:fixed-truncated-Gaussian-TV}
 \TV\bigl(\cN(0,C_K),\cN(0,\Id_n-\Sigma_0)\bigr)
 \le C\sqrt n\,\norm{\Sigma_0}_{\op}^{K+1}\,.
\end{equation}
\end{lemma}
\begin{proof}
Write $H_K\deq \sum_{k=1}^K\Sigma_0^k$. Gaussian integration gives
\[
 \log\mathcal Z_K=\frac12\sum_{k=1}^K\frac1k\tr(\Sigma_0^k)
             -\frac12\log\det(\Id_n+H_K)\,.
\]
This tends to zero as $K\to\infty$, since
$\Id_n+H_\infty=(\Id_n-\Sigma_0)^{-1}$. The bounds
$\norm{H_K-H_\infty}_{\op}
 \le C\,\norm{\Sigma_0}_{\op}^{K+1}$ and
$\sum_{k>K}\tr(\Sigma_0^k)/k
 \le Cn\,\norm{\Sigma_0}_{\op}^{K+1}$ prove
\eqref{eq:fixed-normalizer-log-bound}. The
Gaussian KL formula followed by Pinsker's inequality gives
\eqref{eq:fixed-truncated-Gaussian-TV}.
\end{proof}

\begin{proof}[Proof of \cref{thm:robust-fluctuation}]
    First, we handle the truncation error.
    As in \cref{sec:reusable-correction},
Clark--Ocone gives
$\Phi^\circ(G)-\E\Phi^\circ(G)=\int_0^1A_t\,\dd B_t$.
Under the exact tilt, $Z\mid B\sim\cN(0,\Id_n-\Sigma)$,
and the exponential martingale identity gives
\begin{align*}
 &\E\exp\bigl\{\ii\,\ip u{\Phi^\circ(G)-\E\Phi^\circ(G)
                              +Z}\bigr\}
                              =\exp\bigl(- \frac{1}{2}\,\norm u^2\bigr)\,.
\end{align*}
Hence, $\Phi^\circ(G) + Z \sim \cN(\E \Phi^\circ(G), I_n)$.
Denote this exact joint law of $(B,Z)$ by
$P_\infty$.

Let $Q$ be the law of $(B,Z)$ before correction. Use
\eqref{eq:Gaussian-likelihood-series} with $\Sigma_0=\Sigma$, and
let $P_K\propto\e^{w_K}Q$ be the tilted law. By
\cref{lem:fixed-truncated-normalizer}, the normalizing constant
$\mathcal Z_K\deq\E[\exp w_K(Z)\mid B]$
satisfies
\[
 \abs{\log\mathcal Z_K}
 \le \alpha_K
 \deq CL^{2K+2} n
\]
almost surely. Write $Q_B$ for the original Brownian path law. Integrating
the tilted density over $Z$ shows that the Brownian path marginal of $P_K$ is
\begin{equation*}
 \frac{\dd (P_K)_B}{\dd Q_B}(B)
 =\frac{\mathcal Z_K}{\E_{Q_B}\mathcal Z_K}\,.
\end{equation*}
If $c_{\rm norm}$ is sufficiently small, then
$\mathcal Z_K\in[2/3,3/2]$, so this density takes values in
$[4/9,9/4]$: more quantitatively,
$\abs{\log(\mathcal Z_K/\E_{Q_B}\mathcal Z_K)}\le2\alpha_K$; hence
\begin{equation*}
 \TV\bigl((P_K)_B,Q_B\bigr)
 =\frac12\,\E_{Q_B}
 \Bigl\lvert\frac{\mathcal Z_K}{\E_{Q_B}\mathcal Z_K}-1\Bigr\rvert
 \le C\alpha_K\,.
\end{equation*}
Conditional on $B$, the law of $Z$ under $P_K$ has covariance
$C_K$, whereas under $P_\infty$ it has covariance $\Id_n-\Sigma$.
Thus, $\mathcal Z_\infty = 1$, and the
Brownian path marginal of $P_\infty$ is precisely $Q_B$. Therefore,
disintegration and \eqref{eq:fixed-truncated-Gaussian-TV} give
\begin{align*}
 \TV(P_K,P_\infty)
 &\le \TV\bigl((P_K)_B,Q_B\bigr)
 +\E_{(P_K)_B}\TV\bigl(
       \cN(0,C_K),\cN(0,\Id_n-\Sigma)\bigr)
 \le CL^{2K+2} n\,.
\end{align*}

Next, we handle the clipping error. Each $\widehat{\Sigma^k}$ is odd and
$L^{2k}$-Lipschitz.
Fix $v$ and note that the randomness of $\widehat\Sigma_k$ is independent of
$\widehat{\Sigma^{k-1}}(v)$ conditional on $B$. Hence
\eqref{eq:robust-action-bias} gives
\[
 \norm{\E[\widehat{\Sigma^k}(v)\mid B]
 -\Sigma\,\E[\widehat{\Sigma^{k-1}}(v)\mid B]}
 \le\varepsilon L^2\,
 \E[\norm{\widehat{\Sigma^{k-1}}(v)}\mid B]
 \le\varepsilon
 L^{2k}\,\norm v\,.
\]
Since $\norm{\Sigma}_{\op}\le L^2$, it follows that
\[
 \norm{\E[\widehat{\Sigma^k}(v)\mid B]-\Sigma^kv}
 \le \varepsilon
 L^{2k}\,\norm v
 +L^2\,
 \norm{\E[\widehat{\Sigma^{k-1}}(v)\mid B]-\Sigma^{k-1}v}\,.
\]
Iterating yields
\begin{equation}\label{eq:robust-power-bias}
 \norm{\E[\widehat{\Sigma^k}(v)\mid B]-\Sigma^kv}
 \le k\varepsilon
 L^{2k}\,\norm v\,,
 \qquad k\ge1\,.
\end{equation}
In particular, conditioning first on $Z'$ and using
$\E\norm{Z'}^2=n$ gives
\[
 \abs{\E[
  \ip{Z'}{\widehat{\Sigma^k}(Z')-\Sigma^kZ'}\mid B]}
 \le k\varepsilon
 L^{2k}\, n\,.
\]
The corresponding conditional bound for the $Z$ term is
\[
 \abs{\E[
  \ip Z{\widehat{\Sigma^k}(Z)-\Sigma^kZ}\mid B,Z]}
 \le k\varepsilon
 L^{2k}\,\norm Z^2\,.
\]
Substitution in \eqref{eq:robust-fluctuation-tilt-estimator}, followed by summation
over $k\in\{1,\ldots,K\}$, consequently yields the following bound; here
we use that both the geometric sum and its first moment are bounded by
$CL^2$ because $L^2\le1/4$:
\[
 \abs{\E[\widetilde W_K\mid B,Z]-w_K(Z)}
 \le C\varepsilon L^2\,\{n+\norm Z^2\}\,.
\]
Since $\E_Q\norm Z^2 \vee \E_{P_K}\norm Z^2\le n$, this yields
$\eps_{\rm bias} \le CL^2 n\varepsilon$ in \cref{lem:clipping}.

For the moment bound, first condition on $B$ and all randomness used in the $\widehat \Sigma$ map. Under
$Q$, the $k$-th summand in
\eqref{eq:robust-fluctuation-tilt-estimator} is $\frac1k\,\ip{Z'}{\widehat{\Sigma^k}(Z')}
 -\ip Z{\widehat{\Sigma^k}(Z)}$.
For $k=1$, \cref{lem:gradient-odd} bounds its
$L^p$ norm by $CL^2\,(\sqrt{np}+p)$.
For $k\ge2$, use the identity
\begin{align*}
 &\frac1k\,\ip{Z'}{\widehat{\Sigma^k}(Z')}
 -\ip Z{\widehat{\Sigma^k}(Z)}
 =\ip{Z'}{\widehat{\Sigma^k}(Z')}
  -\ip Z{\widehat{\Sigma^k}(Z)}
  -\Bigl(1-\frac1k\Bigr)\,
   \ip{Z'}{\widehat{\Sigma^k}(Z')}\,.
\end{align*}
Then, \cref{lem:gradient-odd} bounds the $L^p$ norm of the first two terms by $CL^{2k}\,(\sqrt{np}+p)$.
Since $\widehat{\Sigma^k}$ is odd, it vanishes at zero, and hence $\norm{\ip{Z'}{\widehat{\Sigma^k}(Z')}}_{L^p}
 \le L^{2k}\,
 \norm{\norm{Z'}^2}_{L^p}
 \le CL^{2k}\,(n+p)$.
Summing over
$k\in\{1,\ldots,K\}$ gives
\begin{align*}
 \norm{\widetilde W_K}_{L^p(Q)}
 &\le C\,\bigl\{
 L^2\,(\sqrt{np}+p)
 +L^4\,(n+p)\bigr\}
 \le e_p/2\,.
\end{align*}
In the last inequality, the term
$pL^4$ is absorbed by
$pL^2$.

Under $P_K$, conditional on $B$, write
$Z=C_K^{1/2}\widetilde Z$, where $\widetilde Z$ is standard Gaussian and
independent of $Z'$ and the random seeds. Moreover,
$\norm{C_K^{1/2}}_{\op}\le1$ and
$\norm{C_K^{1/2}-\Id_n}_{\op}\le CL^2$. Adding and
subtracting
$\ip{C_K^{1/2}\widetilde Z}{\widehat{\Sigma^k}(\widetilde Z)}$
therefore gives $\abs{\ip Z{\widehat{\Sigma^k}(Z)}
 -\ip{\widetilde Z}{\widehat{\Sigma^k}(\widetilde Z)}}
 \le
 CL^{2k+2}\,\norm{\widetilde Z}^2$.
Summing over $k\in\{1,\ldots,K\}$ and using
$\norm{\norm{\widetilde Z}^2}_{L^p}\le C\,(n+p)$ bounds the additional
term by
$CL^4\,(n+p)$. After increasing the constant in
$e_p$, this gives $\norm{\widetilde W_K}_{L^p(P_K)}
 \le e_p/2$ as well.

All summations used only the triangle inequality, so reusing the earlier
evaluation of $\widehat{\Sigma^k}$ to compute $\widehat{\Sigma^{k+1}}$ is valid, despite the failure of independence.

\Cref{lem:clipping}
now bounds the clipping error by
$C\,\{L^2 n\varepsilon
 +\clipB_{\rm fluc}\,(e_p/\clipB_{\rm fluc})^p\}$.

The preceding estimates control the TV distance between the accepted
joint law of $(B,Z)$ and $P_\infty$. Applying the same post-processing \eqref{eq:gradient-unused-guard} to both laws preserves this
bound by data processing. Under $P_\infty$, conditional on $B$, the
additive Gaussian noise has covariance
\[
 \Id_n-\Sigma
 =\Id_n-\int_0^1 A_tA_t^\T\,\dd t
 \succeq(1-L^2)\,\Id_n\,.
\]
The implemented and comparison outputs differ only by
$\widehat\Phi(G)-\Phi^\circ(G)$. Applying the Gaussian KL formula and Pinsker's inequality conditionally on $B$, then averaging using joint convexity of TV, bounds the TV distance between these outputs by
\[
 \frac{\E\norm{\widehat\Phi(G)-\Phi^\circ(G)}}{2\sqrt{1-L^2}}
 =\frac{e_{\rm val}}{2\sqrt{1-L^2}}\le Ce_{\rm val}\,.
\]
Here, $L^2\le1/4$.
Changing the target mean from $\E\Phi^\circ(G)$ to $\E\Phi(G)$
costs at most $e_{\rm mean}/2$. This proves
\eqref{eq:robust-fluctuation-error}.

Each attempt accepts with probability at least $\exp(-2\clipB_{\rm fluc})$.
It uses one proposal evaluation and a Poisson number of tilt estimators of mean
$2\clipB_{\rm fluc}$. Each tilt estimator uses $2K$ $\widehat \Sigma(\cdot)$ evaluations. Summing over attempts gives the stated
expected costs, since $\exp(2\clipB_{\rm fluc})\le1+C\clipB_{\rm fluc}$.
\end{proof}

\begin{remark}[Sharper truncation error]
Under the same hypotheses, the
truncation contribution in \eqref{eq:robust-fluctuation-error} can be
improved to
$C\min\{n,(K+1)\sqrt n\}\,L^{2K+2}$ by controlling the
fluctuations of $\log\mathcal Z_K$. Since we do not need this more refined bound, we only sketch the details here. Write
$\log\mathcal Z_K=\tr f_K(\Sigma)$, where $f_K(x)\deq\frac12\,\{
 \sum_{k=1}^K\frac{x^k}{k}
 -\log\sum_{k=0}^Kx^k\}$.
 Then,
 $\sup_{x\in[0,L^2]}\abs{f'_K(x)}\le CL^{2K} K$.
For $\Psi_t(x)\deq\E\Phi^\circ(x+\sqrt{1-t}\,G')$, with
$G'\sim\cN(0,\Id_m)$, we have $A_t=D\Psi_t(B_t)$.
Gaussian integration by parts gives, for every fixed matrix
$M\in\R^{n\times m}$,
\[
 \norm{\nabla_x\ip{M}{D\Psi_t(x)}_{\HS}}
 \le\frac{L\,\norm M_{\HS}}{\sqrt{1-t}}\,,
 \qquad 0\le t<1\,.
\]
Consequently, the Malliavin derivative of $\log\mathcal Z_K$ satisfies
\[
    \norm{D_s(\log\mathcal Z_K)}
 \le4L^2\sqrt n\,\norm{f'_K(\Sigma)}_{\op}\sqrt{1-s}\,,
 \qquad 0\le s<1\,.
\]
Gaussian log-Sobolev then gives
\[
 \KL\bigl((P_K)_B\bigm\Vert Q_B\bigr)
 \le\frac12\,\E_{P_K}\int_0^1\norm{D_s (\log\mathcal Z_K)}^2\,\dd s
 \le C L^{4K+4} K^2 n\,.
\]
Pinsker and \eqref{eq:fixed-truncated-Gaussian-TV} yield the claimed
improvement.
\end{remark}

The following lemma shows that the routine $\widehat \Sigma(\cdot)$ can be constructed by composing routines for the Jacobian-vector products.

\begin{lemma}[Approximate Jacobian-vector products]
\label{lem:approximate-derivatives}
Suppose $\Phi^\circ:\R^m\to\R^n$ is $L$-Lipschitz. At almost every $g$, suppose that the randomized routines $\mathfrak J_g$ and
$\mathfrak J_g^\T$ are odd and $L$-Lipschitz almost surely, and satisfy
\begin{align*}
 \norm{\E\mathfrak J_g(v)-D\Phi^\circ(g)v}
 &\le\varepsilon L\,\norm v\,,
 \qquad \norm{\E\mathfrak J_g^\T(w)-D\Phi^\circ(g)^\T w}
 \le\varepsilon L\,\norm w\,.
\end{align*}
These inequalities hold for all $v\in\R^m$ and $w\in\R^n$.
Draw
$T\sim\mathsf{Unif}[0,1]$ independently of $B$. Conditional on $(T,B_T)$, draw $G_1,G_2$ independently from $\cN(B_T,(1-T)\Id_m)$, with the pair $(G_1,G_2)$ conditionally independent of the full path $B$. Set
\[
 \widehat\Sigma(w)\deq
 \mathfrak J_{G_1}\bigl(\mathfrak J_{G_2}^\T(w)\bigr)\,.
\]
Use independent random seeds for the two maps. Then, the hypotheses on
$\widehat\Sigma$ hold with $2\varepsilon$ in
\eqref{eq:robust-action-bias}.
\end{lemma}
\begin{proof}
Clearly, $\widehat \Sigma$ is odd and $L^2$-Lipschitz.
To verify~\eqref{eq:robust-action-bias}, use the triangle inequality.
\end{proof}

In particular, we can now verify the guarantee for the simpler gadget in \cref{thm:reusable-correction}.

\begin{proof}[Proof of \cref{thm:reusable-correction}]
Apply the bounds in
\cref{thm:robust-fluctuation} with
$\widehat\Phi=\Phi^\circ=\Phi$, $\varepsilon=0$,
and the exact routines from \cref{lem:approximate-derivatives}.
The algorithm is then precisely
\cref{alg:reusable-fluctuation}. Choose $p=\lceil C\mathfrak L\rceil$
and $K=\lceil C\mathfrak L\rceil$. Since $nL^4 \le c\clipB_{\rm fluc}$
and $L^2\le1/8$,
\begin{equation*}
 nL^{2K+2}=(nL^4)\,(L^2)^{K-1}
 \le c\clipB_{\rm fluc}\,8^{-(K-1)}\le c\delta\,.
\end{equation*}
This controls the truncation error. The error terms sum to at most $\delta$ after adjusting the
constants.
\end{proof}

In \cref{sec:gradient-only}, we develop an alternative construction of
the fluctuation correction that avoids the use of Jacobian-vector products.

\subsection{Gadget~\ref{gadget:mean}: Mean correction}

Here, we prove the guarantee for \cref{gadget:mean}.

\begin{proof}[Proof of \cref{prop:reusable-mean}]
    Let $Q=\cN(\mu,\Id_m)$, $W \deq \langle \Delta, Y-Y'\rangle$, and $w(Y) \deq \E[W \mid Y] = \langle \E \Delta, Y-\mu\rangle$, where $Y'\sim Q$ and $\Delta$ are independent of $Y$.
    As shown in \cref{sec:reusable-correction}, the normalized tilt $Q_w$ is the desired law
$\cN(\mu+\E\Delta,\Id_m)$.

Under $Q$, the difference
$Y-Y'$ is a centered Gaussian of covariance $2\Id_m$. Under $Q_w$, it
has mean $\E\Delta$ and covariance $2\Id_m$. Under both laws,
$\Delta$ is independent of $Y-Y'$. Conditioning on $\Delta$ and using
Gaussian moments therefore give
\begin{equation*}
 \norm W_{L^p}
 \le C\sqrt p\,\norm{\Delta}_{L^p}
 +\norm{\E\Delta}\,\norm{\Delta}_{L^p}
 \le C\,(\sqrt p\,b+b^2)\,.
\end{equation*}
Gaussian integration gives
\[
 1\le Z\deq\E_Q\exp w(Y)
 =\exp\Bigl(\frac12\,\norm{\E\Delta}^2\Bigr)<\infty\,,
\]
which verifies the normalizing constant hypothesis of \cref{lem:clipping}.
Applying \cref{lem:clipping} with $\eps_{\rm bias}=0$ shows that
clipping changes the output law in TV by at most
\begin{equation*}
 C \clipB_{\rm mean}\,
 \Bigl(\frac{C\,(\sqrt p\,b+b^2)}{\clipB_{\rm mean}}\Bigr)^p
 \le\delta\,,
\end{equation*}
with the stated choices of constants.

By \cref{lem:reusable-FORS}, the expected numbers of proposal and tilt estimator
calls are at most $\exp(2\clipB_{\rm mean})$ and
$2\clipB_{\rm mean}\exp(2\clipB_{\rm mean})$, respectively. Each proposal call uses
one draw from $\widehat Q$, while each tilt estimator call uses one further
draw from $\widehat Q$ and one draw of $\Delta$. Since
$\clipB_{\rm mean}\le1$, the stated call bounds follow.

For approximate proposals, maximally couple every requested draw from
$\widehat Q$ to a draw from $Q$. Until the first coupling failure, the
exact and approximate algorithms make the same requests and the same
acceptance decisions. The expected number of requested Gaussian draws
is bounded by a universal constant, so a union bound gives coupling
failure probability at most $C\varepsilon$. This proves the claimed
output TV bound.
\end{proof}

\subsection{Composition and block evaluation}\label{sec:reusable-proofs}

The map $\Phi$ may represent an entire computation driven by a Gaussian
input. After correcting its fluctuations, we use independent copies of
the output as the Gaussian draws requested by the mean gadget. The
estimator of the missing mean is supplied separately.

\begin{corollary}[Composition of fluctuation and mean correction]
\label{cor:block-correction}
Fix $0<\delta\le1/2$, $p\ge C\log(2/\delta)$, and
$\clipB_{\rm fluc},\clipB_{\rm mean}\in(0,1]$. Assume the hypotheses of
\cref{thm:robust-fluctuation} at series order $K$, with the right-hand
side of \eqref{eq:robust-fluctuation-error} at most $c\delta$.
Suppose a sampler returns independent copies of $\Delta\in\R^n$,
independently of all fluctuation corrector randomness, and $\norm{\Delta}_{L^p}\le b$,
 $\sqrt p\,b+b^2\le c\clipB_{\rm mean}$.
Take the fluctuation corrector to be $\widehat Q$ in
\cref{alg:reusable-mean}. The output law $\widehat P$ then satisfies
\[
 \TV\bigl(\widehat P,
 \cN(\E\Phi(G)+\E\Delta,\Id_n)\bigr)\le\delta\,.
\]
The expected counts of primary proposals, $\widehat\Sigma(\cdot)$ evaluations, and
$\Delta$ samples are at most
\[
 1+C\,(\clipB_{\rm fluc}+\clipB_{\rm mean})\,,\qquad
 C\clipB_{\rm fluc}K\,,\qquad C\clipB_{\rm mean}\,,
\]
respectively.
If $\widehat\Sigma$ is implemented using \cref{lem:approximate-derivatives}, the expected number of evaluations of the forward and transpose Jacobian-vector routines is also at most $C\clipB_{\rm fluc}K$.
\end{corollary}
\begin{proof}
By \cref{thm:robust-fluctuation}, the fluctuation corrector has
TV error at most $c\delta$. Apply \cref{prop:reusable-mean} at a
sufficiently small constant multiple of $\delta$. Its required moment
order is at most $p$, so the stated assumptions give total TV error at
most $\delta$ after adjusting the universal constants.

The mean gadget uses at most $1+C\clipB_{\rm mean}$ expected calls to the
fluctuation gadget. Multiplying this count by the
$1+C\clipB_{\rm fluc}$ primary calls per fluctuation output gives the first
bound; the product term is absorbed since both clipping levels are at
most one. The same argument gives the $\widehat\Sigma(\cdot)$ evaluation count,
and \cref{prop:reusable-mean} gives the mean replica count.
Under the implementation of \cref{lem:approximate-derivatives}, each $\widehat\Sigma$ evaluation uses one forward and one transpose routine evaluation, giving the same bound up to a universal constant.
\end{proof}

\paragraph{Independent Gaussian blocks.}
Let $G\deq (G_1,\ldots,G_J)$ have independent standard Gaussian blocks,
let $\Phi_j$ be $L_j$-Lipschitz, and let $A_j$ be known linear maps
with common output space $\R^n$. Put
\[
 \Phi(G)\deq\sum_{j=1}^JA_j\Phi_j(G_j)\,,\qquad
 L^2\deq\sum_{j=1}^JL_j^2\,\norm{A_j}_{\op}^2\,.
\]
Cauchy--Schwarz gives $\Lip(\Phi)^2\le L^2$.
Suppose that a joint draw of $(\Delta_1,\ldots,\Delta_J)$ is independent
of the fluctuation sampler, with
$\norm{\Delta_j}_{L^p}\le b_{j,p}$ for $j\in\{1,\ldots,J\}$.
The $\Delta_j$ may be coupled. Then,
Minkowski's inequality gives
\[
 \Delta\deq\sum_{j=1}^JA_j\Delta_j\,,\qquad
 \norm{\Delta}_{L^p}\le
 b_p\deq\sum_{j=1}^J b_{j,p}\,\norm{A_j}_{\op}\,.
\]
The corrected target distribution is
\[
 \cN\biggl(\sum_{j=1}^JA_j\,\{\E\Phi_j(G_j)+\E\Delta_j\},\,\Id_n\biggr)\,.
\]
Then, \cref{cor:block-correction} applies, provided that $L^2 \le 1/8$ and
\begin{equation}\label{eq:fixed-node-scale}
    CL^{2K+2}n\le \delta\,, \qquad
 C\,\{L^2\,(\sqrt{np}+p)+L^4 n\}\leq \clipB_{\rm fluc}\,,\qquad
 C\,\{\sqrt p\,b_p+b_p^2\}\leq \clipB_{\rm mean}\,.
\end{equation}

\paragraph{Evaluating one block at a time.}
The block structure also reduces the number of block subroutine calls made by
the fluctuation correction, by subsampling the blocks to correct.
If $L=0$, then $\Phi$ is constant and the raw proposal $\Phi(G)+Z$ already has law $\cN(\E\Phi(G),\Id_n)$, so omit the fluctuation correction and apply the mean correction as needed. Thus, in the following construction, assume $L>0$.
Suppose first that we have exact evaluations of the $D\Phi_j$.
For each $i\in[J]$, write $G_i=B_{i,1}$ for an independent standard
Brownian motion $(B_{i,t})_{0\leq t\leq1}$, let
$\mathcal F_{i,t}\deq \sigma(B_{i,s}:0\leq s\leq t)$, and define
\begin{equation*}
 J_{i,t}\deq\E[D\Phi_i(G_i)\mid\mathcal F_{i,t}]
 =\E_{U_i}D\Phi_i\bigl(B_{i,t}+\sqrt{1-t}\,U_i\bigr)\,,
\end{equation*}
where $U_i$ is an independent standard Gaussian of the same dimension as
$G_i$. Let $B$ denote the collection of these Brownian paths.
Omit indices $j$ for which
$L_j\,\norm{A_j}_{\op}=0$, draw $\mathcal J$ with probabilities $\Pp\{\mathcal J=j\}=L_j^2\,\norm{A_j}_{\op}^2/L^2$,
and draw
$T\sim\mathsf{Unif}[0,1]$ and
independent standard Gaussians $U_{\mathcal J}^{(1)}$, $U_{\mathcal J}^{(2)}$.
Then, we use
\begin{equation*}
 \widehat\Sigma(v)\deq\frac{L^2}{L_{\mathcal J}^2\,\norm{A_{\mathcal J}}_{\op}^2}\,
 A_{\mathcal J}\,D\Phi_{\mathcal J}(B_{\mathcal J,T} + \sqrt{1-T}\,U_{\mathcal J}^{(1)})\,D\Phi_{\mathcal J}(B_{\mathcal J,T} + \sqrt{1-T}\,U_{\mathcal J}^{(2)})^\T\, A_{\mathcal J}^\T\, v\,.
\end{equation*}
Then, we have
\begin{equation}\label{eq:fixed-random-S-properties}
 \E[\widehat\Sigma(v)\mid B]=\Sigma v\,,\qquad
 \Lip(\widehat\Sigma)\le L^2\,,\qquad
 \Sigma\deq\sum_{j=1}^J\int_0^1
 A_jJ_{j,t}J_{j,t}^\T A_j^\T\,\dd t\,.
\end{equation}
When exact Jacobian-vector products are unavailable,
we can again apply \cref{lem:approximate-derivatives}.
The key point is that each derivative action calls only the selected block.

Each tilt estimator for the fluctuation correction makes at most $CK$ Jacobian-vector product calls.
Consequently,
the expected number of such calls is at most $C\clipB_{\rm fluc}K$.

Poisson counts and rejected attempts have uniform exponential tails
for clipping levels at most one. For example, if $M$ is the number of tilt estimators used
in one attempt at level $\clipB\le1$, then for a sufficiently small fixed
$t>0$,
\[
    \E[\exp(tM)\,\mathbf1_{\{\text{rejection}\}}]
 \le1-\e^{-2}+\exp\{2\,(\e^t-1)\}-1<1\,.
\]
This bounds the total number
of tilt estimators by $C\,(u+1)$ with probability at least $1-\exp(-u)$.
Hence, the total number of Jacobian-vector products is at most
$CK\,(u+1)$ on the same event. For $K=O(\mathfrak L)$, this is at most
$C\mathfrak L^2$ when $u=O(\mathfrak L)$, while its expectation is at most
$C\clipB_{\rm fluc}\mathfrak L$.

\subsection{Gadget~\ref{gadget:hidden-RGO}: Hidden RGO sampler}\label{sec:hidden-RGO-proof}
We use the notation of \cref{alg:hidden-RGO}. Section~3.3 already
identifies the ideal center law \eqref{eq:hidden-RGO-center-law} and shows
that the final RGO draw has marginal $R_{\eta,y}$. It remains to control
the likelihood moments, clipping error, and oracle calls.

Fix $y,z\in\R^d$ and deterministic bounds $\varepsilon_0,L\ge0$ with
\begin{equation*}
 \norm{y-z}\le\varepsilon_0\,,
 \qquad
 \norm{g_{\eta-2\theta}(y)}+\norm{g_{\eta-2\theta}(z)}\le L\,.
\end{equation*}
All $L^p$ norms below average over the randomness generated by the
gadget, including the control gradient $g$. For $p\ge2$, put
\begin{equation*}
 \mathfrak b_p\deq C\,\bigl\{
 \sqrt{\theta p}\,(\varepsilon_0+\theta L)
 +\sqrt{\eta\theta}\,(\sqrt{dp}+p)\bigr\}\,.
\end{equation*}

\begin{proposition}[Hidden RGO likelihood moments]\label{prop:symmetric-moments}
Assume $0<\eta-2\theta\le1$ and $0<\theta\leq(\eta-2\theta)/4$.
For every $p\ge2$, under the proposal law,
\begin{equation}\label{eq:symmetric-proposal-moments}
 \norm W_{L^p}
 \leq\mathfrak b_p\,.
\end{equation}
If $\mathfrak b_2\leq c_0$, retain the original distribution of $g$ and,
conditionally on $g$, tilt the center proposal
$\cN(y-2\theta g,2\theta\Id)$ by $\exp\{-W_g\}$.
This gives the center law $q_y$ in \eqref{eq:hidden-RGO-center-law}.
Keeping the conditional law of the auxiliary draws given $(g,U)$ unchanged,
\begin{equation}\label{eq:symmetric-target-moments}
 \norm W_{L^p(q_y)}
 \leq C\mathfrak b_{2p}\,.
\end{equation}
\end{proposition}

\begin{proof}
Conditionally on $g$,
\[
 M_U\sim\cN(y-2\theta g,\theta\Id)\,,
 \qquad
 \Delta_U\sim\cN(0,\theta\Id)\,,
 \qquad
 \Delta_U\perp(M_U,g)\,.
\]
Write
$\xi_\pm\deq\nabla V(X_\pm)
-g_{\eta-2\theta}(M_U\pm S\Delta_U)$. Then
\begin{align*}
 W
 ={}&2\,\ip{g_{\eta-2\theta}(M_U)-g}{\Delta_U}
 +\mathcal R+\ip{\xi_-+\xi_+}{\Delta_U}\,,\\
 \mathcal R
 \deq {}&\ip{g_{\eta-2\theta}(M_U-S\Delta_U)
 +g_{\eta-2\theta}(M_U+S\Delta_U)-2g_{\eta-2\theta}(M_U)}{\Delta_U}\,.
\end{align*}
For fixed $(M_U,S)$, the last map is odd in $\Delta_U$ and its Gaussian weak gradient is bounded by
$4\,\norm{\Delta_U}$ because $g_{\eta-2\theta}$ is $1$-Lipschitz. Gaussian $L^p$ Poincar\'e and
$\theta\leq\sqrt{\eta\theta}$ therefore give
\[
 \norm{\mathcal R}_{L^p}
 \leq C\sqrt{\eta\theta}\,(\sqrt{dp}+p)\,.
\]
Conditioning on $(M_U,g)$, $1$-Lipschitzness of $g_{\eta-2\theta}$ and $\nabla V$ gives
\begin{align*}
 \norm{\ip{g_{\eta-2\theta}(M_U)-g}{\Delta_U}}_{L^p}
 & \leq C\sqrt{\theta p}\,
 \norm{g_{\eta-2\theta}(M_U)-g}_{L^p} \\
 &\le C\sqrt{\theta p}\,\bigl\{\norm{g_{\eta-2\theta}(M_U) - g_{\eta-2\theta}(z)}_{L^p} + \norm{g_{\eta-2\theta}(z) - g}_{L^p}\bigr\} \\
 &\le C\sqrt{\theta p}\,\bigl\{\norm{M_U - z}_{L^p} + \norm{g_{\eta-2\theta}(z) - g}_{L^p}\bigr\}\,.
\end{align*}
Note that $g_{\eta-2\theta}(z)=\E\nabla V(X_0)=\E g$, since $X_0 \sim R_{\eta-2\theta,z}$. The Poincar\'e inequality and Lipschitz concentration give $\norm{g_{\eta-2\theta}(z)-g}_{L^p} = \norm{\nabla V(X_0) - \E\nabla V(X_0)}_{L^p} \le C\sqrt{\eta-2\theta}\,(\sqrt d + \sqrt p)$.
Thus,
\begin{align*}
 \norm{M_U-z}_{L^p}
 &\leq \norm{y-z}+2\theta\,\norm g_{L^p}
      +C\sqrt\theta\,(\sqrt d+\sqrt p)\\
 &\leq C\,\bigl\{
 \varepsilon_0+\theta L
 +\bigl(\sqrt\theta+\theta\sqrt{\eta-2\theta}\bigr)\,
   (\sqrt d+\sqrt p)\bigr\}\,.
\end{align*}
Combining these estimates gives
\begin{align*}
 \norm{\ip{g_{\eta-2\theta}(M_U)-g}{\Delta_U}}_{L^p}
 \leq \mathfrak b_p\,.
\end{align*}
Finally, conditionally on $(M_U,\Delta_U,S)$, the variables
$\xi_-$ and $\xi_+$ are independent and centered. The function $(x_-,x_+)\mapsto
 \ip{\nabla V(x_-)+\nabla V(x_+)}{\Delta_U}$
is $\sqrt2\,\norm{\Delta_U}$-Lipschitz on the product space. The conditional
law of $(X_-,X_+)$ is the product of
$R_{\eta-2\theta,M_U-S\Delta_U}$ and
$R_{\eta-2\theta,M_U+S\Delta_U}$, so Bakry--\'Emery concentration gives
\begin{equation*}
 \bigl(
 \E[
  \abs{\ip{\xi_-+\xi_+}{\Delta_U}}^p
  \mid M_U,\Delta_U,S
 ]
 \bigr)^{1/p}
 \leq C\sqrt{(\eta-2\theta)p}\,\norm{\Delta_U}\,.
\end{equation*}
Taking the $L^p$ norm over $(M_U,\Delta_U,S)$ therefore yields
\begin{equation*}
 \norm{\ip{\xi_++\xi_-}{\Delta_U}}_{L^p}
 \leq C\sqrt{\eta\theta}\,(\sqrt{dp}+p)\,.
\end{equation*}
These estimates prove \eqref{eq:symmetric-proposal-moments}.

Next, note that, conditionally on $g$, the variables
$U$ and $U'$ are independent with common law
$\cN(y-2\theta g,2\theta\Id)$. Therefore $\E[W_g(U')\mid g]=\E[W_g(U)\mid g]$.
Since $\nabla W_g=g_{\eta-2\theta}-g$, the Gaussian $L^p$ Poincar\'e
inequality, conditionally on $g$, gives
\begin{equation*}
 \bigl(
 \E\bigl[
  \abs{W_g(U)-\E[W_g(U)\mid g]}^p
  \bigm|g
 \bigr]
 \bigr)^{1/p}
 \leq C\sqrt{\theta p}\,
 \bigl(
 \E[
  \norm{g_{\eta-2\theta}(U)-g}^p
  \mid g
 ]
 \bigr)^{1/p}\,.
\end{equation*}
The same estimate as above, but with $U$ in place of $M_U$, yields
\begin{align*}
 \norm{\E[W_g(U')\mid g]-W_g(U)}_{L^p}
 &\leq  \mathfrak b_p\,.
\end{align*}
Since $\mathfrak b_p$ grows at most linearly with $p$, the condition $\mathfrak b_2 \le c_0$ implies $\E\exp(2\,\abs{\E[W_g(U')\mid g] - W_g(U)}) \le C$, provided that $c_0$ is sufficiently small.
Jensen's inequality yields
\begin{equation*}
 \E\bigl[
  \exp\{\E[W_g(U')\mid g]-W_g(U)\}
  \bigm|g
 \bigr]\geq1\,.
\end{equation*}
The Radon--Nikodym derivative of the
tilted joint law relative to the proposal joint law is consequently at
most $\exp\{\E[W_g(U')\mid g]-W_g(U)\}$.
So,
H\"older's inequality gives
\begin{align*}
 \norm W_{L^p(q_y)}
 &\leq
 \bigl(
  \E\bigl[
   |W|^p
   \exp\{|\E[W_g(U')\mid g]-W_g(U)|\}
  \bigr]
 \bigr)^{1/p}\\
 &\leq
 \norm W_{L^{2p}}\,
 \bigl(
  \E\exp\bigl(
   2\,|\E[W_g(U')\mid g]-W_g(U)|
  \bigr)
 \bigr)^{1/(2p)}
 \leq C\mathfrak b_{2p}\,.
\end{align*}
All expectations on the right-hand side are under the proposal law.
This proves \eqref{eq:symmetric-target-moments}.
\end{proof}

\begin{proof}[Proof of \cref{lem:hidden-RGO}]
After decreasing $c$, the condition in \cref{lem:hidden-RGO} ensures
that $\mathfrak b_2\le c_0$ and
$\mathfrak b_{2p}\le c_1 \clipB_{\rm RGO}$ for a sufficiently small universal
constant $c_1>0$.

Condition on the retained control gradient $g$. By
\eqref{eq:hidden-RGO-centered-likelihood}, after averaging over $U'$, the
conditional mean of the untruncated tilt estimator is
$\E[W_g(U')\mid g]-W_g(U)$, and its normalizing constant is at least one.
Apply the clipping comparison of \cref{lem:clipping} conditionally on
$g$ and integrate over its original law. The moment bounds
in \cref{prop:symmetric-moments} then bound
the clipping error by
\begin{equation*}
 C \clipB_{\rm RGO}\,\Bigl(\frac{C\mathfrak b_{2p}}{\clipB_{\rm RGO}}\Bigr)^p
 \le\delta\,.
\end{equation*}
\Cref{gadget:FORS} samples this clipped center tilt, and applying the same final RGO kernel cannot
increase TV distance.
This proves the finite accuracy claim.

By \cref{lem:reusable-FORS}, the expected numbers of center proposals and
tilt estimator calls are at most $C$ and $C\clipB_{\rm RGO}$, respectively.
Each proposal uses one cloud draw, while each tilt estimator uses one cloud
draw, two RGO calls, and two gradient evaluations. Including the initial
control RGO and gradient calls and the final RGO call gives the stated
counts.

For approximate RGO calls, couple each requested draw to an exact one
until the first mismatch. The probability of a mismatch is at most
$\varepsilon$ times the expected number of exact RGO calls, hence at
most $C\varepsilon$. Clipping still gives the same acceptance lower bound on every
implemented history, so the expected call bounds remain valid.
\end{proof}

%% file: 5_high_acc_sections/05_marginal_corrections.tex
\section{Applying the cloud corrections to Picard HMC}\label{sec:marginal}

We now apply \cref{gadget:fluctuation,gadget:mean} to the stochastic gradient
evaluations in Picard HMC. To focus on the main ideas, we assume access
to proximal and Hessian oracles for $V$ and defer the gradient-only
implementation to \cref{sec:gradient-only}.

\subsection{Fluctuation estimator}\label{sec:joint-derivatives}

To apply \cref{gadget:fluctuation}, we first verify the derivative bounds for
the gradient estimator.
The next cloud proposal in \cref{alg:fifth-phase} treats
$(\cloud^{[\ell]},G)$ as one standard Gaussian block after whitening. Since
\begin{equation*}
 \norm{\wh g_\eta(c;g)-\wh g_\eta(c';g')}
 \leq\norm{c-c'}+\sqrt\eta\,\norm{g-g'}\,,
\end{equation*}
the required derivative bounds readily follow.

\begin{lemma}[Jacobian-vector products for the gradient estimator]
\label{lem:joint-estimator-derivatives}
For fixed $y$, whiten $c\deq y+\sqrt\theta\,u$ and put
\[
 \Psi_y(u,g)\deq\wh g_\eta(c;g)\,,\qquad
 B_c\deq D\prox_{\eta V}(c)
 =\bigl(\Id+\eta\,\nabla^2V(\prox_{\eta V}(c))\bigr)^{-1}\,.
\]
At every differentiability point,
\begin{align}
 D\Psi_y(u,g)[r,s]
 &=\nabla^2V\bigl(\prox_{\eta V}(c)+\sqrt\eta\,g\bigr)\,
 \bigl(\sqrt\theta\,B_cr+\sqrt\eta\,s\bigr)\,,
 \label{eq:joint-forward}\\
 D\Psi_y(u,g)^\T w
 &=\bigl(
 \sqrt\theta\,B_c\nabla^2V\bigl(\prox_{\eta V}(c)+\sqrt\eta\,g\bigr)w,
 \sqrt\eta\,\nabla^2V\bigl(\prox_{\eta V}(c)+\sqrt\eta\,g\bigr)w
 \bigr)\,.
 \label{eq:joint-adjoint}
\end{align}
Both actions have operator norm at most $\sqrt{\theta+\eta}$ and use
$O(1)$ exact Hessian and proximal queries. They satisfy
\cref{lem:approximate-derivatives} with zero bias, uniformly in the hidden
iterate $y$.
\end{lemma}

\begin{proof}
The chain rule and symmetry of
$\nabla^2V(\prox_{\eta V}(c)+\sqrt\eta\,g)$ and $B_c$ give
\eqref{eq:joint-forward}--\eqref{eq:joint-adjoint}. Their contraction
bounds and Cauchy--Schwarz give the operator norm bound.
\end{proof}

\paragraph{Fluctuation correction with an unknown endpoint.}
Here, we address a subtle algorithmic point.
The derivative evaluations in \cref{alg:reusable-fluctuation} take place
at correlated Gaussian points. Item~\textup{(i)} of that algorithm gives
a Brownian bridge sampling rule when the standardized endpoint $B_1$
is observed. To apply this to Picard HMC\@, we should take $B_1$ to be the standardized version of $(\cloud_0, G)$, where $\cloud_0$ is a cloud draw and $G$ is the Gaussian in the stochastic gradient estimator.
However,
$\cloud_0\sim\cN(y,\theta\Id)$, with $y$ \emph{hidden}. For $\theta>0$, the
corresponding standardized endpoint would include $(\cloud_0-y)/\sqrt\theta$, which cannot be evaluated.

We can instead sample the cloud points needed for the derivative evaluations
directly. For one evaluation at time $t\in[0,1]$, the desired point is
\[
 \widetilde\cloud_t
 \deq y+\sqrt\theta\,\bigl(B_t+\sqrt{1-t}\,G'\bigr)\,,
\]
where $(B_t)_{0\le t\le1}$ is a standard Brownian motion conditioned on
$B_1=(\cloud_0-y)/\sqrt \theta$, and $G'\sim\cN(0,\Id)$ is independent. For fixed $y$,
\[
 \Law(\widetilde\cloud_t\mid \cloud_0)
 =\cN\bigl(t\cloud_0+(1-t)y,\,\theta(1-t^2)\Id\bigr)\,.
\]
Although the mean contains $y$, an observable realization is
\[
 \cloud_t=t\cloud_0+(1-t)\cloud'+\sqrt{2\theta t(1-t)}\,Z\,,
 \qquad \cloud'\sim\cN(y,\theta\Id)\,,\quad Z\sim\cN(0,\Id)\,,
\]
where $\cloud'$ and $Z$ are independent of each other and of $\cloud_0$.
Indeed, its mean conditional on $\cloud_0$ is $t\cloud_0+(1-t)y$ and its conditional
covariance is $\theta\,\{(1-t)^2+2t\,(1-t)\}\,\Id=\theta\,(1-t^2)\,\Id$.

For several evaluations, the shared Brownian path also creates correlations
that must be preserved. The following construction samples their joint law
using additional cloud draws, without evaluating $y$ or the endpoint $(\cloud_0-y)/\sqrt \theta$.

\begin{lemma}[Hidden mean Gaussian batch]\label{lem:hidden-batch}
Let $k\geq1$, $\theta>0$, and $\cloud_0\sim\cN(y,\theta\Id)$, with $y$ fixed.
Fix the requested evaluation times
$t_1,\ldots,t_k\in[0,1]$ on a common
conditioned Brownian path. The target joint law is that of
\[
 \widetilde\cloud_i
 \deq y+\sqrt\theta\,\bigl(B_{t_i}+\sqrt{1-t_i}\,G_i'\bigr)\,,
 \qquad i\in[k]\,,
\]
conditional on $\cloud_0$, where $B_1=(\cloud_0-y)/\sqrt\theta$ and the
$G_i'\sim\cN(0,\Id)$, $i\in[k]$, are independent of each other and of
the path.
Put
\begin{equation*}
 K_{i,j}\deq 
 \begin{cases}
  1-t_i^2\,,&i=j\,,\\
  t_i \wedge t_j-t_it_j\,,&i\ne j\,,
 \end{cases}
 \qquad i,j\in[k]\,.
\end{equation*}
Draw
$\cloud_{{\rm fluc},1},\ldots,\cloud_{{\rm fluc},k}\sim\cN(y,\theta\Id)$ independently
of each other and of $\cloud_0$. Define their average by
$\overline\cloud_{\rm fluc}\deq k^{-1}\sum_{j\in[k]}\cloud_{{\rm fluc},j}$, and draw $Z\sim\cN(0,(K-(\mathbf 1-t)(\mathbf 1-t)^\T/k)\otimes\Id)$
independently. Then
\begin{equation}\label{eq:hidden-batch-sampler}
    \cloud_i=t_i\,\cloud_0+(1-t_i)\,\overline\cloud_{\rm fluc}+\sqrt\theta\,Z_i\,,
 \qquad i\in[k]\,,
\end{equation}
has the same joint law as $(\widetilde\cloud_i)_{i\in[k]}$ conditional on $\cloud_0$.
\end{lemma}

\begin{proof}
Let $\Sigma_{i,j}\deq t_i \wedge t_j -t_it_j$ for $i,j\in[k]$. Then, $\Sigma$ is the covariance matrix of the Brownian bridge
at the requested times. The independent $G_i'$ contribute
$\operatorname{diag}(\mathbf 1-t)$, so the target conditional covariance is $\theta K\otimes\Id$,
where $K=\Sigma+\operatorname{diag}(\mathbf 1-t)$. Weighted Cauchy--Schwarz gives
\[
    (x^\T (\mathbf 1-t))^2
    \leq k\sum_{i\in[k]} (1-t_i)\,x_i^2\,,
 \qquad x\in\R^k\,,
\]
so $K-(\mathbf 1-t)(\mathbf 1-t)^\T/k\succeq0$. Conditional on $\cloud_0$, the constructed vector in
\eqref{eq:hidden-batch-sampler} has mean
$(t_i\cloud_0+(1-t_i)y)_{i\in[k]}$. The cloud average contributes covariance
$\theta\,((\mathbf 1-t)(\mathbf 1-t)^\T/k)\otimes\Id$, and $Z$ supplies the remainder, giving
$\theta K\otimes\Id$. Joint
Gaussianity completes the proof.
\end{proof}

\paragraph{Batching an entire FORS attempt.}
For one fluctuation FORS attempt, first draw
$M\sim\mathsf{Poisson}(2\clipB_{\rm fluc})$ and all block indices and Brownian
evaluation times for its $M$ tilt estimators. These choices are independent
of the Brownian paths and can be sampled in advance. Collect all $k$
requested cloud evaluation points across the $M$ estimators into one list,
listing the two points of each sampled covariance action separately even
though their times coincide. If $k>0$, use the joint sampling construction
in \cref{lem:hidden-batch} for this entire list; for Picard HMC, use the stacked cloud in $\R^{Jd}$.
For the Gaussian label coordinates, use one common Brownian bridge
conditioned on the observed $G$, independently of the cloud construction,
and fresh independent Gaussians at each evaluation point.
Reuse the same sampled evaluation points when applying a given covariance
map to both $Z$ and $Z'$.

The whole collection then has the joint law prescribed by
\cref{alg:reusable-fluctuation}, including the correlations between
different tilt estimators. In the representation with the common Brownian
path, these estimators are independent conditional on the full proposal,
as required by \cref{alg:reusable-FORS}. The batch uses one fresh cloud per
requested point, so this joint construction preserves the cloud request
count. When $M=0$, no additional clouds are needed. Each new attempt uses
fresh randomness for its proposal and its entire batch.

\subsection{Mean estimator}\label{sec:gaussian-stein}

In our application of \cref{gadget:mean} to Picard HMC\@, we must produce an unbiased estimator of
\begin{align*}
    g_\eta(y) - \E \widehat g_\eta(\cloud; G)
    &= g_\eta(y) - \E \overline g_\eta(\cloud)\,,
\end{align*}
where $\cloud \sim \cN(y, \theta\Id)$ and $G \sim \cN(0, \Id)$ are independent, $y$ is crucially \emph{unknown}, and $\overline g_\eta(\mathsf c) \deq \E \widehat g_\eta(\mathsf c; G)$ for $\mathsf c\in\R^d$.
However, recall that $g_\eta$ itself is a conditional mean: $g_\eta(y) = \E_{X\sim R_{\eta,y}} \nabla V(X)$.
From the definition of $\widehat g_\eta$, we can therefore write this difference as
\begin{align*}
    &\E_{\cloud\sim \cN(y, \theta\Id)}\bigl[\E_{X\sim R_{\eta,y}} \nabla V(X) - \E_{G\sim \cN(0, \Id)}\nabla V(\prox_{\eta V}(\cloud) +\sqrt \eta\,G)\bigr] \\
    &\qquad = \E_{\cloud\sim \cN(y, \theta\Id)}\E_{G\sim \cN(0,\Id)}\bigl[\nabla V(\prox_{\eta V}(\cloud) + \sqrt\eta\, G)\, \{\rho_{\cloud}(G)-1\} \bigr]
\end{align*}
where, conditionally on $\cloud$,
\begin{align*}
    \log \rho_{\cloud}(G)
    ={}&-V\bigl(\prox_{\eta V}(\cloud)+\sqrt\eta\,G\bigr)
    +\sqrt\eta\,\ip{\nabla V(\prox_{\eta V}(\cloud))}{G}
    -\frac{\ip{\cloud-y}{G}}{\sqrt\eta} + \text{constant}\,,
\end{align*}
is the log-density ratio between the laws of $(X-\prox_{\eta V}(\cloud))/\sqrt \eta$ and $G$, and the constant can depend on $\cloud$.

This identity by itself does not furnish a suitable unbiased estimator, because the expression for $\log \rho_{\cloud}(G)$ depends on the unknown $y$, and also includes the log-normalizing constant which cannot be evaluated.
To deal with the second issue, it would be convenient to instead have the relative score $\nabla \log \rho_{\cloud}(G)$.
To make the relative score appear, we apply Gaussian integration by parts by writing the centered version of $\nabla V$ as a Gaussian divergence.

More precisely, for $\mathsf c\in\R^d$, put
\begin{equation*}
 F_{\mathsf c}(u)\deq \nabla V\bigl(\prox_{\eta V}(\mathsf c)+\sqrt\eta\,u\bigr)\,.
\end{equation*}
We seek $A_{\mathsf c}$ such that
$F_{\mathsf c}-\E_\gamma F_{\mathsf c}=\delta_\gamma A_{\mathsf c}$, where $\delta_\gamma$ is the
Gaussian divergence operator acting row-wise on matrix-valued maps:
\begin{equation*}
 (\delta_\gamma A)_i(u)
 \deq\sum_{j\in[d]}\bigl\{u_jA_{i,j}(u)-\partial_jA_{i,j}(u)\bigr\}\,,
 \qquad i\in[d]\,.
\end{equation*}
The solution can be written in terms of the Ornstein--Uhlenbeck (OU) generator
$\mathcal L_{\mathsf{OU}}\deq\Delta-u\cdot\nabla$, or via the corresponding semigroup ${(P_t)}_{t\ge 0}$, as
\begin{equation*}
    A_{\mathsf c}\deq D(-\mathcal L_{\mathsf{OU}})^{-1}(F_{\mathsf c}-\E_\gamma F_{\mathsf c}) =\int_0^\infty \exp(-t)\,P_tDF_{\mathsf c}\,\dd t\,.
\end{equation*}
In the following proposition, we develop an identity using this reasoning, but we perform an additional integration by parts to remove the dependence on the hidden $y$.

\begin{proposition}[Gaussian Stein identity]\label{thm:marginal-cloud}
Let $\cloud\sim\cN(y,\theta\Id)$ and
$X\sim R_{\eta,y}$ be independent, and put
$Z_{\cloud}\deq (X-\prox_{\eta V}(\cloud))/\sqrt\eta$.  Then
\begin{align}
 g_\eta(y)-\E\overline g_\eta(\cloud)
 =-\E\Bigl[\sqrt\eta\,
 A_{\cloud}(Z_{\cloud})\,
 \bigl\{\nabla V(X)-\nabla V(\prox_{\eta V}(\cloud))\bigr\}
 +\frac{\theta}{\sqrt\eta}
 \divg_{\cloud}A_{\cloud}(Z_{\cloud})
 \Bigr]\,,\label{eq:marginal-cloud-identity}
\end{align}
where $(\divg_{\cloud}A)_i\deq \sum_{j\in[d]}\partial_{\cloud_j}A_{i,j}$.  For $C^{1,1}$ potentials, the identity holds in
the weak Sobolev sense.
\end{proposition}

\begin{proof}
    By the preceding argument and Gaussian integration by parts,
    \begin{align*}
        g_\eta(y) - \E\overline g_\eta(\cloud)
        &= \E[A_{\cloud}(G)\,\nabla \rho_{\cloud}(G)]
        = \E[A_{\cloud}(Z_{\cloud})\,\nabla \log \rho_{\cloud}(Z_{\cloud})] \\
        &=-\E\Bigl[
 A_{\cloud}(Z_{\cloud})\,\Bigl\{
 \sqrt\eta\,\bigl\{\nabla V(X)-\nabla V(\prox_{\eta V}(\cloud))\bigr\}
 +\frac{\cloud-y}{\sqrt\eta}
 \Bigr\}\Bigr]\,.
    \end{align*}
    For the last term, we now use Gaussian integration
by parts in $\cloud$, which turns the last term into
$\theta\,\E[\divg_{\cloud}A_{\cloud}(Z_{\cloud})]/\sqrt\eta$.  Mollification and Gaussian moments justify the
weak formulation.
\end{proof}

It remains to produce an unbiased estimator of the right-hand side of~\eqref{eq:marginal-cloud-identity}.
We first explain how to do so making use of exact Hessian and proximal queries.
Note that the semigroup representation shows that $\norm{A_{\mathsf c}(u)}_{\op}\leq\sqrt\eta$ for all $u$.

If
$T\sim\mathsf{Exp}(1)$ and $G\sim\cN(0, I)$ are independent, then
\begin{equation}\label{eq:OU-HVP}
 \widehat A_{\mathsf c}(u,v)
 \deq\sqrt\eta\,\nabla^2V\bigl(
 \prox_{\eta V}(\mathsf c)
 +\sqrt\eta\,\{\exp(-T)\,u+\sqrt{1-\exp(-2T)}\,G\}\bigr)\,v
\end{equation}
satisfies
\begin{equation}\label{eq:action-op-bound}
 \E\widehat A_{\mathsf c}(u,v)=A_{\mathsf c}(u)\,v\,,\qquad
 \norm{\widehat A_{\mathsf c}(u,v)}\leq\sqrt\eta\,\norm v\,.
\end{equation}
For the divergence term, put
\[
 B_{\mathsf c}\deq D\prox_{\eta V}(\mathsf c)
 =\bigl(\Id+\eta\,\nabla^2V(\prox_{\eta V}(\mathsf c))\bigr)^{-1}\,.
\]
This matrix is symmetric with \(0\preceq B_{\mathsf c}\preceq\Id\).  Put
\begin{equation*}
 q(t)\deq \frac{\e^{-t}\,(1-\e^{-t})}{\sqrt{1-\e^{-2t}}}\,,
 \qquad
 W_t\deq \exp(-t)\,x+(1-\exp(-t))\prox_{\eta V}(\mathsf c)
 +\sqrt\eta\sqrt{1-\exp(-2t)}\,G\,.
\end{equation*}
For a smooth mollification of \(V\), differentiating the OU representation
and then integrating by parts in \(G\) give
\begin{align}
 \divg_{\mathsf c}A_{\mathsf c}\Bigl(\frac{x-\prox_{\eta V}(\mathsf c)}{\sqrt\eta}\Bigr)
 &=\sqrt\eta\int_0^\infty \e^{-t}\,(1-\e^{-t})\,
 \E\bigl[
 \divg_w\{\nabla^2V(w)\,B_{\mathsf c}\}\big|_{w=W_t}
 \bigr]\,\dd t\notag\\
 &=\int_0^\infty q(t)\,
 \E[\nabla^2V(W_t)\,B_{\mathsf c}G]\,\dd t\,.
 \label{eq:divergence-formula}
\end{align}
The formula for \(V\in C^2\) follows by letting the mollification scale vanish: the Hessians are uniformly
bounded, the proximal points converge, and dominated convergence applies to the Gaussian
integral.

Therefore, an unbiased estimator for the divergence term can be produced as follows.
Let \(Z_q\deq \int_0^\infty q(t)\,\dd t = \tfrac{\pi}{2} - 1\), draw \(T\) with density \(q/Z_q\), and draw
\(G\sim\cN(0,\Id)\). Then, let
\begin{align*}
    \widehat{\mathcal D}_{\mathsf c}
    &\deq Z_q\,\nabla^2 V(W_T)\,B_{\mathsf c} G\,,
\end{align*}
where in the definition of $W_T$, we take $x = X \sim R_{\eta,y}$.

This estimator makes use of an exact draw from $R_{\eta,y}$, which is unavailable because $y$ is unknown; we address this issue via \cref{gadget:hidden-RGO}, developed in \cref{sec:hidden-RGO-proof}.
It also calls $\prox_{\eta V}$ and inverts the matrix $\Id+\eta\nabla^2 V(\prox_{\eta V}(\mathsf c))$ to obtain $B_{\mathsf c}$.
In \cref{sec:gradient-only}, we develop an alternative estimator which does not require these stronger oracles.

The following proposition records the moment bound for this estimator.

\begin{proposition}[Mean estimator]\label{prop:marginal-score}
Let $0\leq\theta\leq\eta\leq1$, with $\eta>0$. Given independent exact draws
$\cloud\sim\cN(y,\theta\Id)$ and $X\sim R_{\eta,y}$, the preceding estimators
produce a vector $\widehat\Delta(y)$ satisfying
\begin{equation*}
 \E\widehat\Delta(y)
 =g_\eta(y)-\E\overline g_\eta(\cloud)\,,
\end{equation*}
and, for $p\geq2$,
\begin{equation}\label{eq:marginal-score-size}
 \norm{\widehat\Delta(y)}_{L^p}
 \leq C\,\Bigl(\eta^{3/2}+\frac{\theta}{\sqrt\eta}\Bigr)\,
 (\sqrt d+\sqrt p)\,.
\end{equation}
Besides the two supplied draws, the estimator uses $O(1)$ gradient,
proximal, and Hessian queries.
\end{proposition}

\begin{proof}
Use the estimator \eqref{eq:OU-HVP} for the first term in
\eqref{eq:marginal-cloud-identity} and the estimator
$\widehat{\mathcal D}_{\cloud}$ above for the second, with fresh auxiliary seeds.
The preceding discussion and \cref{thm:marginal-cloud} give unbiasedness.
Lipschitzness of $\nabla V$, strong log-concavity of $\cN(y,\theta I)$ and $R_{\eta,y}$, and $\theta\leq\eta$ give
\begin{align*}
    \norm{\nabla V(X)-\nabla V(\prox_{\eta V}(\cloud))}_{L^p}
    &\le \norm{X-\prox_{\eta V}(y)}_{L^p} + \norm{\prox_{\eta V}(y) - \prox_{\eta V}(\cloud)}_{L^p} \\
    &\le C\sqrt \eta\,(\sqrt d + \sqrt p) + \norm{\cloud-y}_{L^p}
    \leq C\sqrt\eta\,(\sqrt d+\sqrt p)\,.
\end{align*}
By~\eqref{eq:action-op-bound}, this gives the first term in
\eqref{eq:marginal-score-size}.
Since $B_{\cloud}$ and $\nabla^2V(W_T)$ are contractions,
$\norm{\widehat{\mathcal D}}\leq Z_q\,\norm G$; the Gaussian moment bound
and the factor $\theta/\sqrt\eta$ give the second term.
\end{proof}

%% file: 5_high_acc_sections/06_recursive_assembly.tex
\section{Recursive assembly, accuracy, and complexity}\label{sec:recursive}

The correction laws, gradient estimator, and hidden mean sampling components
are now available. We assemble them into a cloud
transition, control its adaptive calls and implementation error, and
then propagate accuracy through the smoothing levels.

\subsection{Recursive cloud implementation and adaptive calls}
\label{sec:recursive-cloud}

In \cref{alg:fifth-phase}, drawing cloud samples at the next level and correcting them requires cloud samples at an earlier level, each of which requires another recursive computation.
This subsection describes those computations, bounds their total
query count, and controls their accumulated error.

\paragraph{How to read this subsection.}
We begin by explaining how the correction gadgets are applied to Picard HMC, and then we carefully count the number of recursive calls.
The \textbf{main idea} of this subsection is to track these calls via tree structures which consist of a mandatory predecessor chain $\cloud^{[\ell]} \to \cloud^{[\ell-1]} \to \cloud^{[\ell-2]} \to \cdots \to \cloud^{[0]}$, for which the recursive cloud calls cannot be avoided, together with side branches.
The side branches are only generated when a FORS Poisson count returns a non-zero value, and this occurs with probability at most $O(\clipB)$.
By ensuring that $\clipB$ is sufficiently small, we control the size of the tree and ensure that only polylog recursive cloud calls are necessary.
Formalizing this entails detailed bookkeeping, and the details below were largely written by GPT\@.
We include them for completeness, but they should be skipped on a first read.

\paragraph{One cloud sample generation.}
Fix a phase $n$ and condition on its starting position
$X_{nh}$ and refreshed momentum $P_{nh}^-$. Without loss of generality, let $n = 0$ in the stacked vectors
$X^{[\ell]}$ and $\cloud^{[\ell]}$, and note that we can obtain the same results by translating all times by $nh$, and set $\eta=\eta_n$ and
$\theta=\theta_n>0$. Thus, in this local argument, $X_{t_j}^{[\ell]}$ and
$\cloud_{t_j}^{[\ell]}$ denote the components at physical time $nh+t_j$.
A \emph{level-$\ell$ cloud request} asks for one
fresh sample with target law
\[
 \cN(X^{[\ell]},\theta\Id_{Jd})\,,\qquad
 X^{[\ell]}=(X_{t_j}^{[\ell]})_{j\in[J]}\,,\qquad
 \ell\in\{0,\ldots,K-1\}\,.
\]
A routine that needs a single-node draw from
$\cN(X_{t_j}^{[\ell]},\theta\Id_d)$, such as a mean cloud
$\cloud_{{\rm mean},t_j}^{[\ell]}$ or a cloud draw of the hidden RGO
sampler at node $j$, makes one level-$\ell$ request and uses the
$j$-th block of the returned sample. Under the exact cloud laws,
requests made with fresh seeds are independent conditional on the
fixed phase input $(X_{nh},P_{nh}^-)$, so this realizes the
independent per-node draws of \cref{alg:fifth-phase}.
At level zero, the center
$X^{[0]}=(X_{0}+t_jP_{0}^-)_{j\in[J]}$ is known, so the request is
answered by adding an independent $\cN(0,\theta\Id_{Jd})$ vector.
For $\ell\in\{0,\ldots,K-2\}$, a request at level $\ell+1$ is
answered as follows.
\begin{enumerate}[label=\textup{(\roman*)},leftmargin=2.4em]
\item Request one level-$\ell$ cloud $\cloud^{[\ell]}$. Use it in the noisy
Picard proposal of \cref{alg:fifth-phase}, with fresh gradient labels
and independent output Gaussian noise.
\item Apply the fluctuation and mean corrections of
\cref{cor:block-correction}. Whenever a correction needs another
level-$\ell$ cloud, make a new request using independent randomness.
Complete that recursive request before using its returned sample.
\item Return the accepted proposal. If a correction rejects a
proposal, generate the new samples required by its next attempt.
\end{enumerate}

Recall how these additional requests arise. First-order rejection
sampling (FORS), \cref{alg:reusable-FORS}, uses a \emph{clipping level}
$\clipB\in(0,1]$ to bound the absolute value of each random likelihood
estimate. An \emph{attempt} consists of a proposal, a count
$M\sim\mathsf{Poisson}(2\clipB)$, and $M$ such estimates used in the
acceptance product. Each estimate contributes one \emph{likelihood
factor}. When $M=0$, the acceptance
product is empty and equals one. In particular, no likelihood
estimates or additional samples are needed. We write
$\clipB_{\rm fluc},\clipB_{\rm mean},\clipB_{\rm RGO}\in(0,1]$ for the clipping
levels in the fluctuation, mean, and hidden RGO corrections.
The first two refer to interior clouds; write
$\clipB_{{\rm fluc},{\rm end}}$ and $\clipB_{{\rm mean},{\rm end}}$ for the clipping parameter in the
endpoint correction, which occurs only once per phase.

A mean factor uses a fresh
cloud $\cloud_{\rm mean}^{[\ell]}$ for the estimator in
\cref{prop:marginal-score}, and clouds $\cloud_{\rm RGO}^{[\ell]}$ for
the hidden RGO sampler in \cref{alg:hidden-RGO}. That sampler targets
$R_{\eta,y}$ at an unknown Picard center $y$, using cloud samples
around $y$ and RGO calls at centers which it can compute. We refer to
the latter as \emph{explicit RGO requests}.

\paragraph{Computation tree.}
Represent the original cloud request by a vertex, called the
\emph{root}. Whenever answering a request requires another cloud
sample, create a new vertex for that request and draw an arrow from
the requesting vertex to the new one. The requesting vertex is the
\emph{parent}, and the new vertex is its \emph{child}. Repeating this
rule produces a \emph{call tree}. Its arrows record requests; returned
samples travel in the opposite direction. Cloud levels decrease
along every arrow, and level-$0$ requests have no children. A
vertex together with all the requests below it is a \emph{subtree}.

The first level-$\ell$ request used to propose a level-$(\ell+1)$
sample is its \emph{mandatory predecessor}. Following these first
requests down to level $0$ gives a path, which we call a
\emph{mandatory chain}. All other cloud requests are additional
branches. For counting, replace each maximal mandatory chain by one
vertex and keep the arrows corresponding to additional requests.
This operation is called \emph{contraction}. A vertex of the resulting
\emph{contracted tree} represents the queries along a whole mandatory
chain; its children represent the additional chains it requests.
For example, with no additional requests, sampling a level-two cloud
uses the chain $\cloud^{[2]}\to \cloud^{[1]}\to \cloud^{[0]}$, which becomes one
vertex. A fresh level-$1$ cloud requested by its fluctuation
correction starts a second chain and hence a second contracted vertex.
\Cref{fig:adaptive-cloud-forest} illustrates this construction.

A collection of these trees is a \emph{forest}, with one root for
each original output request. For a complete phase, place
the endpoint correction request at the root and attach its requests
for level-$(K-1)$ clouds. Its mandatory chain contains at most $K$
correction levels, including the endpoint correction. Over $N$ phases
there are $N$ such roots. Attach each explicit RGO request to
the contracted vertex that makes it, without expanding its internal
computation into this forest. We call the forest \emph{adaptive}
because the actual requests and their inputs depend on earlier
random outcomes and acceptance decisions. Different roots need not
be independent.

\begin{figure}[!ht]
\centering
\begin{tikzpicture}[font=\small,
  call/.style={draw=blue!55!black,fill=blue!4,rounded corners=2pt,
    minimum width=16mm,minimum height=8mm,align=center},
  request/.style={draw=orange!70!black,fill=orange!7,rounded corners=2pt,
    minimum height=8mm,align=center},
  edge/.style={-{Stealth[length=1.8mm]},thick,draw=black!65},
  side/.style={edge,dashed,draw=blue!65!black}]
\node[font=\small\bfseries] at (2.2,1.1) {Cloud requests};
\node[call] (top) at (0,0) {$\cloud^{[\ell+1]}$};
\node[call] (pre) at (0,-1.25) {$\cloud^{[\ell]}$};
\node (dots) at (0,-2.25) {$\vdots$};
\node[call] (base) at (0,-3.15) {$\cloud^{[0]}$};
\draw[edge] (top)--(pre);
\draw[edge] (pre)--(dots);
\draw[edge] (dots)--(base);
\node[align=left,font=\footnotesize] at (-1.05,-1.95)
  {mandatory\\chain};
\node[call] (fluc) at (2.2,0) {$\cloud_{\rm fluc}^{[\ell]}$};
\node[call] (mean) at (2.2,-1.25) {$\cloud_{\rm mean}^{[\ell]}$};
\node[call] (hidden) at (4.3,-1.25) {$\cloud_{\rm RGO}^{[\ell]}$};
\draw[side] (top)--(fluc);
\draw[side] (top.south east)--(mean.north west);
\draw[side] (top.north)--(0,.66)--(4.3,.66)--(hidden.north);
\node[align=center,font=\footnotesize] at (3.55,-2.55)
  {Additional calls, only when requested;\\each has its own predecessor chain.};
\draw[edge] (5.25,-.85)--(6.25,-.85);
\node[font=\footnotesize,align=center] at (5.75,-.25)
  {contract\\chains};
\node[font=\small\bfseries] at (8.65,1.1) {One contracted tree};
\node[call] (root) at (8.4,0) {one root};
\node[call] (child1) at (7.1,-1.5) {child};
\node[call] (child2) at (9.2,-1.5) {child};
\node[call] (child3) at (7.1,-3.05) {child};
\draw[side] (root)--(child1);
\draw[side] (root)--(child2);
\draw[side] (child1)--(child3);
\node[request] (rgo) at (10.2,-3.05) {explicit\\RGO calls};
\draw[edge,draw=orange!70!black] (child2)--(rgo);
\end{tikzpicture}
\caption{Arrows point from a sampling request to the calls it needs.
Solid arrows on the left form the mandatory predecessor chain; dashed
arrows are additional recursive cloud requests. After contraction, every
blue vertex represents a whole mandatory chain. Orange boxes record RGO
requests, whose query counts are charged separately. Repeating the construction
for all output requests gives the adaptive forest.}
\label{fig:adaptive-cloud-forest}
\end{figure}
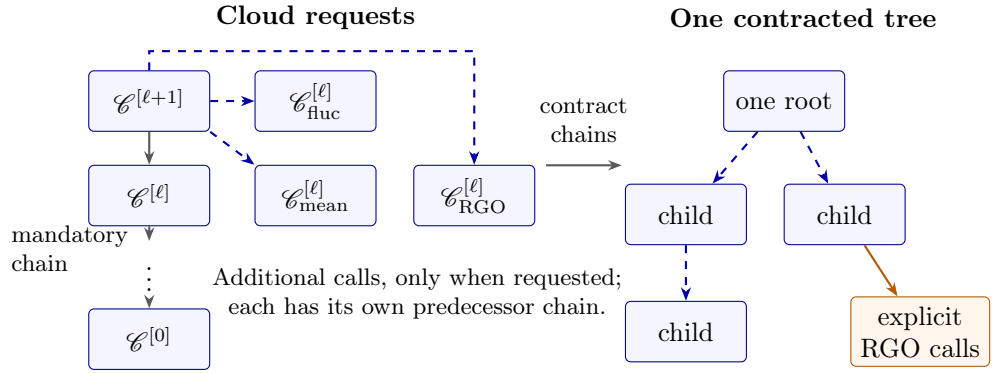

\paragraph{Parameters for recursive accounting.}
Fix $0<\varepsilon\leq1/4$ and put
$\mathfrak L\deq\log(\kappa d/\varepsilon)$, where $\varepsilon$ is the
final TV accuracy. We use $J=\lceil C_J\mathfrak L\rceil$ quadrature nodes and
$K=1+\lceil C_K\mathfrak L/\log\mathfrak L\rceil$ Picard layers.
Set
$K_{\rm ser}=\lceil C_{\rm ser}\mathfrak L\rceil$ for the number of
terms retained in the fluctuation likelihood series of
\cref{alg:reusable-fluctuation}. The constants $C_J,C_K,C_{\rm ser}$ are fixed
universal constants chosen sufficiently large for the accuracy
estimates below.

\paragraph{Query count.}
A \emph{reference query} is an evaluation
of $\nabla V$, an exact Hessian evaluation, or an exact
proximal evaluation.
The \emph{local query count} at a contracted vertex is the number of
reference queries along its mandatory chain, excluding queries made
inside additional cloud subtrees and explicit RGO samplers. Additional cloud samplers
are accounted for at their own vertices; RGO query counts are attached
separately in \cref{lem:RGO-forest}.

For stopping and error accounting, we also record the occurrences of
operations. A \emph{slot} is one place in this record allocated to a
reference query, a cloud or RGO request, a FORS attempt, or the
evaluation of a likelihood factor. An RGO request occupies
one slot, but its sampler can make many reference queries.

To count calls without conditioning on acceptance decisions, we use
an enlarged record. For each FORS invocation, keep drawing Poisson
counts until the first zero count, and list all operations that would
be needed if every earlier attempt were rejected. This list is a
\emph{potential schedule}. It includes every attempt that the actual
sampler can use, because a zero count always accepts. Apply the same
rule to nested FORS invocations and to the additional cloud requests.
The resulting tree contains every actual request and may contain
unused ones; this produces a \emph{dominating}
call tree.

A \emph{cap} is a deterministic upper limit on such a count. The
\emph{local cap} bounds the total potential slots at each contracted
vertex by $Q_{\rm loc}=C\mathfrak L^2$, excluding slots inside its additional
cloud subtrees and its RGO samplers. The probability that this local cap bound fails is bounded in
\eqref{eq:cloud-local-cap} below. A \emph{global cap} bounds the sum
of these slot counts over the forest by a deterministic number
$T_{\max}$. We impose the caps on the potential schedules: if a cap
would be exceeded, stop the entire sampler and return a fixed point
chosen in advance in its output space. This stopping rule also
bounds every executed slot. The term \emph{capped adaptive tree}
refers to the request tree with these local stopping rules. The next
two lemmas count it before the additional global stopping rule is
imposed; \cref{lem:adaptive-tv-comparison} bounds the contribution of stopping
to the TV error.

To count vertices, keep an ordered list of pending requests, initially
containing the roots. Remove the first vertex, reveal its children,
and append those children to the end of the list. This is
\emph{breadth-first exploration}. It is used only for counting.

\begin{lemma}[Size and query count of a capped adaptive tree]\label{lem:adaptive-tree}
Let $q\ge1$ be the number of root requests in the breadth-first
exploration described above. Let $\cF_v$ be the complete exploration history
after $v$ vertices have been processed, and let $A_v$ be the number of
pending vertices. When $A_v>0$, denote by $\zeta_{v+1}$ the number of new
child requests created by the next vertex, including potential
requests that may not be executed. Thus,
\[
 A_0=q\,,\qquad A_{v+1}=A_v-1+\zeta_{v+1}\,.
\]
Suppose that deterministic parameters
$m\ge1$ and $\lambda\ge0$ satisfy, on $\{A_v>0\}$,
\begin{equation}\label{eq:tree-exponential}
\E[\exp(s\zeta_{v+1})\mid\cF_v]
 \leq\exp\bigl\{\lambda\,(\e^{ms}-1)\bigr\}\,,
 \qquad 0\leq s\leq m^{-1}\,,\qquad m\lambda\leq\frac14\,.
\end{equation}
Let
$T\deq\inf\{v\ge0:A_v=0\}$ be the total number of explored vertices
until the pending list is empty, before any global cap is imposed. Then,
\begin{equation}\label{eq:tree-tail}
 \E T\leq\frac{4q}{3}\,,\qquad
 \Pp\{T>C\,(q+mu)\}\leq\exp(-u)\,,\qquad u\geq1\,.
\end{equation}
If the local query count at each vertex is at most $Q_{\rm loc}$, then
the total number of reference queries is at most
$Q_{\rm loc}T$. In particular, it exceeds
$CQ_{\rm loc}\,(q+mu)$ with probability at most $\exp(-u)$.
\end{lemma}

\begin{proof}
Differentiating \eqref{eq:tree-exponential} at zero gives
$\E[\zeta_{v+1}\mid\cF_v]\le m\lambda\le1/4$ on $\{T>v\}$.
After the queue empties, set $A_v\deq 0$ and $\zeta_{v+1}\deq 0$ for $v\ge T$.
For each integer $b\ge1$, sum the queue increments until time
$T\wedge b$. Each processed vertex removes one pending request and
adds $\zeta_{v+1}$ children, so
\[
 A_{T\wedge b}
 =q-(T\wedge b)
   +\sum_{v=0}^{b-1}\mathbf1_{\{T>v\}}\zeta_{v+1}\,.
\]
The event $\{T>v\}$ is
$\cF_v$-measurable: the exploration history determines whether the
queue is still non-empty. Consequently,
\begin{align*}
 \E\sum_{v=0}^{b-1}\mathbf1_{\{T>v\}}\zeta_{v+1}
 &=\sum_{v=0}^{b-1}\E\bigl[
    \mathbf1_{\{T>v\}}\E[\zeta_{v+1}\mid\cF_v]\bigr]
 \leq\frac14\sum_{v=0}^{b-1}\Pp\{T>v\}
 =\frac14\,\E(T\wedge b)\,.
\end{align*}
The last equality follows from
$T\wedge b=\sum_{v=0}^{b-1}\mathbf1_{\{T>v\}}$.
Taking expectations and using
$A_{T\wedge b}\ge0$ therefore yields
\[
 0\leq\E A_{T\wedge b}
 \leq q-\frac34\,\E(T\wedge b)\,.
\]
Thus, $\E(T\wedge b)\leq4q/3$ for every $b$. Monotone convergence as
$b\to\infty$ gives $\E T\leq4q/3<\infty$, which also implies
$T<\infty$ almost surely.

For the tail, take $s=(2m)^{-1}$. Then,
\[
 s-\lambda\,(\e^{ms}-1)
 \geq\frac1m\,\Bigl\{\frac12-\frac{\e^{1/2}-1}{4}\Bigr\}
 \geq\frac1{3m}\,.
\]
Consequently,
\[
\E\bigl[\exp(sA_{v+1})\,\mathbf1_{\{T>v+1\}}\mid\cF_v\bigr]
\leq\exp\bigl(-\frac{1}{3m} + sA_v\bigr)\,\mathbf1_{\{T>v\}}\,.
\]
Iteration gives
$\Pp\{T>v\}\le\exp\{q/(2m)-v/(3m)\}$. This implies
\eqref{eq:tree-tail}.
\end{proof}

The next lemma counts the explicit RGO requests attached to
the vertices of this same forest. A mean correction likelihood factor invokes
the hidden RGO sampler. In
\cref{alg:hidden-RGO}, that sampler first calls an RGO at the supplied
reference center $z$, then uses two RGO draws to evaluate each of its
own likelihood factors, and finally calls an RGO at the accepted
center. These are the requests counted below. Their centers may
depend on all previous cloud and RGO outputs, and a single such
request can require many reference queries.

\begin{lemma}[RGO requests in an adaptive forest]\label{lem:RGO-forest}
Consider the $q$-root forest of \cref{lem:adaptive-tree}, writing
$m_{\rm child}$ for its parameter $m$ and
$\lambda_{\rm child}$ for its parameter $\lambda$. Suppose that the
local query count at each vertex is at most $Q_{\rm loc}$.
Let $M_v$ be the number of potential explicit RGO requests
attached to the $v$-th explored vertex, and set $M_v\deq 0$ after the
pending list is empty. Suppose that parameters $m_{\rm RGO}\ge1$ and
$\lambda_{\rm RGO}\ge0$ satisfy
\begin{equation}\label{eq:RGO-request-mgf}
\E[\exp(sM_v)\mid\cF_{v-1}]
 \leq\exp\bigl\{\lambda_{\rm RGO}\,(\e^{m_{\rm RGO}s}-1)\bigr\}\,,
 \qquad 0\leq s\leq m_{\rm RGO}^{-1}\,.
\end{equation}
At every center where an RGO call is actually made,
suppose that the conditional expected number of reference queries used by the
RGO sampler, given the history immediately before that call, is at
most $Q_{\rm RGO}$.
Then the total number $N_{\rm query}$ of executed reference queries,
including those inside all RGO calls,
and $M_{\rm tot}\deq\sum_{v=1}^T M_v$ satisfy
\begin{align}
 &\E N_{\rm query}\leq\frac{4q}{3}\,
 \{Q_{\rm loc}+Q_{\rm RGO}m_{\rm RGO}\lambda_{\rm RGO}\}\,,
 \label{eq:RGO-request-mean}\\
 &\Pp\bigl\{M_{\rm tot}>Cm_{\rm RGO}\,
       \{\lambda_{\rm RGO}\,(q+m_{\rm child}u)+u\}\bigr\}
       \leq2\exp(-u)\,,\qquad u\geq1\,.\notag
\end{align}
\end{lemma}

\begin{proof}
Differentiation at zero gives
$\E[M_v\mid\cF_{v-1}]
\le m_{\rm RGO}\lambda_{\rm RGO}$ while vertices remain pending.
Since $\{T\ge v\}\in\cF_{v-1}$, summing conditional expectations
gives $\E M_{\rm tot}\le
m_{\rm RGO}\lambda_{\rm RGO}\,\E T$.
There are at most $M_{\rm tot}$ executed RGO calls. Charge each one's
query count when the request is made and use its conditional expectation
bound $Q_{\rm RGO}$. The local query count is at most $Q_{\rm loc}T$, proving
\eqref{eq:RGO-request-mean}.

For the count bound, \cref{lem:adaptive-tree} gives a deterministic
$t_0=\lceil C\,(q+m_{\rm child}u)\rceil$ with
$\Pp\{T>t_0\}\le\exp(-u)$. Iterating
\eqref{eq:RGO-request-mgf} at $s=m_{\rm RGO}^{-1}$ gives
\[
 \E\exp\Bigl\{m_{\rm RGO}^{-1}\sum_{v=1}^{t_0}M_v\Bigr\}
 \leq\exp\{(\e-1)\,\lambda_{\rm RGO}t_0\}\,.
\]
Chernoff's inequality bounds $\sum_{v=1}^{t_0} M_v$ by
$m_{\rm RGO}\,\{(\e-1)\,\lambda_{\rm RGO}t_0+u\}$ outside an event of probability
$\exp(-u)$. A union bound proves the claim.
\end{proof}

\paragraph{Verifying the forest bounds for the cloud sampler.}
Put
\begin{equation*}
 \mathcal B_{\rm fluc}\deq \clipB_{\rm fluc} K+\clipB_{{\rm fluc},{\rm end}}\,,\qquad
 \mathcal B_{\rm mean}\deq \clipB_{\rm mean}K+\clipB_{{\rm mean},{\rm end}}\,.
\end{equation*}
These bound the sums of the clipping levels along one mandatory chain. We will take
\begin{equation}\label{eq:cloud-forest-caps}
 \begin{gathered}
 Q_{\rm loc}=m_{\rm child}=m_{\rm RGO}=C\mathfrak L^2\,,\\
 \lambda_{\rm child}
 =\frac{C}{m_{\rm child}}\,
 (\mathcal B_{\rm fluc} K_{\rm ser}+\mathcal B_{\rm mean} J)\,,\qquad
 \lambda_{\rm RGO}
 =\frac{CJ}{m_{\rm RGO}}\,\mathcal B_{\rm mean}\,.
 \end{gathered}
\end{equation}

First consider the potential schedule, defined above, for one FORS
invocation with clipping level $\clipB$. Recall that it continues until
the first zero Poisson count, even if the actual sampler accepts
earlier. Let $(M_a)_{a\geq1}$ be its independent
$\mathsf{Poisson}(2\clipB)$ counts, and define
\[
 A_{\clipB}\deq\inf\{a\geq1:M_a=0\}\,,\qquad
 N_{\rm FORS}\deq A_{\clipB}+\sum_{a=1}^{A_{\clipB}}M_a\,.
\]
Thus, $N_{\rm FORS}$ counts the number of attempts and the sum of the counts.
Put $p_0\deq\Pp\{M_1=0\}=\exp(-2\clipB)$. On the event $\{A_{\clipB}=a\}$,
the first $a-1$ counts are positive and the last count is zero.
Consequently, independence and a geometric series summation give
\begin{align*}
 \E\exp(tN_{\rm FORS})
 &=\sum_{a\geq1}p_0 \exp(ta)\,
 \bigl(\E[\exp(tM_1)\,\mathbf1_{\{M_1>0\}}]\bigr)^{a-1}
 =\frac{\e^t\, p_0}
 {1-\e^t\,\E[\exp(tM_1)\,\mathbf1_{\{M_1>0\}}]}\,.
\end{align*}
The Poisson MGF yields $\E[\exp(tM_1)\,\mathbf1_{\{M_1>0\}}]
 =\exp(2\clipB\,(\e^t-1))-p_0$,
and hence
\[
 \E\exp(tN_{\rm FORS})
 =\frac{\e^t\, p_0}
 {1-\e^t\,\{\exp(2\clipB\,(\e^t-1))-p_0\}}
 \leq C\,,\qquad 0\le\clipB\le1\,,
\]
for all sufficiently
small universal $t>0$.

We will also use the expected number of extra attempts and likelihood
factors. Since $A_{\clipB}$ is geometric with success probability $p_0$,
\[
 \E(A_{\clipB}-1)=\e^{2\clipB}-1\leq C\clipB\,,\qquad
 \E\sum_{a=1}^{A_{\clipB}}M_a=2\clipB\,\E A_{\clipB}=2\clipB\e^{2\clipB}\leq C\clipB\,.
\]
For the second identity, write the sum as
$\sum_{a\geq1}\mathbf1_{\{A_{\clipB}\geq a\}}M_a$ and condition on the
counts before attempt $a$; they determine $\{A_{\clipB}\geq a\}$ and are
independent of $M_a$. These bounds hold conditionally on the history
before an invocation, since its Poisson counts are fresh.

For one fluctuation invocation, let $A_{\rm fluc}$ be the number of
attempts in its potential schedule and let $M_{\rm fluc}$ be the sum
of their Poisson counts, that is, the total number of likelihood
factors. A proposal evaluates $J$ gradient blocks, and a likelihood
factor makes $O(K_{\rm ser})$ derivative actions. Its local query count and
request count are therefore at most
$C\,\{A_{\rm fluc} J +K_{\rm ser}M_{\rm fluc}\}$.
A hidden RGO invocation has at most $CN_{\rm FORS}$ local queries and
cloud or RGO request slots.
Each mean factor uses one more fluctuation output and at most $J$
component estimators and hidden RGO invocations.

We now pass from these individual schedules to one complete cloud
correction. Let $N_{\rm fluc}$ be the local slot count of one
fluctuation invocation. Since
$A_{\rm fluc}+M_{\rm fluc}=N_{\rm FORS}$ for its potential schedule,
the preceding bounds give $N_{\rm fluc}
 \leq C\,(J+K_{\rm ser})\,N_{\rm FORS}$.
Consequently, for a universal constant $c>0$,
\[
 \E\bigl[
  \exp\bigl(\frac{ctN_{\rm fluc}}{J+K_{\rm ser}}\bigr)
  \,\bigm|\,\text{past}\bigr]
 \leq\exp(Ct)\,,\qquad 0\leq t\leq1\,.
\]
Here and below, \emph{past} denotes all randomness revealed before the
invocation. The same bound holds for the local slot count of one
hidden RGO invocation because that count is at most
$CN_{\rm FORS}$.

Next consider one mean likelihood factor. Apart from one fluctuation
output, it uses at most $J$ component routines, each consisting of a
bounded number of local operations and one hidden RGO invocation.
If their local slot counts are $Y_1,\ldots,Y_J$, conditional
H\"older's inequality gives
\[
 \E\bigl[
  \exp\bigl(\frac{ct}{J}\sum_{j\in[J]}Y_j\bigr)
  \,\bigm|\,\text{past}\bigr]
 \leq
 \prod_{j\in[J]}
 \E[\exp(ctY_j)\mid\text{past}]^{1/J}
 \leq\exp(Ct)\,.
\]
Combining this estimate with the
fluctuation output estimate shows that the local slot count
$N_{\rm mean}$ of one mean factor satisfies
\[
 \E\bigl[
  \exp\bigl(\frac{ctN_{\rm mean}}{J+K_{\rm ser}}\bigr)
  \,\bigm|\,\text{past}\bigr]
 \leq\exp(Ct)\,,\qquad 0\leq t\leq1\,.
\]

Finally, condition on the potential schedule of the outer mean FORS
invocation. It contains $N_{\rm FORS}$ attempts and likelihood
factors. Each corresponding subcall is either a fluctuation output or
a mean factor, so successive conditioning on the subcalls bounds the
conditional exponential moment of the complete local slot count by
$\exp(CtN_{\rm FORS})$. Choose $t>0$ sufficiently small and average
over the outer schedule. The exponential moment estimate for
$N_{\rm FORS}$ then gives a universal bound.

Index the possible corrections along one mandatory chain by
$\ell\in\{0,\ldots,K-1\}$, with $\ell=K-1$ reserved for the endpoint
correction. Let $N_\ell$ be the total local slot count of the
level-$\ell$ correction, and set $N_\ell=0$ when a shorter chain does
not contain that correction. Each counter includes reference queries, cloud and
RGO request slots, and FORS attempts and factors; it excludes queries
inside the requested cloud and RGO samplers.
The preceding argument gives, conditional on the history before
this correction,
\[
 \E\exp\Bigl(\frac{cN_\ell}{J+K_{\rm ser}}\Bigr)\le C\,,
 \qquad \ell\in\{0,\ldots,K-1\}\,.
\]
Define the total counter for this contracted vertex by
$N_{\rm loc}\deq\sum_{\ell=0}^{K-1}N_\ell$. Iterating the conditional bound
and applying Chernoff's inequality gives
\begin{equation}\label{eq:cloud-local-cap}
 \Pp\bigl\{N_{\rm loc}>C\,(J+K_{\rm ser})\,(K+u)\bigr\}
 \leq\exp(-u)\,,\qquad u\ge1\,.
\end{equation}
Since $J+K_{\rm ser}\le C\mathfrak L$ and $K\le C\mathfrak L$,
taking $u=C_{\rm cap}\mathfrak L$, with a sufficiently large fixed
universal constant $C_{\rm cap}$, justifies the local cap
$Q_{\rm loc}=C\mathfrak L^2$ in \eqref{eq:cloud-forest-caps}. Both the number
of additional cloud requests and the number of attached RGO requests
are bounded by this same local counter.

To bound the expected number of children, count cloud requests
separately from local queries. A fluctuation proposal uses one joint
cloud sample to evaluate all $J$ blocks. A fluctuation attempt with
$M$ likelihood factors uses at most $CMK_{\rm ser}$ additional clouds in
one joint batch, by the block derivative construction and
\cref{lem:hidden-batch}. Thus, after
reserving the first proposal cloud for the mandatory chain, one
fluctuation invocation makes at most $A_{\rm fluc}-1+CK_{\rm ser}M_{\rm fluc}$
additional cloud requests. Their conditional expectation is at most
$CK_{\rm ser}\clipB_{\rm fluc}$ by the preceding FORS identities.

Let $A_{\rm mean}$ and $M_{\rm mean}$ be the attempt and likelihood
factor counts in the outer mean potential schedule. It calls the
fluctuation routine $A_{\rm mean}+M_{\rm mean}$ times: once for each
proposal and once for the independent output in each mean factor.
Exactly one of their first proposal clouds belongs to the mandatory
chain. The conditional expected number of additional clouds from
these fluctuation invocations is therefore at most
\begin{align*}
 &\E[A_{\rm mean}+M_{\rm mean}-1\mid\text{past}]
    +C\clipB_{\rm fluc} K_{\rm ser}\,
        \E[A_{\rm mean}+M_{\rm mean}\mid\text{past}]
        \leq C\,(\clipB_{\rm mean}+ \clipB_{\rm fluc} K_{\rm ser})\,.
\end{align*}
Each mean factor also invokes at most $J$ component routines and
hidden RGO samplers. Each of these uses a bounded expected number
of clouds and explicit RGO requests, conditionally on its incoming
history, because $\clipB_{\rm RGO}\leq1$. These calls contribute at most
$C\clipB_{\rm mean} J$ in expectation. Consequently, one correction level
creates at most $C\,(\clipB_{\rm fluc} K_{\rm ser}+\clipB_{\rm mean} J)$
additional cloud requests and $C\clipB_{\rm mean} J$ explicit RGO requests
in conditional expectation. Imposing the local cap can only decrease
these counts.

Use these potential schedules to explore the dominating forest,
before generating the Gaussian cloud samples. The
actual sampler evaluates only the slots it needs, and unused slots
do not contribute to the executed count. Summing the conditional
expectation bounds over the interior levels and the single endpoint gives
\[
 \E[\zeta_{v+1}\mid\cF_v]
 \leq C\,(\mathcal B_{\rm fluc} K_{\rm ser}+\mathcal B_{\rm mean} J)\,,\qquad
 \E[M_v\mid\cF_{v-1}]\leq C\mathcal B_{\rm mean} J\,.
\]
To pass to exponential moments, use the deterministic local cap and
the convexity inequality
\[
 \exp(sx)\leq1+\frac{x}{m}\,(\e^{ms}-1)\,,
 \qquad 0\leq x\leq m\,,\quad s\geq0\,.
\]
Applying it to $\zeta_{v+1}\leq m_{\rm child}$ and then taking
conditional expectations gives
\[
 \E[\exp(s\zeta_{v+1})\mid\cF_v]
 \leq1+\frac{\E[\zeta_{v+1}\mid\cF_v]}{m_{\rm child}}\,
              (\e^{m_{\rm child}s}-1)
 \leq\exp\{\lambda_{\rm child}\,(\e^{m_{\rm child}s}-1)\}\,.
\]
The same argument for $M_v\leq m_{\rm RGO}$ proves
\eqref{eq:RGO-request-mgf}, with the parameters in
\eqref{eq:cloud-forest-caps}.

The branching condition $m_{\rm child}\lambda_{\rm child}\leq1/4$
therefore follows whenever
\[
    \mathcal B_{\rm fluc} K_{\rm ser}+\mathcal B_{\rm mean} J\leq c\,.
\]
Since $K\leq C\mathfrak L$,
$K_{\rm ser}\leq C\mathfrak L$, and $J\leq C\mathfrak L$,
a sufficient bound is
\begin{equation}\label{eq:cloud-branching-smallness}
 \mathfrak L^2\,(\clipB_{\rm fluc}+\clipB_{\rm mean})
 +\mathfrak L\,(\clipB_{{\rm fluc},{\rm end}}+\clipB_{{\rm mean},{\rm end}})\le c\,.
\end{equation}
The parameter choices in \cref{sec:work} enforce
\eqref{eq:cloud-branching-smallness}. The conditional expected number
of children is then at most $1/4$, as required by
\cref{lem:adaptive-tree}.

\paragraph{Global caps.}
Consider the forest of $N$ phase roots with the local stopping rules
above, and assume \eqref{eq:cloud-branching-smallness}. Applying
\cref{lem:adaptive-tree} with child cap $C\mathfrak L^2$ and
$u=C_{\rm cap}\mathfrak L$ bounds its number of contracted vertices by
$C\,(N+\mathfrak L^3)$ outside probability
$\exp(-C_{\rm cap}\mathfrak L)$. Here $C_{\rm cap}$ is a fixed universal
constant that can be chosen as large as needed. Moreover,
$m_{\rm RGO}\lambda_{\rm RGO}=C\mathcal B_{\rm mean}J\leq C$,
so \cref{lem:RGO-forest} gives the same bound for the number of explicit
RGO requests outside an additional probability
$2\exp(-C_{\rm cap}\mathfrak L)$. Each vertex has at most
$C\mathfrak L^2$ local slots. Thus a deterministic global slot cap is
\begin{equation}\label{eq:cloud-global-cap}
 T_{\max}\deq\bigl\lceil C\,(N+\mathfrak L^3)\,\mathfrak L^2\bigr\rceil\,,
\end{equation}
with probability of exceeding it at most
$3\exp(-C_{\rm cap}\mathfrak L)$. Enlarging $C$ also accommodates
the single terminal RGO call and $O(N)$ additional Gaussian
draws. The expected numbers of vertices and RGO requests remain
$O(N)$ by \cref{lem:adaptive-tree,lem:RGO-forest}.

\paragraph{Accuracy of the finite recursion.}
The likelihood series cutoff and the local stopping event in
\eqref{eq:cloud-local-cap} have error exponentially small in
$\mathfrak L$ after increasing their fixed constants. The required
local accuracy logarithms are $O(\mathfrak L)$: a depth-$K$ forest
with at most $C\mathfrak L^2$ children per vertex has at most
$N\,(C\mathfrak L^2)^{K+1}$ potential correction and RGO request slots,
whose logarithm is
$O(\mathfrak L)$ for the displayed $K$ and the phase count below.
At any such slot, compare the exact conditional law of the requested
sample with the law returned by its finite routine, at the same
history of earlier requests and outputs. The TV distance between
these laws is that slot's \emph{replacement error}. We sum these
errors by changing one conditional sampling rule at a time. The following lemma
justifies this comparison when later requests depend on earlier outputs and also
includes the probability of stopping at a cap.

\begin{lemma}[TV error bound]\label{lem:adaptive-tv-comparison}
For an ideal adaptive algorithm, let $K_t(H,\cdot)$ be the conditional
law of the output at call slot $t$ given the preceding history $H$.
Let $\widehat K_t(H,\cdot)$ be the corresponding conditional law for
the implemented algorithm.
Stop both algorithms at the same deterministic cap $T_{\max}$, and
suppose the ideal algorithm without this global cap exceeds it with
probability at most $\varepsilon_{\rm cap}$. After either algorithm
terminates, fill its remaining slots with a dummy symbol. Let $H_{t-1}$ be the ideal capped history
before slot $t$, and define the mean replacement error by
\[
 \delta_t\deq\E
 \TV\bigl(K_t(H_{t-1},\cdot),\widehat K_t(H_{t-1},\cdot)\bigr)
\,,\qquad t\in\{1,\ldots,T_{\max}\}\,.
\]
For the implemented capped output $\widehat X$ and ideal uncapped
output $X$,
\[
 \TV(\Law(\widehat X),\Law(X))
 \leq\sum_{t=1}^{T_{\max}}\delta_t+\varepsilon_{\rm cap}\,.
\]
In particular, the bound is $T_{\max}\varepsilon+\varepsilon_{\rm cap}$
if every conditional kernel is uniformly within $\varepsilon$ in TV.
\end{lemma}

\begin{proof}
For $t\in\{0,\ldots,T_{\max}\}$, let the $t$-th intermediate algorithm use
ideal kernels through slot $t$ and implemented kernels thereafter.
Consecutive intermediate algorithms have the same ideal history through
slot $t-1$; data processing bounds the TV distance between their outputs
by $\delta_t$. Sum over the
slots, then compare the capped and uncapped ideal algorithms.
\end{proof}

\FloatBarrier

\subsection{Analysis of the Gaussian cloud sampler}\label{sec:work}

In this subsection, we establish the following theorem.

\begin{theorem}[Gaussian cloud sampler]\label{thm:core}
Assume \eqref{eq:HVP-regularity} with $1\le\kappa\le\kappa_0$ for a fixed
universal constant $\kappa_0\ge3$, and that the supplied reference point satisfies
\eqref{eq:reference-point}. For every $0<\varepsilon\leq1/4$,
\cref{alg:fifth-phase}, with all RGO calls implemented by
\cref{prop:gradient-RGO} and all local corrections realized as in
\cref{sec:gradient-only}, returns a law $\widehat\pi=\Law(X_{\rm out})$ such that
\begin{equation}\label{eq:fifth-core-TV}
 \TV(\widehat\pi,\pi)\leq\varepsilon\,.
\end{equation}
Its expected total number of gradient queries is at most
\begin{equation}\label{eq:fifth-core-complexity}
    C\mathfrak L^{11/2}\,(d+\mathfrak L)^{1/5}\,.
\end{equation}
\end{theorem}

The constants below may depend on the fixed bound $\kappa_0$.
We choose the parameters
\begin{equation*}
 h=c\mathfrak L^{-3/2}\,(d+\mathfrak L)^{-1/5}\,,
 \qquad
 \eta_0=C_\eta\mathfrak L h^2\,,
 \qquad
 \tau=c_\tau h^3\,,
\end{equation*}
and, at smoothing level $\eta_n\deq\eta_0+n\tau$, set
\begin{equation}\label{eq:fifth-cloud-scale}
 \theta_n=C_{\rm cl}\mathfrak L^{9/2}\eta_nh^4\sqrt{d+\mathfrak L}\,.
\end{equation}
The universal multiplicative constants $C_\eta,C_{\rm cl}$ are chosen
sufficiently large;
the universal constants $c$, $c_\tau$ are then chosen sufficiently small,
with $c_\tau$ chosen before $c$.
Choose
\begin{equation}\label{eq:fifth-phase-count}
 \begin{gathered}
 J=\lceil C_J\mathfrak L\rceil\,,\qquad
 K=1+\Bigl\lceil\frac{C_K\mathfrak L}{\log\mathfrak L}\Bigr\rceil\,,\qquad
 N=\lceil C_N\mathfrak L/h\rceil\,.
 \end{gathered}
\end{equation}
Choose $C_N$ sufficiently large for the contraction estimate below,
and then $c_\tau$ sufficiently small that
\begin{equation*}
 \frac{N\tau}{\eta_0}
 \le\frac{c_\tau}{C_\eta}\,
 \Bigl(C_N+\frac{h}{\mathfrak L}\Bigr)
 \le\frac14\,,\qquad
 \eta_n\in[\eta_0,5\eta_0/4]\quad(0\le n\le N)\,.
\end{equation*}
Since $\eta_0\le C_\eta c^2\mathfrak L^{-2}$ and
$\theta_n/\eta_n\le C_{\rm cl}c^4\mathfrak L^{-3/2}$, these choices give
$4\eta_0\le1$ and
\begin{equation*}
 2\eta_n\leq1\,,
 \qquad
 \theta_n\leq\eta_n/8\,,
 \qquad
 h\mathfrak L^{3/2}\leq Cc/\kappa\,.
\end{equation*}
In particular, $\eta_n-2\theta_n$ belongs to $(0,1]$ and
$\theta_n\leq (\eta_n-2\theta_n)/4$, and the terminal RGO variance
$2\eta_{\rm f}$ is at most one.

\subsubsection{Contraction and preliminary estimates}

Let $H_\tau$ add an independent $\cN(0,\tau\Id)$ variable to position and leave momentum unchanged.
Then $\Pi_\eta H_\tau=\Pi_{\eta+\tau}$,
and $H_\tau$ is non-expansive in $W_{2,M_\kappa}$. For a phase $n$,
let $K_n$ denote the ideal depth-$K$ numerical phase at smoothing
$\eta_n$, before adding position noise. The following lemma supplies the
contraction and local error estimates.

\begin{lemma}[Wasserstein estimates for the ideal depth-$K$ phase]
\label{lem:fifth-numerical-contraction}
Uniformly over all phases $n$,
\begin{align*}
 W_{2,M_\kappa}(\delta_zK_n,\delta_{z'}K_n)
 &\leq(1-ch/\kappa)\,\norm{z-z'}_{M_\kappa}\,,
 \\
 W_{2,M_\kappa}(\Pi_{\eta_n}K_n,\Pi_{\eta_n})
 &\leq\exp(-C\mathfrak L)\,.
\end{align*}
\end{lemma}

\begin{proof}
The contraction holds for every Wasserstein order by
\cref{prop:reference-Picard}. For the second bound, we only need
$p=2$, for which $B_{J,2}\le C_0\,(C\sqrt J)^J$. Since $J\asymp\mathfrak L$,
\begin{equation*}
 B_{J,2}\le C_0\,(C\sqrt{\mathfrak L})^J\,,\qquad
 \frac h{\sqrt{\eta_n}}\le C_\eta^{-1/2}\mathfrak L^{-1/2}\,.
\end{equation*}
For fixed $C_J$, choosing $C_\eta$ sufficiently large gives
$B_{J,2}\,(h/\sqrt{\eta_n})^J\le C_0\exp(-J)$.
Also, $h\le\mathfrak L^{-3/2}$ gives
$h^{2K+1}\le\exp\{-3C_K\mathfrak L\}$.
Choose $C_J,C_K$ sufficiently large and then $C_\eta$ as above.
Equation \eqref{eq:reference-defect} absorbs the factor $\sqrt{d+2}$
and proves the claimed accuracy.
\end{proof}

Consequently,
\begin{equation}\label{eq:fifth-moving-recursion}
 W_{2,M_\kappa}
 \bigl(\mu K_nH_{\tau},\Pi_{\eta_n+\tau}\bigr)
 \leq
 (1-ch/\kappa)\,W_{2,M_\kappa}(\mu,\Pi_{\eta_n})
 +\exp(-C\mathfrak L)\,.
\end{equation}

Let $\mu_0\deq \delta_{x_{\rm ref}}\otimes\cN(0,\Id)$ and
$\mu_{n+1}\deq \mu_nK_nH_\tau$ for $0\le n<N$, so that $\mu_n$
is the ideal phase-space law at smoothing level $\eta_n=\eta_0+n\tau$.
Writing $\rho\deq1-ch/\kappa$, iteration of
\eqref{eq:fifth-moving-recursion} gives
\begin{equation}\label{eq:fifth-global-W2}
 W_{2,M_\kappa}(\mu_n,\Pi_{\eta_n})
 \le\rho^n\, W_{2,M_\kappa}(\mu_0,\Pi_{\eta_0})
      +\frac{\exp(-C\mathfrak L)}{1-\rho}\,,
 \qquad 0\le n\le N\,.
\end{equation}

We also need concentration and moment estimates.

\begin{lemma}[Concentration along the ideal run]
\label{lem:fifth-run-concentration}
Let $\mu_n$ be the ideal phase-space law after $n$ phases,
for $0\le n\le N$. Then,
\begin{equation*}
 \log\E_{\mu_n}\exp\{t\,(f-\E_{\mu_n}f)\}
 \le C\kappa t^2\,,\qquad t\in\R\,,
\end{equation*}
for every $1$-Lipschitz function $f$ in the metric $M_\kappa$.
\end{lemma}
\begin{proof}
We use the following elementary propagation rule. Suppose that a law
$\mu$ has centered log-MGF at most $SL^2t^2$ for every $L$-Lipschitz
function, and that a kernel $K$ satisfies
\begin{equation*}
 \Lip(Kf)\le\alpha\Lip(f)
 \quad\text{and}\quad
 \log\E_{K(x,\cdot)}
 \exp\{t\,(f-Kf(x))\}\le vL^2t^2
\end{equation*}
uniformly in $x$ for every $L$-Lipschitz $f$. Then $\mu K$ has centered
log-MGF at most $(v+\alpha^2S)L^2t^2$. This follows immediately by
conditioning on the input to $K$.

We first record the conditional concentration supplied by one phase.
Fix its input $(X,P)$, and couple two values of the first refresh by
changing only $P^-$. If
\begin{equation*}
 q_\ell\deq
 \sup_{P^-\ne\widetilde P^-}
 \max_{i\in[J]}
 \frac{\norm{X_{t_i}^{[\ell]}(P^-)
             -X_{t_i}^{[\ell]}(\widetilde P^-)}}
      {\norm{P^--\widetilde P^-}}\,,
\end{equation*}
then $q_0\le h$. Since every $g_{\eta_n}$ is $1$-Lipschitz, the Picard
recursion and the estimate
\eqref{eq:integrated-weights-l1}
give
\begin{equation*}
 q_{\ell+1}\le h+h^2q_\ell\,,
 \qquad
 \sup_{\ell\ge0}q_\ell\le\frac{h}{1-h^2}\le\frac43\,h\,.
\end{equation*}
At the endpoint, the same
estimate, together with $\sum_{j\in[J]}\omega_j=h$, yields
\begin{align*}
 \norm{\widetilde X_{(n+1)h}^{[K]}(P^-)
       -\widetilde X_{(n+1)h}^{[K]}(\widetilde P^-)}
 &\le (h+h^2q_{K-1})\,\norm{P^--\widetilde P^-}\,,\\
 \norm{\widetilde P_{(n+1)h}^{[K]}(P^-)
       -\widetilde P_{(n+1)h}^{[K]}(\widetilde P^-)}
 &\le (1+hq_{K-1})\,\norm{P^--\widetilde P^-}\,.
\end{align*}
Thus the endpoint is $C$-Lipschitz in $P^-$, uniformly in $J$ and in
the Picard depth. Since
$P^-=a_hP+\sigma_h\zeta_0$ and
$\sigma_h^2=1-\e^{-h}\le h$, its dependence on the first standard
Gaussian $\zeta_0$ has Lipschitz constant at most $C\sqrt h$ in the
metric $M_\kappa$. The second refresh contributes
$\sigma_h\zeta_1$, and the position noise contributes
$\sqrt\tau\, G_x$. Their Lipschitz constants are at most $\sqrt h$ and
$C\sqrt\tau\le C\sqrt h$, respectively. Consequently, conditional on
$(X,P)$, the entire phase is a $C\sqrt h$-Lipschitz function of the
independent standard Gaussian vector $(\zeta_0,\zeta_1,G_x)$.
Gaussian concentration therefore gives the conditional bound
\begin{equation*}
 \log\E\bigl[\exp\{t\,(F-\E[F\mid X,P])\}\mid X,P\bigr]
 \le Ch t^2
\end{equation*}
whenever $F$ is a $1$-Lipschitz function of the phase output. The
constants above are uniform in the Picard depth.

Thus the propagation rule applies with $(\alpha,v)=(\rho,Ch)$,
where $\rho=1-ch/\kappa$ by \cref{prop:reference-Picard}.
The initial
law, with fixed position and standard Gaussian momentum, has constant
$C$. Iterating the rule gives
\begin{align*}
 \log\E_{\mu_n}\exp\{t\,(f-\E_{\mu_n}f)\}
 &\le Ct^2\,\Bigl\{1+h\sum_{r\ge0}\rho^{2r}\Bigr\}
 \le C\kappa t^2\,.
\end{align*}
Since $n$ was
arbitrary, the estimate holds uniformly.
\end{proof}

Below, we condition on a single phase $n$ and suppress
the time offset $nh$: thus $X_{t_j}^{[\ell]}$ denotes
$X_{nh+t_j}^{[\ell]}$, and $P^-$ denotes $P_{nh}^-$. The stacked vectors
$X^{[\ell]}$ and $\cloud^{[\ell]}$ likewise refer to this fixed phase.

\begin{lemma}[Moment estimates]\label{lem:fifth-stage-moments}
For $2\leq p\leq C\mathfrak L$, uniformly over all phases, Picard depths, and nodes of the ideal
sampler,
\begin{align}
    \norm{g_{\eta_n}(X_{t_j}^{[\ell]})}_{L^p}
 +\norm{g_{\eta_n-2\theta_n}(X_{t_j}^{[\ell]})}_{L^p}
 &\leq C\sqrt{\kappa\, (d+\mathfrak L)}\,,\label{eq:fifth-stage-score}\\
 \norm{X_{t_j}^{[\ell]}-X_{t_j}^{[0]}}_{L^p}
 &\leq Ch^2\sqrt{\kappa\, (d+\mathfrak L)}\,.
 \label{eq:fifth-stage-displacement}
\end{align}
\end{lemma}

\begin{proof}
Initialization and \eqref{eq:reference-point} give
$W_{2,M_\kappa}(\mu_0,\Pi_{\eta_0})\le C\sqrt{\kappa d}$.
Increasing the constants if necessary, we may assume that
$(\kappa/h)\exp(-C\mathfrak L)\le1$. Hence
\eqref{eq:fifth-global-W2} gives, uniformly in $n$,
$W_{2,M_\kappa}(\mu_n,\Pi_{\eta_n})\le C\sqrt{\kappa d}$.

Under a stationary phase input, \eqref{eq:score-moment},
$1$-Lipschitzness, and $X_{t_j}^{[0]}=X+t_jP^-$ bound the $L^2$ norm
of $g_{\eta_n}(X_{t_j}^{[0]})$ by $C\sqrt d$.
The coupling $X=Y+\sqrt{2\theta_n}\,G$, with
$Y\sim\pi_{\eta_n-2\theta_n}$, gives the same bound for
$g_{\eta_n-2\theta_n}(X_{t_j}^{[0]})$. The Picard bound
\eqref{eq:integrated-weights-l1} propagates these estimates through the
stages. It also makes each stage a $C$-Lipschitz function of its phase
input, uniformly in $J,K$. Coupling with the preceding Wasserstein
bound thus bounds both score norms in $L^2$ by $C\sqrt{\kappa d}$.

By \cref{lem:fifth-run-concentration}, and the independent first
momentum refresh, the norm of either stage score has centered $L^p$
norm at most $C\sqrt{\kappa p}$. Adding its mean proves
\eqref{eq:fifth-stage-score}. Finally, the Picard update rule gives \eqref{eq:fifth-stage-displacement} from
the score bound. In particular, high moments require concentration of
the ideal run, while the quadrature error is used only in $W_2$.
\end{proof}

\subsubsection{Proof of Theorem~\ref{thm:core}}

\begin{proof}[Proof of \cref{thm:core}]
We first verify the local correction bounds for the parameter choices
above.

\medskip\noindent\textbf{Cloud corrections.}
For an interior cloud update, fix a phase $n$ and a level
$\ell\in\{0,\ldots,K-2\}$, and set $\eta=\eta_n$, $\theta=\theta_n$.
Condition on $X^{[\ell]}$ and write
\[
 U_j\deq\frac{\cloud_{t_j}^{[\ell]}-X_{t_j}^{[\ell]}}{\sqrt\theta}\,,
 \qquad \xi_j\deq(U_j,G_j)\sim\cN(0,\Id_{2d})\,,
 \qquad j\in[J]\,.
\]
The blocks $(\xi_j)_{j\in[J]}$ are independent; put
$\xi\deq(\xi_j)_{j\in[J]}\sim\cN(0,\Id_{2Jd})$. Define
\[
 \Phi_j(\xi_j)
 \deq\widehat g_\eta
 \bigl(X_{t_j}^{[\ell]}+\sqrt\theta\,U_j;G_j\bigr)
 =\widehat g_\eta(\cloud_{t_j}^{[\ell]};G_j)\in\R^d\,.
\]
We apply the fluctuation correction to the map
\[
 \Phi : \xi \longmapsto
 \theta^{-1/2}\,(X_{t_i}^{[0]})_{i\in[J]}
       -\theta^{-1/2}\,
        \Bigl(\sum_{j\in[J]}\omega_{i,j}\Phi_j(\xi_j)\Bigr)_{i\in[J]}
 \in\R^{Jd}\,.
\]
Indeed, if $G'\sim\cN(0,\Id_{Jd})$ is the independent output noise,
then the whitened cloud proposal is precisely $\Phi(\xi)+G'$.
This map has the affine block structure
\[
 \Phi(\xi)=b+\sum_{j\in[J]}A_j\Phi_j(\xi_j)\,,
 \qquad b\deq\theta^{-1/2}\,(X_{t_i}^{[0]})_{i\in[J]}\,, \qquad
 A_jv\deq-\theta^{-1/2}\,(\omega_{i,j}v)_{i\in[J]}\,,
 \qquad v\in\R^d\,.
\]
The deterministic offset $b$ does not affect the fluctuation bound, and
$\norm{A_j}_{\op}^2=\theta^{-1}\sum_{i\in[J]}\omega_{i,j}^2$. By
\cref{lem:joint-estimator-derivatives}, the squared Lipschitz constant
of the gradient estimator as a function of that input block is at most
$\eta+\theta$. Let $L_{\rm fluc}$ denote the Lipschitz coefficient
supplied to the fluctuation correction. The
independent block bound following \cref{cor:block-correction} gives
\begin{equation}\label{eq:fifth-alpha}
 L_{\rm fluc}^2
 \leq\frac{\eta+\theta}{\theta}
       \sum_{i,j\in[J]}\omega_{i,j}^2
 \leq C\,\frac{(\eta+\theta)\,h^4}{\theta}\,.
\end{equation}
The last inequality follows by summing the first estimate in
\eqref{eq:integrated-weights-l2} over the $J$ output nodes.
The choice in \eqref{eq:fifth-cloud-scale},
\(\theta_n=C_{\rm cl}\mathfrak L^{9/2}\eta_nh^4
\sqrt{d+\mathfrak L}\), verifies $L_{\rm fluc}^2
 \leq\frac{C}{C_{\rm cl}}\,
 \mathfrak L^{-9/2}\,(d+\mathfrak L)^{-1/2}<\frac18$,
after increasing \(C_{\rm cl}\).
For the fluctuation clipping level, for output dimension
\(Jd\) and moment order \(p\le C\mathfrak L\), the requirements in
\eqref{eq:fixed-node-scale} read
\begin{equation*}
 C L_{\rm fluc}^{2K_{\rm ser}+2}Jd\le\delta_{\rm loc}\,,
 \qquad
 C\,\bigl\{
 L_{\rm fluc}^2\,\bigl(\sqrt{Jdp}+p\bigr)
 +L_{\rm fluc}^4Jd\bigr\}\le \clipB_{\rm fluc}\,.
\end{equation*}
Here \(\delta_{\rm loc}\) denotes any of the local accuracy budgets,
with \(\log(1/\delta_{\rm loc})=O(\mathfrak L)\). The first inequality
holds for \(K_{\rm ser}=C_{\rm ser}\mathfrak L\), after increasing
\(C_{\rm ser}\). For the second, since
\(J,p\le C\mathfrak L\), the preceding bound on \(L_{\rm fluc}\) allows us to take
$\clipB_{\rm fluc}
 \leq c\mathfrak L^{-3}$, where
$c$ can be made arbitrarily small by increasing
\(C_{\rm cl}\).

For the interior mean correction,
\eqref{eq:integrated-weights-l2} and Cauchy--Schwarz give
\begin{equation*}
 \sum_{j\in[J]}\norm{A_j}_{\op}
 \leq\frac{C\sqrt J\,h^2}{\sqrt{\theta_n}}
 \,.
\end{equation*}
The estimate \eqref{eq:marginal-score-size} states that the
mean estimator has \(L^p\) norm at most $C\,\bigl(\eta_n^{3/2}+\frac{\theta_n}{\sqrt{\eta_n}}\bigr)\,
 \sqrt{d+p}$.
Consequently, the whitened displacement \(b_{\rm int}\) satisfies
\begin{align*}
 b_{\rm int}
 &\leq\frac{C\sqrt J\,h^2}{\sqrt{\theta_n}}\,
 \Bigl(\eta_n^{3/2}+\frac{\theta_n}{\sqrt{\eta_n}}\Bigr)\,
 \sqrt{d+\mathfrak L}
 \leq C\,\bigl\{
 \mathfrak L^{-3/4}\,(d+\mathfrak L)^{1/4}\,h^2
 +\mathfrak L^{11/4}\,(d+\mathfrak L)^{3/4}\,h^4\bigr\}
 \le 1\,.
\end{align*}
The mean correction requirement in \eqref{eq:fixed-node-scale} is $C\,(\sqrt p\,b_{\rm int}+b_{\rm int}^2) \le \clipB_{\rm mean}$.
Since \(b_{\rm int}\le1\) and
\(p\le C\mathfrak L\), we may choose
\(\clipB_{\rm mean}\leq C\,\{\mathfrak L^{-1/4}\,(d+\mathfrak L)^{1/4}\,h^2
+\mathfrak L^{13/4}\,(d+\mathfrak L)^{3/4}\,h^4\}
\leq Cc^2\mathfrak L^{-11/4}\).
Decreasing $c$ makes this at most any prescribed small multiple of
$\mathfrak L^{-2}$.

At a hidden RGO call, take $y=X_{t_j}^{[\ell]}$,
$z=X_{t_j}^{[0]}$.
Fix $p_0=C\mathfrak L$ in the moment range of
\cref{lem:fifth-stage-moments}. That lemma and Markov's inequality give
bounds
\begin{equation*}
 \varepsilon_0\le Ch^2\sqrt{\kappa\,(d+\mathfrak L)}\,,\qquad
 L_{\rm score}\le C\sqrt{\kappa\,(d+\mathfrak L)}\,,
\end{equation*}
such that the event $\mathcal G_{\rm RGO}$ defined by
$\norm{y-z}\le\varepsilon_0$ and
$\norm{g_{\eta_n-2\theta_n}(y)}
 +\norm{g_{\eta_n-2\theta_n}(z)}\le L_{\rm score}$ satisfies
$\Pp(\mathcal G_{\rm RGO}^{\comp})\le\exp(-cp_0)$.
Condition on an incoming history in $\mathcal G_{\rm RGO}$, before
drawing the control gradient $g$. In the notation preceding
\cref{prop:symmetric-moments}, the condition in
\cref{lem:hidden-RGO} is $\mathfrak b_p\le c\clipB_{\rm RGO}$.
For \(p\le C\mathfrak L\), the definition of $\mathfrak b_p$ and the
likelihood moment bound
\eqref{eq:symmetric-proposal-moments} give
\begin{align*}
 \mathfrak b_p
 &\leq C\sqrt{\mathfrak L}\,\bigl\{
    \sqrt{\theta_n}\,\varepsilon_0
    +\theta_n^{3/2}L_{\rm score}
    +\sqrt{\eta_n\theta_n(d+\mathfrak L)}\bigr\}\\
 &\leq C\,\bigl\{
    \mathfrak L^{13/4}\sqrt\kappa\,(d+\mathfrak L)^{3/4}\,h^5
    +\mathfrak L^{35/4}\sqrt\kappa\,(d+\mathfrak L)^{5/4}\,h^9
    +\mathfrak L^{15/4}\,(d+\mathfrak L)^{3/4}\,h^4\bigr\}\,.
\end{align*}
The third term dominates, and we can take $\clipB_{\rm RGO}\le Cc^4\mathfrak L^{-9/4}\le1$.
On $\mathcal G_{\rm RGO}$, \cref{lem:hidden-RGO} therefore gives the
prescribed local TV accuracy, uniformly in the incoming history.
The event $\mathcal G_{\rm RGO}$ depends only on the phase
input $(X_{nh},P_{nh}^-)$, including the first momentum refresh, and
the pair $(\ell,j)$, not on the particular hidden RGO invocation.
Let $\mathcal G_n$ be the intersection of these events over the at
most $JK$ pairs in phase $n$. Under the ideal phase law,
\[
 \Pp(\mathcal G_n^\comp)\le JK\exp(-cp_0)
 \le\frac{\varepsilon}{256N}\,,
\]
after enlarging the constant in $p_0=C\mathfrak L$. On
$\mathcal G_n$, every hidden RGO call has the uniform guarantee
above, regardless of how many such calls are made. On its complement,
bound the TV error of the whole phase by $1$. Thus, this exceptional
event is charged once to the per-phase budget below.

It remains to verify the fluctuation and mean corrections for the full
endpoint update. In the whitening \eqref{eq:fifth-whitened-endpoint},
the block map multiplying the gradient estimator at node $j$ is
\[
 \mathsf B_j : v \longmapsto
 \begin{bmatrix}
  -\omega_{J,j}v/\sqrt{\tau}\\[1mm]
  -a_h\omega_jv/\sigma_h
 \end{bmatrix}\,.
\]
Since \(\tau\asymp h^3\), \(\sigma_h^2\asymp h\), and the 
weights satisfy \eqref{eq:integrated-weights-l2}, the maps from
\eqref{eq:fifth-whitened-endpoint} obey
\begin{equation*}
 \sum_{j\in[J]}\norm{\mathsf B_j}_{\op}^2\leq\frac{Ch}{J}\,,
 \qquad
 \sum_{j\in[J]}\norm{\mathsf B_j}_{\op}\leq C\sqrt h\,.
\end{equation*}
It follows that the squared Lipschitz coefficient of the whitened
endpoint proposal satisfies
\begin{equation*}
 L_{{\rm fluc},{\rm end}}^2
 \leq C\,\frac{(\eta_n+\theta_n)\,h}{J}
 \leq C\mathfrak L^{-1}\eta_nh\,.
\end{equation*}
For output dimension \(2d\) and \(p\le C\mathfrak L\), the
fluctuation requirements in \eqref{eq:fixed-node-scale} become $C L_{{\rm fluc},{\rm end}}^{2K_{\rm ser}+2}d
 \le\delta_{\rm loc}$,
which follows as before from the choice of \(K_{\rm ser}\). Similar arguments as before show that we can take $\clipB_{{\rm fluc},{\rm end}}
 \le Cc^3\mathfrak L^{-4}$.

Also, the endpoint mean
displacement \(b_{\rm end}\) satisfies
\begin{equation*}
 b_{\rm end}
 \le C\sqrt h\,
 \Bigl(\eta_n^{3/2}+\frac{\theta_n}{\sqrt{\eta_n}}\Bigr)\,
 \sqrt{d+\mathfrak L}\,.
\end{equation*}
This implies that we can choose
\begin{equation}\label{eq:fifth-endpoint-mean}
 \begin{split}
 \clipB_{{\rm mean},{\rm end}}
 &\le C\,\bigl\{\mathfrak L^{2}\sqrt{d+\mathfrak L}\,h^{7/2}
 +\mathfrak L^{11/2}\,(d+\mathfrak L)\,h^{11/2}\bigr\}
 \le Cc^{7/2}\mathfrak L^{-11/4}\,.
 \end{split}
\end{equation}

It remains to ensure that the recursive cloud requests are
subcritical; i.e., they do not explode the complexity. The required condition \eqref{eq:cloud-branching-smallness}
is $\mathfrak L^2\,(\clipB_{\rm fluc}+\clipB_{\rm mean})
 +\mathfrak L\,
   (\clipB_{{\rm fluc},{\rm end}}+\clipB_{{\rm mean},{\rm end}})
 \le c$.
This is easily verified using the bounds just proved by increasing
\(C_{\rm cl}\) and
then decreasing \(c\). The
endpoint clipping levels are multiplied by only one power of
\(\mathfrak L\) because there is only one endpoint correction along a
mandatory chain.

\medskip\noindent\textbf{Explicit RGO calls.}
Every RGO variance used in a phase belongs to $[\eta_n/2,2\eta_n]$.
The logarithms of the inverse accuracy budgets are
$O(\mathfrak L)$. Let $Q_{\rm RGO}$ denote a uniform upper bound on
the conditional expected number of reference queries made by one
explicit RGO call in this phase. For an RGO variance
\(a\asymp\eta_n\), the direct
sampler condition \eqref{eq:rgo-base-rgo} becomes $a\,\{\sqrt{d\mathfrak L}+\mathfrak L\}\le c$.
Since \(d+\mathfrak L\ge\mathfrak L\), this holds whenever
\(\eta_n\sqrt{(d+\mathfrak L)\,\mathfrak L}<c_0\). Otherwise, the
RGO bound in \cref{thm:fourth-RGO} gives
\begin{equation*}
 Q_{\rm RGO}
 \le C\,\bigl\{1+\mathbf1_{\mathcal A_n}
       \mathfrak L^{17/4}\,
       (d+\mathfrak L)^{1/4}\,h\bigr\}\,,\qquad
 \mathcal A_n\deq
 \{\eta_n\sqrt{(d+\mathfrak L)\,\mathfrak L}\ge c_0\}\,.
\end{equation*}
Here $c_0>0$ is a sufficiently small universal constant. Outside
$\mathcal A_n$, every call therefore makes $O(1)$ reference queries in
conditional expectation. On
$\mathcal A_n$, 
\(\eta_n\asymp\mathfrak L h^2\) and the choice of $h$ imply
\begin{equation*}
 c_0\le C\mathfrak L^{3/2}h^2\sqrt{d+\mathfrak L}
 \le C\mathfrak L^{-3/2}\,(d+\mathfrak L)^{1/10}
 \quad\Longrightarrow\quad d+\mathfrak L\ge c\mathfrak L^{15}\,.
\end{equation*}
After decreasing the step size constant $c$, this ensures
$d+\mathfrak L\ge C\mathfrak L^9$, so the sharper RGO bound applies
for every local accuracy budget.

We now combine the query count for one RGO call with its request
rate. One contracted
cloud vertex uses at most \(C\mathfrak L^2\) expected local reference
queries, as recorded in \eqref{eq:cloud-forest-caps}. The same display
and the RGO request estimate \eqref{eq:RGO-request-mean} bound its
expected number of explicit RGO requests by $C\mathcal B_{\rm mean}J$, where $\mathcal B_{\rm mean}
 \deq \clipB_{\rm mean} K+\clipB_{{\rm mean},{\rm end}}$.
The branching estimate makes the expected number of vertices per
phase universally bounded. Hence, the expected reference query count
in one phase is at most $C\,\{\mathfrak L^2+
\mathcal B_{\rm mean} J Q_{\rm RGO}\}$.
The constant part of \(Q_{\rm RGO}\) is absorbed into
\(C\mathfrak L^2\), because the branching inequality gives
\(\mathcal B_{\rm mean} J\le C\).

It remains to control the non-constant part on \(\mathcal A_n\).
Multiplying the two mean bounds by $KJ\le C\mathfrak L^2$
and $J\le C\mathfrak L$, respectively, gives
\begin{align*}
 \mathcal B_{\rm mean}JQ_{\rm RGO}
 &\le C\mathfrak L^2\,.
\end{align*}
Hence,
\begin{equation}\label{eq:fifth-phase-cost}
 \E[\text{reference queries in one phase}\mid\text{incoming history}]
 \le C\mathfrak L^2\,.
\end{equation}

\medskip\noindent\textbf{Accuracy.}
The recursive implementation is justified by induction from the
explicit level-$0$ cloud: local caps and decreasing cloud levels make
the request tree finite.
Separate recursive calls use independent seeds, with Poisson
counts, node choices, and Brownian evaluation times revealed before
their fresh Gaussian inputs.
Disjoint subtrees for the fluctuation correction, mean estimator,
and hidden RGO sampler ensure that, under exact cloud and RGO laws,
\cref{prop:marginal-score} supplies an independent unbiased mean
displacement; hence \cref{cor:block-correction} applies with the
parameters above.

Consider the ideal algorithm with exact RGO kernels, exact weak derivatives, and
unclipped likelihood tilts, but retain the depth-$K$ numerical phase.
Applying \eqref{eq:fifth-global-W2} with $\kappa\le\kappa_0$ and choosing
$N=\lceil C_N\mathfrak L/h\rceil$ with $C_N$ sufficiently large gives
\begin{equation}\label{eq:fifth-final-W2}
 W_2(\mu_N^X,\pi_{\eta_{\rm f}})
 \leq\frac{\varepsilon}{16}\sqrt{\eta_{\rm f}}\,,
 \qquad
 \eta_{\rm f}=\eta_0+N\tau\in[\eta_0,5\eta_0/4]\,.
\end{equation}
The moment bounds along the complete ideal run follow from
\cref{lem:fifth-run-concentration,lem:fifth-stage-moments}.

Insert first an uncapped reference algorithm which uses the clipped, truncated
cloud corrections above with exact Hessian and proximal queries, and keeps every RGO kernel exact.
Here, \emph{uncapped} refers only to the global slot cap: the
reference algorithm keeps the local stopping rules of
\cref{sec:recursive-cloud}. By \eqref{eq:cloud-local-cap} with
$u=C_{\rm cap}\mathfrak L$, each contracted vertex triggers a local
stop with conditional probability at most $\exp(-C_{\rm cap}\mathfrak L)$,
and the expectation bound in \cref{lem:adaptive-tree} gives at most
$4/3$ expected vertices per phase. After enlarging $C_{\rm cap}$, the
resulting stopping probability, at most $\varepsilon/(256N)$ per phase,
is included in the per-phase budget below.
At a clipping or truncation
step, include the ideal history in the proposal variable of
\cref{lem:clipping}. The ideal stage moments then bound the resulting integrated conditional
error. Choose the clipping ratios,
series cutoffs, and local recursive budgets so that their
combined error is at most $\varepsilon/(128N)$. Together with the
local stopping probability and $\Pp(\mathcal G_n^\comp)$ above, this
puts one complete reference phase within $\varepsilon/(64N)$, in
integrated total variation along the ideal history, of its ideal
counterpart. The corrections use only polylogarithmically many reference queries. The expectation
bound in \cref{lem:adaptive-tree} controls the sum over side branches.
Telescoping the $N$ phase kernels with
\cref{lem:adaptive-tv-comparison} gives
TV at most $\varepsilon/32$ between the exact ideal algorithm and this uncapped reference
algorithm.

Use the global slot cap \eqref{eq:cloud-global-cap} from
\cref{sec:recursive-cloud}.
Choose $C_{\rm cap}$ so that the reference algorithm exceeds this cap
with probability at most $\varepsilon/32$.
Stop the reference and implemented algorithms at this same cap. Give every
implemented RGO call before the terminal lift the TV budget
$\varepsilon/(64T_{\max})$. Then, \cref{lem:adaptive-tv-comparison}, applied with the
uncapped reference algorithm as its ideal algorithm and
$\varepsilon_{\rm cap}=\varepsilon/32$, puts the capped implementation
within $T_{\max}\cdot\varepsilon/(64T_{\max})+\varepsilon/32
=3\varepsilon/64$ of the uncapped algorithm. A final
triangle inequality with the preceding bound $\varepsilon/32$ gives
the total $5\varepsilon/64$, and hence
\begin{equation}\label{eq:fifth-executable-error}
 \TV\bigl(
 \Law(\widehat X_{Nh},\widehat P_{Nh}),
 \Law(X_{Nh},P_{Nh})\bigr)
 \leq\varepsilon/8\,.
\end{equation}
The bound
$\mathcal B_{\rm mean} J Q_{\rm RGO}\le C\mathfrak L^2$, together with
\eqref{eq:RGO-request-mean}, gives at most $CN\mathfrak L^2$ expected
pre-terminal reference queries. It also applies to the capped
implementation: the clipped acceptance probabilities and offspring
bounds hold on every implemented history, and
\cref{thm:fourth-RGO} supplies the required conditional expected query
bound uniformly at every requested center.

Add an independent $\cN(0,\eta_{\rm f}\Id)$ variable to the final position.  By
\eqref{eq:W2-KL} and \eqref{eq:fifth-final-W2},
\[
 \KL\bigl(
 \mu_N^X*\cN(0,\eta_{\rm f}\Id)\bigm\Vert\pi_{2\eta_{\rm f}}
 \bigr)
 \leq\frac{\varepsilon^2}{512}\,.
\]
Apply $R_{2\eta_{\rm f}}$ and use
$\pi_{2\eta_{\rm f}}R_{2\eta_{\rm f}}=\pi$. Implement this single terminal RGO call with TV budget $\varepsilon/32$. Data processing, Pinsker,
this terminal budget, and \eqref{eq:fifth-executable-error} bound the
TV distance to $\pi$ by
$\varepsilon/32+\varepsilon/32+\varepsilon/8=3\varepsilon/16$ when the
local corrections use exact Hessian and proximal queries. The
remaining $13\varepsilon/16$ is reserved for the gradient-only
realization of the local corrections, which is controlled in
\cref{thm:gradient-realization}. This proves
\eqref{eq:fifth-core-TV}.

\medskip\noindent\textbf{Query complexity.}
Finally, $N\leq C\mathfrak L\,h^{-1}
 \leq C\mathfrak L^{5/2}\,(d+\mathfrak L)^{1/5}$.
Equation \eqref{eq:fifth-phase-cost} bounds the expected number of
reference queries by
$CN\mathfrak L^2$.
The terminal RGO is also absorbed: the non-constant part of its
expected query bound, divided by
$N\mathfrak L^2$ is at most
$C\mathfrak L^{5/4}\,(d+\mathfrak L)^{1/4}\,h^2\le C$.
Thus the complete reference realization uses at most
$C\mathfrak L^{9/2}\,
 (d+\mathfrak L)^{1/5}$ expected reference queries.
Finally, \cref{thm:gradient-realization} proves the center caps and
realizes each local reference query using
$C\mathfrak L$ gradients; its separately charged RGO
normalizations use at most $C\mathfrak L$ gradient queries per
requested RGO. Since there are $O(N)$ such requests in expectation,
their expected total is absorbed by the displayed bound times
$C\mathfrak L$.
This proves \eqref{eq:fifth-core-complexity}.
\end{proof}

\subsection{Composition with the proximal sampler}\label{sec:proximal-composition}
In this subsection, we assume that $\nabla V$ is $1$-Lipschitz, as well as \eqref{eq:LSI} and~\eqref{eq:reference-point}.
By combining the proximal sampler with \cref{thm:core}, we obtain linear dependence on the LSI
constant $\kappa$.
The details are standard, so we only sketch the proof.

\begin{algorithm}[Proximal composition]\label{alg:proximal-composition}
Draw $X_0\sim\cN(x_{\rm ref},\frac12\Id)$ and choose
\begin{equation*}
 N_{\rm prox}\deq\Bigl\lceil4\kappa\log\frac{8\sqrt{3\kappa d}}{\varepsilon}\Bigr\rceil\,,
 \qquad \delta\deq\frac{\varepsilon}{8N_{\rm prox}}\,,\qquad
 B_{\rm cap}\deq C\kappa
       \sqrt{\frac{(N_{\rm prox}+1)\,d}{\varepsilon}}\,.
\end{equation*}
For $n=0,\ldots,N_{\rm prox}-1$, draw a fresh standard Gaussian $G_n$
and put $Y_n=X_n+G_n/\sqrt2$. If $\norm{Y_n-x_{\rm ref}}>B_{\rm cap}$,
stop and return $x_{\rm ref}$. Otherwise sample $R_{1/2,Y_n}$ within
TV error $\delta$ using the rescaled sampler of
\cref{thm:core}, with a reference point constructed by gradient descent,
to obtain $X_{n+1}$. If $\norm{X_{n+1}-x_{\rm ref}}>B_{\rm cap}$,
stop and return $x_{\rm ref}$. After all steps, return $X_{N_{\rm prox}}$.
\end{algorithm}

\begin{proof}[Proof of \cref{thm:fifth-root}]
Each $R_{1/2,y}$ has curvature in $[1,3]$. Gradient descent supplies an
admissible reference point in $O(\log(2+\norm{\nabla V(y)}))$ queries.
After rescaling, \cref{thm:core} gives TV error $\delta$ with
total expected gradient query bound
\begin{equation*}
 C\mathfrak L_y^{11/2}\,(d+\mathfrak L_y)^{1/5}\,,\qquad
 \mathfrak L_y\deq\log\frac{12d\,(1+\norm{\nabla V(y)})}{\delta}\,.
\end{equation*}

Write $\mu_n$ for the laws of the uncapped exact proximal chain.
LSI, smoothness, and \eqref{eq:reference-point} give
$\KL(\mu_0\mmid\pi)\leq6\kappa d$. Thus,
\cite[Theorem~3]{Chen+22ProxSampler} and Pinsker's inequality yield
$\TV(\mu_{N_{\rm prox}},\pi)\leq\varepsilon/8$.
The LSI transport inequality \cite{OttVil00LSI} and the triangle
inequality through $\pi$ give, along the uncapped exact chain,
$\E\norm{X_n-x_{\rm ref}}^2+\E\norm{Y_n-x_{\rm ref}}^2
\leq C\kappa^2d$ for all $n\geq0$.
Markov's inequality and a union bound therefore charge at most
$\varepsilon/8$ for the caps. The RGO replacement errors sum to
$N_{\rm prox}\delta=\varepsilon/8$, proving accuracy.

At every implemented RGO center,
$\norm{\nabla V(Y_n)}\leq\sqrt{d/\kappa}+B_{\rm cap}$, so
$\mathfrak L_{Y_n}\leq C\mathfrak L$. Summing the RGO query bound over
$N_{\rm prox}=O(\kappa\mathfrak L)$ calls gives
$C\mathfrak L^{13/2}\kappa\,(d+\mathfrak L)^{1/5}$, as claimed.
\end{proof}

%% file: 5_high_acc_sections/appendix_a_rgo.tex
\section{Improved RGO sampler}\label{sec:rgo}

The hidden RGO sampler (\cref{gadget:hidden-RGO}) also requires explicit RGO calls, and in order for their cost to be acceptable, we must develop an improved RGO sampler.

We build it from the same Gaussian corrections as
the main sampler. The additional observation is that an RGO target has a known quadratic part. We solve the resulting linear term in each Picard
update exactly and apply the cloud correction only to the residual gradient.
In turn, this requires further RGO calls, but each RGO used by this construction is localized
enough for a direct rejection routine. Thus, the construction
uses the local results of \cref{sec:fixed-correction,sec:marginal,sec:recursive},
but does not invoke \cref{thm:core} or the oracle being proved.

We remark that another approach to develop an improved RGO sampler would be to adapt proximal BPS~\cite{ProximalBPS}, with the Picard HMC warm start~\cite{PicardHMC}, to the RGO structure.

\begin{theorem}[Improved RGO sampler]\label{thm:fourth-RGO}
Let $V\in C^2(\R^d)$ satisfy $0\preceq\nabla^2V\preceq\Id$.
For every $0<\eta\leq1$, $y\in\R^d$, and $0<\varepsilon<1/4$, there is
an almost surely terminating routine with output law $\widehat R_{\eta,y}$
such that $\TV(\widehat R_{\eta,y},R_{\eta,y})\leq\varepsilon$. With
$\mathfrak L\deq\log(d/(\eta\varepsilon))$, its expected query count
satisfies
\begin{equation}\label{eq:QE4-cost}
 \sup_{y\in\R^d}\E[N_{\nabla V}+N_{\prox}+N_{\nabla^2V}]
 \le C\,\bigl\{1+\mathfrak L^{17/4}\,
 (\eta^2\,(d+\mathfrak L))^{1/4}\bigr\}\,.
\end{equation}
Here the counts refer to gradient, proximal, and full Hessian queries.
If $d+\mathfrak L\ge\mathfrak L^9$, the exponent $17/4$
in \eqref{eq:QE4-cost} improves to $15/4$.
\end{theorem}

Using the direct RGO sampler of \cref{lem:rgo-base} for every
RGO call in the Gaussian cloud sampler would yield an overall query
complexity of $\wtO(\kappa d^{1/4})$. Indeed, suppressing logarithmic
factors, its constant-cost guarantee requires $\eta\lesssim d^{-1/2}$,
which forces $h\lesssim d^{-1/4}$ through $\eta\asymp h^2$.
The improved bound $\wtO(1+\sqrt{\eta}\,d^{1/4})$ per RGO call
allows $h\asymp d^{-1/5}$, yielding the overall complexity
$\wtO(\kappa d^{1/5})$; see the cost balance in \cref{ssec:roadmap}.

\subsection{Direct RGO sampler}
We make use of the FORS sampler of
\cite[Theorems~3.1 and~3.3]{Chen+26HighAccDiffusion} as a baseline RGO sampler.

\begin{lemma}[Direct RGO sampler]\label{lem:rgo-base}
Let $V\in C^2(\R^d)$ satisfy $0\preceq\nabla^2V\preceq\Id$, $0<\eta\leq1$, and $0<\varepsilon<1/2$.
If
\begin{equation}\label{eq:rgo-base-rgo}
 \eta\,\bigl\{\sqrt{d\log(1/\varepsilon)}+\log(1/\varepsilon)\bigr\}\leq c\,,
\end{equation}
then $R_{\eta,y}$ can be sampled to TV accuracy $\varepsilon$ using one exact
proximal query and $O(1)$ expected gradient queries, uniformly in $y$.
\end{lemma}
\begin{proof}
Apply \cite[Theorem~3.3]{Chen+26HighAccDiffusion} with
smoothness exponent $s=1$, smoothness constant $\beta_1=1$, and
$x_+=\prox_{\eta V}(y)$. The proximal residual in that theorem is
zero, and its chi-squared guarantee implies the required TV bound.
The clipping level is a universal constant, so
\cite[Theorem~3.1]{Chen+26HighAccDiffusion} gives $O(1)$ expected
gradient queries after the single proximal query.
\end{proof}

\subsection{Removing the quadratic part of the Picard update}\label{ssec:remove-quadratic}

The RGO potential contains a known quadratic term. We use this structure to
reduce the part of each Picard update that needs Gaussian correction.
Following the progression in \cite[Section 3.2]{PicardHMC}, we first write
the continuous Hamiltonian dynamics, then introduce Picard iteration
and quadrature, and finally replace the residual gradient by a computable
estimator.

\paragraph{Separating the quadratic and residual potentials.}
Compute $x_\star\deq\prox_{\eta V}(y)$ and write
$x=x_\star+\sqrt{\eta/2}\,z$. The optimality condition
$y=x_\star+\eta\nabla V(x_\star)$ cancels the linear term in the
RGO potential. Up to an additive constant, the potential in the new
coordinates is therefore
\begin{equation*}
 \widetilde V(z)\deq\frac14\,\norm z^2+F(z)\,,\qquad
 F(z)\deq V(x_\star+\sqrt{\eta/2}\,z)-V(x_\star)
      -\sqrt{\eta/2}\,\ip{\nabla V(x_\star)}z\,.
\end{equation*}
Thus $\nabla F(0)=0$, $0\preceq\nabla^2F\preceq(\eta/2)\Id$, and
$\frac12\Id\preceq\nabla^2\widetilde V\preceq\Id$. The normalized
target is well-conditioned, while its residual gradient has Lipschitz
constant of order $\eta$.
Let $\widetilde\pi$ be this target, and write
$\widetilde\pi_\sigma\deq\widetilde\pi*\cN(0,\sigma^2\Id)$,
$\widetilde V_\sigma\deq-\log\widetilde\pi_\sigma$, and
$\widetilde\Pi_\sigma\deq\widetilde\pi_\sigma\otimes\cN(0,\Id)$.
Here, $\eta$ remains the requested RGO variance in the original
coordinates, whereas $\sigma^2$ is the smoothing variance in the
normalized coordinates. The cloud variance $\theta$ below is also
measured in the normalized coordinates.
Completing the quadratic terms in the posterior
$R^{\widetilde V}_{\sigma^2,z}$ gives
$R^F_{2\sigma^2/(2+\sigma^2),\,2z/(2+\sigma^2)}$. Consequently,
\begin{equation}\label{eq:rgo-residual-gradient}
 \nabla\widetilde V_\sigma(z)=\frac{z}{2+\sigma^2}+r_\sigma(z)\,,\qquad
 r_\sigma(z)\deq\frac{2}{2+\sigma^2}\,
 g^F_{2\sigma^2/(2+\sigma^2)}\Bigl(\frac{2z}{2+\sigma^2}\Bigr)\,.
\end{equation}
Here $g^F_b(v)\deq\E_{R^F_{b,v}}\nabla F$ is well-defined even when
$\exp(-F)$ is not integrable.
The posterior covariance identities give
$0\preceq Dg_b^F\preceq(\eta/2)\Id$, so
$\Lip(r_\sigma)\le\eta/2$. These identities require integrability
of the strongly log-concave posterior, which holds for every $b>0$.

\paragraph{Exact integration of the linear dynamics.}
Fix the smoothing level $\sigma$ during one phase and put
$\lambda_\sigma\deq(2+\sigma^2)^{-1}$. Without loss of generality, consider time $nh$ to be time $0$
in the local time subscripts, so that $X_0$ denotes $X_{nh}$, and $P_0^-$ denotes
the momentum $P_{nh}^-$ after the first OU half-refresh. By
\eqref{eq:rgo-residual-gradient}, the Hamiltonian dynamics between
refreshments are
\begin{equation*}
 \dot X_t=P_t\,,\qquad
 \dot P_t=-\lambda_\sigma X_t-r_\sigma(X_t)\,,\qquad 0\le t\le h\,.
\end{equation*}
The initial momentum is $P_0^-$. With the residual omitted, this is a
harmonic oscillator, which can be solved exactly. Define
\begin{equation*}
 c_\sigma(t)\deq\cos(\sqrt{\lambda_\sigma}\,t)\,,\qquad
 s_\sigma(t)\deq\frac{\sin(\sqrt{\lambda_\sigma}\,t)}{\sqrt{\lambda_\sigma}}\,.
\end{equation*}
The linear flow is
\begin{equation}\label{eq:rgo-linear-flow}
 \begin{aligned}
 X_t^{\rm lin}&=c_\sigma(t)\,X_0+s_\sigma(t)\,P_0^-\,, \qquad
 P_t^{\rm lin}=-\lambda_\sigma s_\sigma(t)\,X_0+c_\sigma(t)\,P_0^-\,.
 \end{aligned}
\end{equation}
Variation of constants now expresses the full flow as
\begin{equation}\label{eq:rgo-variation-of-constants}
 \begin{aligned}
 X_t&=X_t^{\rm lin}-\int_0^t s_\sigma(t-s)\,r_\sigma(X_s)\,\dd s\,,\qquad
 P_t=P_t^{\rm lin}-\int_0^t c_\sigma(t-s)\,r_\sigma(X_s)\,\dd s\,.
 \end{aligned}
\end{equation}
All dependence on the unknown trajectory is now contained in $r_\sigma$.

\paragraph{Picard iteration for the residual.}
Start from $\overline X_t^{[0]}\deq X_t^{\rm lin}$ and replace the
residual in \eqref{eq:rgo-variation-of-constants} by its value on the
preceding trajectory:
\begin{equation*}
 \overline X_t^{[\ell+1]}
 \deq X_t^{\rm lin}-\int_0^t s_\sigma(t-s)\,
       r_\sigma(\overline X_s^{[\ell]})\,\dd s\,,
 \qquad 0\le t\le h\,,\quad \ell\ge0\,.
\end{equation*}
The bars distinguish these continuous trajectories from the quadrature
iterates below. Since $\abs{s_\sigma(t-s)}\le t-s$ and
$\Lip(r_\sigma)\le C\eta$, this iteration contracts in the uniform
norm by at most $C\eta h^2$.
Equivalently, each update solves a linear equation with the preceding
residual as its forcing function. Integrating that equation twice gives
\begin{equation}\label{eq:rgo-implicit-Picard}
 \overline X_t^{[\ell+1]}
 =X_0+tP_0^--\int_0^t(t-s)\,
 \bigl\{\lambda_\sigma\overline X_s^{[\ell+1]}
              +r_\sigma(\overline X_s^{[\ell]})\bigr\}\,\dd s\,.
\end{equation}
The linear term uses the \emph{new} trajectory, whereas the residual
uses the preceding trajectory. This is the form we discretize, so that
we can use the same quadrature weights as the main sampler.

\paragraph{Chebyshev--Lobatto quadrature.}
Use the nodes $t_j\in[0,h]$, cardinal polynomials $\ell_j$,
and integrated weights from \cref{sec:preliminaries}. Replace the
expression in braces in \eqref{eq:rgo-implicit-Picard} by its Lagrange
interpolant at these nodes. Writing $X^{[\ell]}$ for the resulting
stacked numerical positions, the node equations become
\begin{equation}\label{eq:rgo-quadrature-equations}
 X_{t_i}^{[\ell+1]}
 =X_0+t_iP_0^--\sum_{j\in[J]}\omega_{i,j}\,
 \bigl\{\lambda_\sigma X_{t_j}^{[\ell+1]}
                   +r_\sigma(X_{t_j}^{[\ell]})\bigr\}\,,
 \qquad i\in[J]\,.
\end{equation}
Put
$W\deq(\omega_{i,j})_{i,j\in[J]}$,
$w\deq(\omega_j)_{j\in[J]}$, and
$X^{\rm free}\deq(X_0+t_jP_0^-)_{j\in[J]}$. Then,
\eqref{eq:rgo-quadrature-equations} is
\begin{equation*}
 (\Id_J+\lambda_\sigma W)X^{[\ell+1]}
 =X^{\rm free}-Wr_\sigma(X^{[\ell]})\,.
\end{equation*}
These deterministic matrices act on node vectors by tensoring with
$\Id_d$, and $r_\sigma$ acts componentwise. Thus, we can solve the
linear part of the quadrature equations exactly by precomputing
\begin{equation*}
 Q_\sigma\deq(\Id_J+\lambda_\sigma W)^{-1}\,,\qquad
 A\deq Q_\sigma X^{\rm free}\,.
\end{equation*}
Initialize $X^{[0]}=A$ and perform the interior updates
\begin{equation}\label{eq:rgo-residual-Picard}
 X^{[\ell+1]}=A-Q_\sigma W r_\sigma(X^{[\ell]})\,,
 \qquad \ell=0,\ldots,K-2\,.
\end{equation}
The anchor $A$ is the quadrature approximation to the linear trajectory
in \eqref{eq:rgo-linear-flow}.

The transformed weights retain the scales $h^2$ and $h$.
For a scalar matrix $B\in\R^{J\times J}$, write
$\norm{B}_{\infty\to\infty}\deq
\max_{i\in[J]}\sum_{j\in[J]}\abs{B_{i,j}}$. If $h\le c$, the
Neumann series, \eqref{eq:Cheb-bounds}, and
\eqref{eq:integrated-weights-l1} give
\begin{equation*}
 \norm{Q_\sigma}_{\infty\to\infty}\le2\,,\qquad
 \norm{Q_\sigma W}_{\infty\to\infty}\le2h^2\,,\qquad
 \norm{w^\T Q_\sigma}_1\le2h\,.
\end{equation*}

The endpoint update counts as the final Picard layer. Its position is
the last component of $A-Q_\sigma W r_\sigma(X^{[K-1]})$.
For momentum, integrate the same linear term evaluated on this new
layer and the residual evaluated on $X^{[K-1]}$, now using the weights
$w$. The identity
$\Id_J-\lambda_\sigma Q_\sigma W=Q_\sigma$ then gives
\begin{align}\label{eq:rgo-residual-endpoint}
 \widetilde X_h^{[K]}
 &=e_J^\T A-e_J^\T Q_\sigma W
       r_\sigma(X^{[K-1]})\,, \qquad
 \widetilde P_h^{[K]}
 =P_0^--\lambda_\sigma w^\T A-w^\T Q_\sigma
       r_\sigma(X^{[K-1]})\,,
\end{align}
where $e_J$ is the last coordinate vector.

\paragraph{Implementing the residual evaluations.}
The preceding equations still use the exact conditional expectation
$r_\sigma$. The computable residual estimator is
\begin{equation*}
 \widehat r_\sigma(z;G)\deq\frac{2}{2+\sigma^2}\,\nabla F\Bigl(
 \prox_{\frac{2\sigma^2}{2+\sigma^2}F}\Bigl(\frac{2z}{2+\sigma^2}\Bigr)
 +\sqrt{\frac{2\sigma^2}{2+\sigma^2}}\,G\Bigr)\,,
 \qquad G\sim\cN(0,\Id)\,.
\end{equation*}
Its Lipschitz constants in $z$ and $G$ are at most $C\eta$ and
$C\eta\sigma$, respectively, by the Hessian bound on $F$ and
non-expansiveness of its proximal map. We replace residual evaluations
by this estimator and apply the fluctuation and mean corrections.
The known anchor and transformed weights require no such correction.
For clarity, all residual oracles reduce to the stated oracles
for $V$. Put $s=\sqrt{\eta/2}$ and $g_\star=\nabla V(x_\star)$.
Then
\begin{equation}\label{eq:rgo-residual-oracles}
 \begin{aligned}
 \nabla F(v)&=s\,\{\nabla V(x_\star+sv)-g_\star\}\,,\qquad
 \nabla^2F(v)=s^2\,\nabla^2V(x_\star+sv)\,,\\
 \prox_{bF}(v)
 &=s^{-1}\,\bigl\{\prox_{s^2bV}
       (x_\star+sv+s^2b g_\star)-x_\star\bigr\}\,.
 \end{aligned}
\end{equation}
The initial proximal query determines both $x_\star$ and
$g_\star=(y-x_\star)/\eta$; each subsequent residual oracle call
uses one corresponding oracle call to $V$.
Use the same final cloud in both endpoint components, then apply the
joint endpoint correction after the second OU half-refresh and position
smoothing of variance $\tau$, as in \cref{alg:fifth-phase}. The smaller
residual Lipschitz constants are what improve the query count; their
consequences are quantified in \cref{lem:rgo-cloud-bounds}.

\subsection{Proof of Theorem~\ref{thm:fourth-RGO}}

The key improvement comes from $\norm{\nabla^2F}_{\op}\le\eta/2$.
The derivative bounds for
the residual estimator therefore carry a factor $\eta$. This gives a
factor $\eta^2$ in the fluctuation bound and in the part of the mean
error caused by replacing an RGO draw with a Gaussian. The cloud
variance can consequently be reduced by a factor $\eta^2$ as well.
We first record the residual versions of the local correction
bounds. We then verify the parameter constraints, compare the residual
iteration with the common collocation limit, and apply the recursive
call and TV error bounds from \cref{sec:recursive-cloud}.

\begin{lemma}[Residual correction bounds]\label{lem:rgo-cloud-bounds}
Fix $0<\sigma^2\le1$, $0<\theta\le\sigma^2/8$, and $0<h\le c$.
For the residual updates
\eqref{eq:rgo-residual-Picard}--\eqref{eq:rgo-residual-endpoint}, with
endpoint position smoothing $\tau\asymp h^3$, the squared Lipschitz
coefficients supplied to the fluctuation correction satisfy
\begin{equation}\label{eq:rgo-residual-fluctuation}
 L_{\rm fluc}^2\le
 C\eta^2\,\frac{(\sigma^2+\theta)\,h^4}{\theta}\,,\qquad
 L_{{\rm fluc},{\rm end}}^2\le
 C\eta^2\,\frac{(\sigma^2+\theta)\,h}{J}\,.
\end{equation}
For independent $\cloud\sim\cN(z,\theta\Id)$ and $G\sim\cN(0,\Id)$,
the mean correction has an unbiased
estimator $\widehat\Delta_{\rm res}(z)$ of
$r_\sigma(z)-\E\widehat r_\sigma(\cloud;G)$ satisfying
\begin{equation}\label{eq:rgo-residual-mean}
 \norm{\widehat\Delta_{\rm res}(z)}_{L^p}
 \le C\,\Bigl(\eta^2\sigma^3+\frac{\eta\theta}{\sigma}\Bigr)\,
 (\sqrt d+\sqrt p)\,,\qquad p\ge2\,.
\end{equation}
Besides the required exact cloud and RGO draws, this estimator uses
$O(1)$ gradient, proximal, and Hessian queries.
\end{lemma}
\begin{proof}
The joint derivative bound in \cref{lem:joint-estimator-derivatives}
is now $C\eta\sqrt{\sigma^2+\theta}$, by the residual estimator in
\cref{ssec:remove-quadratic}. The transformed weights preserve the square estimates used
in \cref{sec:work}: $\norm{Q_\sigma W}_{\HS}\le Ch^2$,
each row of $Q_\sigma W$ has norm at most $Ch^2/\sqrt J$, and
$\norm{w^\T Q_\sigma}\le Ch/\sqrt J$. Indeed,
$\norm W_{\op}\le\norm W_{\HS}\le Ch^2$ and
$\lambda_\sigma\le1/2$, so the Neumann series gives
$\norm{Q_\sigma}_{\op}\le2$ after decreasing $c$. Together with
$\norm{Q_\sigma}_{\infty\to\infty}\le2$ and the row estimates for $W$,
this proves the three displayed bounds. The same whitening and independent block
calculation as in \eqref{eq:fifth-alpha} gives
\eqref{eq:rgo-residual-fluctuation}.

For the mean, apply the proof of \cref{prop:marginal-score} to $F$, with posterior
variance $2\sigma^2/(2+\sigma^2)$ and cloud variance
$4\theta/(2+\sigma^2)^2$. In \eqref{eq:marginal-cloud-identity}, the
OU action has norm $C\eta\sigma$ and the gradient difference has
$L^p$ norm $C\eta\sigma\,(\sqrt d+\sqrt p)$.
The divergence estimator in \eqref{eq:divergence-formula} contains one
Hessian of $F$, so its contribution is $C\eta\theta/\sigma$.
The remaining rescaling factors are bounded by universal constants.
Only convexity of $F$, its Hessian bound, and the resulting
strong log-concavity of its posteriors are used here; integrability
of $\exp(-F)$ is unnecessary.
This proves \eqref{eq:rgo-residual-mean} and the query claim.
\end{proof}

\begin{proof}[Proof of \cref{thm:fourth-RGO}]
Put $D=d+\mathfrak L$. If
$\eta\sqrt{D\mathfrak L}\le c_0$ for a sufficiently small universal
$c_0$, use \cref{lem:rgo-base}. In the remaining case we use the
residual updates above and the local correction and call count
lemmas, without invoking \cref{thm:core}.

\paragraph{Parameters and explicit RGO calls.}
The rescaling $x=x_\star+\sqrt{\eta/2}\,z$ sends $R^F_{b,v}$ to
\[
 R^V_{\eta b/2,\,x_\star+\sqrt{\eta/2}\,v+(\eta b/2)\nabla V(x_\star)}\,.
\]
In a phase with smoothing variance $\sigma^2$, put
$a_\sigma=2/(2+\sigma^2)$ and $b_\sigma=a_\sigma\sigma^2$.
The mean estimator requests $R^F_{b_\sigma,a_\sigma z}$ using
cloud variance $a_\sigma^2\theta$ and anchor $a_\sigma A_j$ at node
$j$. Its hidden RGO implementation therefore makes explicit calls
with variance $b_\sigma-2a_\sigma^2\theta$.
All these variances are $O(\eta\sigma^2)$ after rescaling to $V$.

Take $J,K$ as in \eqref{eq:fifth-phase-count}, and set
\begin{equation}\label{eq:rgo-cloud-scales}
 \begin{gathered}
 h=c\mathfrak L^{-5/4}\,(\eta^2D)^{-1/4}\,,\qquad
 \sigma_0^2=C_\sigma\mathfrak L^2h^2\,,\qquad \tau=c_\tau h^3\,,\\
 N=\lceil C_N\mathfrak L/h\rceil\,,\qquad
 \sigma_n^2=\sigma_0^2+n\tau\,,\qquad
 \theta_n=C_{\rm cl}\mathfrak L^{9/2}\eta^2\sigma_n^2h^4
              \sqrt D\,,\quad 0\le n\le N\,.
 \end{gathered}
\end{equation}
Choose the large constants first, then $c_\tau$, and finally $c$
sufficiently small. Since $\eta\sqrt{D\mathfrak L}>c_0$,
\begin{equation*}
 h\le Cc\mathfrak L^{-1}\,,\qquad
 \sigma_0^2\le C_\sigma c^2/c_0\,,\qquad
 \frac{N\tau}{\sigma_0^2}\le Cc_\tau\mathfrak L^{-1}\,.
\end{equation*}
Thus $\sigma_n^2\in[\sigma_0^2,5\sigma_0^2/4]$ and $\sigma_N^2\le1$.
Moreover,
\begin{equation*}
 \frac{\theta_n}{\sigma_n^2}
 =C_{\rm cl}c^4(\mathfrak L D)^{-1/2}\le\frac18\,,\qquad
 \eta\sigma_n^2\{\sqrt{d\mathfrak L}+\mathfrak L\}\le Cc^2\,.
\end{equation*}
Consequently, every explicit RGO call has constant expected query cost
by \cref{lem:rgo-base}, for local accuracy $\delta_{\rm loc}$ with
$\log(1/\delta_{\rm loc})=O(\mathfrak L)$.
This also covers the terminal call:
$R^{\widetilde V}_{2\sigma_N^2,u}$ is
$R^F_{2\sigma_N^2/(1+\sigma_N^2),\,u/(1+\sigma_N^2)}$, whose variance
in the original coordinates is $\eta\sigma_N^2/(1+\sigma_N^2)$.

\paragraph{Reference accuracy and stage moments.}
Initialize $X_0=0$ and $P_0\sim\cN(0,\Id)$; the origin minimizes
$\widetilde V$. At infinite depth, both the ordinary and residual
iterations solve
\[
 X=X^{\rm free}-W\nabla\widetilde V_\sigma(X)\,.
\]
This equation has a unique solution for $h\le c$. Taking the
depth to infinity in \cref{prop:reference-Picard}, with $\kappa=2$,
therefore gives a common phase kernel with contraction $1-c_1h$
and stationary $W_2$ defect
$C\sqrt d\,B_{J,2}h(h/\sigma)^J$.
Here $h/\sigma_n\le C_\sigma^{-1/2}\mathfrak L^{-1}$ and
$B_{J,2}\le(CJ)^J$, so the defect is at most
$\exp(-C_{\rm err}\mathfrak L)$ after choosing $C_J,C_\sigma$ large
enough. The constant $C_{\rm err}$ can be chosen as large as needed.
The proof of \cref{lem:fifth-run-concentration} is uniform in
Picard depth and hence also applies to this limiting kernel.
Together with the $W_2$ recursion, it gives phase-input $L^p$
moments $C\sqrt{d+p}$ for the infinite-depth run.

Put $q=C\eta h^2<1/2$ and
$\delta_K=Cq^{K-1}$. Let $\Phi_{\rm res}^{[K]}$ and $\Phi_\infty$
denote the finite residual and limiting phase outputs, using common
refresh noises. The residual fixed-point contraction and the endpoint
weight bounds give, conditionally on a phase input $\zeta\in\R^{2d}$,
\[
 \norm{\Phi_{\rm res}^{[K]}(\zeta)-\Phi_\infty(\zeta)}_{L^p(M_2)}
 \le \delta_K\,(\norm\zeta+\sqrt{d+p})\,.
\]
Since $\eta h^2=c^2\mathfrak L^{-5/2}D^{-1/2}$, the chosen depth
makes $\delta_K\le\exp(-C_{\rm err}\mathfrak L)$ after increasing
$C_K$. Let $e_{n,p}$ be the $W_{p,M_2}$ distance between the two
phase-input laws, including the retained position noise. The
limiting kernel contraction and the preceding bound imply
\[
 e_{n+1,p}\le(1-c_1h+C\delta_K)\,e_{n,p}
                 +C\delta_K\sqrt{d+p}\,,\qquad e_{0,p}=0\,.
\]
Taking $C_K$ large enough that $C\delta_K\le c_1h/2$, we obtain
$e_{n,p}\le C\delta_Kh^{-1}\sqrt{d+p}$ uniformly in $n$.
In particular, the finite residual run has phase-input $L^p$
moments $C\sqrt{d+p}$.
Since $\nabla F(0)=0$, the posterior moment bounds and
\eqref{eq:rgo-residual-Picard} now give residual score and
anchor-displacement bounds $C\eta\sqrt{d+p}$ and
$C\eta h^2\sqrt{d+p}$, respectively, uniformly over stages and nodes.

\paragraph{Clipping levels and recursive calls.}
Substitute \cref{lem:rgo-cloud-bounds} into
\eqref{eq:fixed-node-scale}, with $p=O(\mathfrak L)$.
The interior and endpoint output dimensions are $Jd$ and $2d$,
respectively. Using the transformed weight bounds and
$J\asymp\mathfrak L$, we can choose valid clipping levels satisfying
\begin{equation}\label{eq:rgo-clipping-scales}
 \begin{aligned}
 \clipB_{\rm fluc}
 &\le \frac{C}{C_{\rm cl}}\,\mathfrak L^{-7/2}\,,
 & \clipB_{\rm mean}
 &\le Cc^2\mathfrak L^{-7/4}D^{-1/4}\,,\\
 \clipB_{{\rm fluc},{\rm end}}
 &\le Cc^3\eta^{1/2}\mathfrak L^{-9/4}D^{-1/4}\,,
 & \clipB_{{\rm mean},{\rm end}}
 &\le Cc^{7/2}\eta^{1/4}\mathfrak L^{-7/8}D^{-3/8}\,.
 \end{aligned}
\end{equation}
For example, before substituting $h$, the two contributions to
$\clipB_{\rm mean}$ are bounded by
$C\eta\mathfrak L^{3/4}D^{1/4}h^2$ and
$C\eta^2\mathfrak L^{13/4}D^{3/4}h^4$.
Since $D\ge\mathfrak L$ and $\eta\le1$, the displayed bounds give
\[
 \mathfrak L^2(\clipB_{\rm fluc}+\clipB_{\rm mean})
 +\mathfrak L(\clipB_{{\rm fluc},{\rm end}}+\clipB_{{\rm mean},{\rm end}})
 \le C/C_{\rm cl}+C(c^2+c^3+c^{7/2})\,.
\]
This verifies \eqref{eq:cloud-branching-smallness}. The fluctuation
series truncation is negligible with $K_{\rm ser}=C_{\rm ser}\mathfrak L$.

For the hidden RGO calls, write
$b=b_{\sigma_n}$ and $\vartheta=a_{\sigma_n}^2\theta_n$.
The stage moments give, except on an event of probability
$\exp(-C_{\rm cap}\mathfrak L)$, the bound
$\varepsilon_0\le C\eta h^2\sqrt D$ and a residual score bound
$L_{\rm score}\le C\eta\sqrt D$.
Tracking $\Lip(\nabla F)\le\eta/2$ in the proof of
\cref{prop:symmetric-moments} gives the sufficient clipping bound
\[
 C\eta\bigl\{\sqrt{\vartheta p}\,
       (\varepsilon_0+\vartheta L_{\rm score})
       +\sqrt{b\vartheta}\,(\sqrt{dp}+p)\bigr\}
 \le c \clipB_{\rm RGO}\,.
\]
Thus we can choose
\[
 \clipB_{\rm RGO}\le C\eta\sigma_n\sqrt{\theta_n\mathfrak L D}
 \le Cc^4\mathfrak L^{-1/4}D^{-1/4}\le1\,.
\]
All constants here are uniform in the original center $y$.

\paragraph{Accuracy and query count.}
Use the independent recursive calls and local caps of
\cref{sec:recursive-cloud}, starting from the explicit cloud
$\cN(A,\theta_n\Id_{Jd})$. The preceding branching and hidden RGO
bounds, together with \cref{lem:adaptive-tree,lem:RGO-forest},
give $O(N\mathfrak L^2)$ expected local queries and $O(N)$ expected
explicit RGO calls. Each latter call costs $O(1)$, as verified above.
The same lemmas give the global slot cap
$T_{\max}=C\,(N+\mathfrak L^3)\,\mathfrak L^2$ outside probability
$\exp(-C_{\rm cap}\mathfrak L)$.
Allocate a sufficiently small multiple of $\varepsilon/T_{\max}$
to each local clipping, truncation, and RGO error. Their inverse
accuracy budgets have logarithms $O(\mathfrak L)$, as required.
\Cref{lem:adaptive-tv-comparison} and the stage-moment exceptional
bounds then control the total implementation error.

The reference comparison above, with $N=\lceil C_N\mathfrak L/h\rceil$,
gives final position error at most $\varepsilon\sigma_N/16$ in $W_2$
relative to $\widetilde\pi_{\sigma_N}$.
Add an independent $\cN(0,\sigma_N^2\Id)$ variable and apply the
terminal RGO $R^{\widetilde V}_{2\sigma_N^2}$ with TV budget
$\varepsilon/32$. The Gaussian smoothing inequality
\eqref{eq:W2-KL} and RGO disintegration give TV error at most
$\varepsilon$ after all budgets are combined and the output is
rescaled to the original coordinates.
The capped recursion terminates, and each direct FORS call
terminates almost surely. By \eqref{eq:rgo-residual-oracles},
including initialization and the terminal call, the expected total
query count is
\[
 C\,(1+N\mathfrak L^2)
 \le C\,\{1+\mathfrak L^{17/4}\eta^{1/2}\,
                         (d+\mathfrak L)^{1/4}\}\,,
\]
uniformly in $y$.

For the sharper bound when $D\ge\mathfrak L^9$, replace the
first two parameters in \eqref{eq:rgo-cloud-scales} by
\[
 h=c\mathfrak L^{-3/4}\,(\eta^2D)^{-1/4}\,,\qquad
 \sigma_0^2=C_\sigma\mathfrak L h^2\,.
\]
Keep the same formulas for $\tau,N,\sigma_n^2,\theta_n$.
The direct RGO conditions still hold, since
\begin{align*}
 h&\le Cc\mathfrak L^{-1/2}\,,\qquad
 \sigma_0^2\le C_\sigma c^2/c_0\,,\qquad
 N\tau/\sigma_0^2\le Cc_\tau\,,\\
 \theta_n/\sigma_n^2
 &=C_{\rm cl}c^4\mathfrak L^{3/2}D^{-1/2}\le1/8\,,\qquad
 \eta\sigma_n^2\,\{\sqrt{d\mathfrak L}+\mathfrak L\}\le Cc^2\,.
\end{align*}
\cref{prop:reference-Picard} gives
$B_{J,2}\le C_0(C\sqrt J)^J$, so
$h/\sigma_n\le C_\sigma^{-1/2}\mathfrak L^{-1/2}$ suffices for the
same quadrature accuracy.
Also, $\eta h^2=c^2\mathfrak L^{-3/2}D^{-1/2}$ makes the
finite-depth error exponentially small with the same $K$.
The interior fluctuation bound is unchanged, while substitution
in \cref{lem:rgo-cloud-bounds} gives
\begin{align*}
\clipB_{\rm mean}
 &\le Cc^2\mathfrak L^{1/4}D^{-1/4}\le Cc^2\mathfrak L^{-2}\,,\\
 \clipB_{{\rm fluc},{\rm end}}
 &\le Cc^3\eta^{1/2}\mathfrak L^{-7/4}D^{-1/4}\,,\\
 \clipB_{{\rm mean},{\rm end}}
 &\le Cc^{7/2}\eta^{1/4}\mathfrak L^{11/8}D^{-3/8}\,,\\
 \clipB_{\rm RGO}
 &\le Cc^4\mathfrak L^{3/4}D^{-1/4}\le1\,.
\end{align*}
Using $D\ge\mathfrak L^9$, these bounds verify
\eqref{eq:cloud-branching-smallness}.
The same accuracy argument therefore applies, and the total
query count is $C\,(1+N\mathfrak L^2)
\le C\,\{1+\mathfrak L^{15/4}\,(\eta^2D)^{1/4}\}$, as claimed.
\end{proof}

%% file: 5_high_acc_sections/appendix_b_gradient_realization.tex
\section{Realization with gradient queries}\label{sec:gradient-only}
The robust fluctuation theorem, \cref{thm:robust-fluctuation}, takes
approximate Jacobian-vector products and approximate evaluations
as inputs. We now construct these inputs using only $\nabla V$.
The remaining
parts of this appendix implement the mean estimators and RGO sampler with gradients as well.

\subsection{Approximation of a proximal map}

\begin{lemma}[Non-expansive gradient iteration]\label{lem:gradient-prox}
Let $V$ be convex and $1$-smooth, and let $0<\eta\le1$. Define the iteration
\[
 x_0(y)\deq y\,,\qquad
 x_{j+1}(y)\deq \frac{y+\eta\,\{x_j(y)-\nabla V(x_j(y))\}}{1+\eta}
 \quad (0\le j<M)\,.
\]
For a fixed integer $M$, the map $x_M$ is $1$-Lipschitz, and
\begin{equation}\label{eq:gradient-prox-error}
 \begin{aligned}
  \norm{x_M(y)-\prox_{\eta V}(y)}
  &\le \eta\, \Bigl(\frac{\eta}{1+\eta}\Bigr)^M\,
  \norm{\nabla V(y)}\,,\\[0.25em]
  \norm{y-x_M(y)-\eta\nabla V(x_M(y))}
  &\le2\eta\,\Bigl(\frac{\eta}{1+\eta}\Bigr)^M\,
  \norm{\nabla V(y)}\,.
 \end{aligned}
\end{equation}
One evaluation costs $M$ gradient queries.
\end{lemma}
\begin{proof}
    The map ${\mathop{\text{id}}}-\nabla V$ is non-expansive. Thus, the iteration contracts
by $\eta/(1+\eta)$. Its fixed point is $\prox_{\eta V}(y)$, and
monotonicity gives
$\norm{y-\prox_{\eta V}(y)}\le \eta\,\norm{\nabla V(y)}$. This proves the
first error bound; the map ${\mathop{\text{id}}} + \eta\nabla V$ is $(1+\eta)$-Lipschitz, giving the
second. If $x_j$ is
non-expansive, the displayed recursion gives
$\Lip(x_{j+1})\le(1+\eta)/(1+\eta)=1$.
\end{proof}

\subsection{Smoothed secants and compositions}

To build gradient-only implementations, we first smooth each elementary map by
Gaussian convolution and approximate the derivative of the smoothed map by a
centered secant. The approximation is accurate in conditional expectation,
while every realized secant remains odd and has the same Lipschitz bound as
the underlying map. We then propagate these value and derivative
approximations through a finite composition.

For $f:\R^m\to\R^n$, write
\[
 (P_s f)(x)\deq\E f(x+s\xi)\,,\qquad \xi\sim\cN(0,\Id_m)\,.
\]
For fixed $r, s>0$, use the same $\xi$ at both endpoints of
\begin{equation}\label{eq:gradient-secant}
 \mathfrak D_{r,s}f(x;v,\xi)
 \deq\frac{f(x+s\xi+rv)-f(x+s\xi-rv)}{2r}\,.
\end{equation}
\begin{lemma}[Secant bounds]\label{lem:gradient-secant}
If $f$ is $L$-Lipschitz, then
\begin{align*}
 \sup_{x\in\R^m}\norm{P_s f(x)-f(x)}&\le Ls\sqrt m\,,
 \qquad \norm{D(P_s f)(x)-D(P_s f)(x')}_{\op}
 \le \sqrt{\frac{2}{\uppi}}\,Ls^{-1}\,\norm{x-x'}\,.
\end{align*}
For each $\xi$, $\mathfrak D_{r,s} f(x;v,\xi)$ is odd and $L$-Lipschitz as a function of $v$,
and
\begin{equation}\label{eq:gradient-secant-bias}
 \norm{\E\mathfrak D_{r,s}f(x;v,\xi)-D(P_s f)(x)v}
 \le \frac{Lr}{\sqrt{2\uppi}\,s}\,\norm v^2\,.
\end{equation}
If $f$ is a gradient of a convex function, $D(P_s f)$ is symmetric
positive semidefinite. This applies both to $f=\nabla V$ and to
$f=\prox_{aV}$.
\end{lemma}
\begin{proof}
    The first two assertions are standard.
    Also, $\E \mathfrak D_{r,s} f(x; v,\xi)$ is the average of
$D(P_s f)(x+urv)v$ over $-1\le u\le1$, proving
\eqref{eq:gradient-secant-bias}. The other assertions are also standard.
\end{proof}

Evaluate $x_M$ from \cref{lem:gradient-prox} in place
of $\prox_{\eta V}$ to generate a proximal secant. For $\norm{x-x_{\rm ref}}\le R$,
Gaussian integration by parts and \eqref{eq:gradient-prox-error} give
\begin{equation}\label{eq:gradient-prox-smoothed-derivative}
 \norm{D P_s(x_M-\prox_{\eta V})(x)}_{\op}
 \le \frac{C\eta}{s}\,
 \Bigl(\frac{\eta}{1+\eta}\Bigr)^M\,
 \{\norm{\nabla V(x_{\rm ref})}+R+s\sqrt m\}\,.
\end{equation}
Thus, for fixed $r, s > 0$, a sufficiently large deterministic $M$ suffices to make the error
arbitrarily small. No estimate for
$Dx_M-D\prox_{\eta V}$ is needed.

\begin{lemma}[Gradient implementation of a finite composition]\label{lem:gradient-graph}
Let $\Phi:\R^m\to\R^n$ be evaluated by the following $S\ge1$ steps:
\[
 z_0\deq g\,,\qquad
 y_i\deq c_i+\sum_{j=0}^{i-1}A_{i,j}z_j\,,\qquad
 z_i\deq f_i(y_i)\quad(1\le i\le S)\,,\qquad \Phi(g)\deq z_S\,.
\]
Here, the matrices $A_{i,j}$ and vectors $c_i$ are known.
More generally, it suffices that the vectors $c_i+A_{i,0}g$ be
computable at every input $g$ where the routines below are called,
and that known upper bounds on $W$ and $R$ below be available;
$c_i$ and $g$ need not be available separately.
Each $f_i$ is
an identity map, a gradient $\nabla V_i$, or a proximal map
$\prox_{\eta_iV_i}$ with $0<\eta_i\le1$. Each $V_i$ is convex with
$1$-Lipschitz gradient, and evaluations of $\nabla V_i$ are available.
Thus, every $f_i$ is $1$-Lipschitz.

Assume the following explicit bound on the coefficients, for some $L>0$.
For every choice
of square matrices $B_i$ of the appropriate dimensions with
$\norm{B_i}_{\op}\le1$, the matrices
\begin{equation}\label{eq:gradient-composition-bound}
 T_0\deq\Id_m\,,\qquad
 T_i\deq B_i\sum_{j=0}^{i-1}A_{i,j}T_j\quad(1\le i\le S)
 \qquad\text{satisfy}\qquad \norm{T_S}_{\op}\le L\,.
\end{equation}
In particular, $\Phi$ is $L$-Lipschitz.

Let $m_0$ bound all vector dimensions. Let $W\ge2$ bound
$\norm{c_i}$, $\sum_{j=0}^{i-1}\norm{A_{i,j}}_{\op}$, and
$\norm{\nabla V_i(0)}$ whenever $V_i$ occurs. Fix a
radius $R\ge1$, an accuracy $0<\epsilon<1/2$, and put
\[
 \mathfrak L_*\deq\log\frac{SWm_0\,(1+R)\,(L+L^{-1})}{\epsilon}\,.
\]
There are a smooth function $\Phi^\circ$ and a deterministic approximation
$\widehat\Phi$, both $L$-Lipschitz, such that
\[
 \sup_{g\in\R^m:\,\norm g\le R}\,
 \bigl\{\norm{\Phi^\circ(g)-\Phi(g)}
       +\norm{\widehat\Phi(g)-\Phi(g)}\bigr\}\le\epsilon\,.
\]
The function $\Phi^\circ$ is defined by replacing each $f_i$ in the
recursion by a Gaussian convolution; it is used only for comparison.
The approximation $\widehat\Phi$ uses only gradients.

There are also gradient routines $\mathfrak J_g(v)$ and
$\mathfrak J_g^\T(w)$ approximating a Jacobian-vector product and a
transposed Jacobian-vector product, respectively. For
$g,v\in\R^m$ and $w\in\R^n$ with norms at most $R$,
\begin{align*}
 \norm{\E\mathfrak J_g(v)-D\Phi^\circ(g)v}
 &\le\epsilon L\,\norm v\,, \qquad
 \norm{\E\mathfrak J_g^\T(w)-D\Phi^\circ(g)^\T w}
 \le\epsilon L\,\norm w\,.
\end{align*}
The expectations are over the routines' auxiliary Gaussians, with $g$
and the direction fixed. For every fixed choice of these Gaussians,
both routines are odd and $L$-Lipschitz in the direction on the whole
space. Each evaluation of $\widehat\Phi$, $\mathfrak J_g$, or
$\mathfrak J_g^\T$ costs at most $CS^3\mathfrak L_*$ gradient queries.
\end{lemma}
\begin{proof}
We first describe the three routines, then choose their accuracies.
All proximal evaluations use the same fixed number $M$ of iterations
from \cref{lem:gradient-prox}. Write $\widehat f_i$ for the resulting
approximation to $f_i$; identities and gradients are evaluated exactly.
Every $\widehat f_i$ remains $1$-Lipschitz.

\emph{Values and the comparison function.}
Evaluate $\widehat\Phi(g)$ by substituting $\widehat f_i$ for $f_i$ in
the displayed recursion. Under the generalized access
assumption, compute each stored input as
$\widehat y_i=(c_i+A_{i,0}g)+\sum_{j=1}^{i-1}A_{i,j}\widehat z_j$.
The forward and reverse routines below use only these stored inputs
and the known matrices, so they do not require $c_i$ or $g$
separately. Store the computed vectors $\widehat y_i$.
For positive radii $s_i$, define
\[
 f_i^\circ\deq P_{s_i}f_i\,,\qquad
 z_0^\circ\deq g\,,\qquad
 y_i^\circ\deq c_i+\sum_{j=0}^{i-1}A_{i,j}z_j^\circ\,,\qquad
 z_i^\circ\deq f_i^\circ(y_i^\circ)\quad(1\le i\le S)\,.
\]
Set $\Phi^\circ(g)\deq z_S^\circ$. Each $f_i^\circ$ is smooth and
$1$-Lipschitz, and its derivative is symmetric by
\cref{lem:gradient-secant}. No Gaussian expectation in this definition
needs to be computed.

\emph{Jacobian-vector products.}
At the stored input $\widehat y_i$, use the randomized map
\[
 \mathcal D_i(u)\deq
 \mathfrak D_{r_i,s_i}\widehat f_i(\widehat y_i;u,\xi_i)\,,
 \qquad \xi_i\sim\cN(0,\Id)\,.
\]
For an identity map, simply take $\mathcal D_i(u)=u$.
For a forward product, compute
\[
 v_0\deq v\,,\qquad
 v_i\deq \mathcal D_i\Bigl(\sum_{j=0}^{i-1}A_{i,j}v_j\Bigr)
 \quad(1\le i\le S)\,,\qquad \mathfrak J_g(v)\deq v_S\,.
\]
For a transposed product, initialize $w_S\deq w$ and $w_j\deq 0$ for
$0\le j<S$. For $i=S,S-1,\ldots,1$, compute $u_i\deq \mathcal D_i(w_i)$
and update
\[
 w_j\gets w_j+A_{i,j}^\T u_i\quad(0\le j<i)\,.
\]
Return $\mathfrak J_g^\T(w)\deq w_0$. Draw a fresh $\xi_i$ when its step
is reached, independently of the incoming direction. In the reverse
calculation, all contributions to $w_i$ are therefore added before
$\xi_i$ is drawn. Replacing $\mathcal D_i$ by
$Df_i^\circ(y_i^\circ)$ gives exactly the two chain rule calculations;
symmetry of this derivative is what permits the reverse calculation
to use the same kind of secant.

Every $\mathcal D_i$ is odd and $1$-Lipschitz for fixed $\xi_i$.
To check the Lipschitz bound for the full forward routine, compare two
directions. At each step, its output difference is some matrix of norm
at most one applied to its input difference. Thus,
\eqref{eq:gradient-composition-bound} gives the bound $L$.
The reverse differences obey the transposed calculation with matrices
of norm at most one, so the same assumption gives the same bound.
Oddness follows directly from the recursions. The identical argument
for differences of values proves that $\Phi$, $\widehat\Phi$, and
$\Phi^\circ$ are $L$-Lipschitz.

\emph{Accuracy and cost.}
We give the parameter choices to explain why no modulus of continuity of
the Hessians is needed. Let $B\deq(8W)^{4(S+1)}\,m_0\,(1+R)$.
Repeated triangle inequalities in the finite recursions give the
following bounds. All stored inputs and all intermediate directions
have norm at most $B$ when the initial inputs and directions have norm
at most $R$. A value error at any step is amplified by at most
$(2W)^{S+1}$ on its way to the output. Since
$\norm{P_{s_i}f_i-f_i}_{\sup}\le s_i\sqrt{m_0}$ and a proximal
error is at most $2^{-M}\,\norm{\nabla V_i(y)}$, increasing radii
$s_1\le\cdots\le s_S\le1$ give
\begin{align*}
 \norm{\Phi^\circ(g)-\Phi(g)}+
 \norm{\widehat\Phi(g)-\Phi(g)}
 &\le B\,(s_S+2^{-M})\,,\qquad
 \norm{\widehat y_i-y_i^\circ}
 \le B\,(s_{i-1}+2^{-M})\,,
\end{align*}
for $1 \le i \le S$,
where $s_0\deq0$. The factor $B$ includes the sum of the errors from
all earlier steps and the bound
$\norm{\nabla V_i(y)}\le W+\norm y$.

There are three errors in the mean of $\mathcal D_i(u)$: the secant
error, the replacement of a proximal map by its $M$-step approximation,
and the use of $\widehat y_i$ in place of $y_i^\circ$.
By \cref{lem:gradient-secant} and
\eqref{eq:gradient-prox-smoothed-derivative}, for $\norm u\le B$ their sum
is bounded by
\[
 \norm{\E\mathcal D_i(u)-Df_i^\circ(y_i^\circ)u}
 \le CB\,\Bigl(\frac{r_i}{s_i}
       +\frac{2^{-M}}{s_i}
       +\frac{s_{i-1}}{s_i}\Bigr)\,\norm u\,.
\]
This explains the increasing smoothing radii: an earlier value error
must be small compared with the smoothing radius at the step where it
is differentiated.

Choose
\[
 t\deq\frac{\epsilon\min\{1,L\}}{CB^2}\,,\qquad
 s_i\deq t^{S-i+1}\,,\qquad r_i\deq t s_i
 \quad(1\le i\le S)\,,\qquad
 M\deq\lceil(S+1)\log_2(1/t)\rceil\,.
\]
For a sufficiently large universal $C$, the value error is at most
$\epsilon$, and the error of each mean derivative is at most
$(\epsilon L/B)\,\norm u$. To propagate this last error, condition on
the incoming direction at each step. The exact matrix
$Df_i^\circ(y_i^\circ)$ is deterministic, so it commutes with taking
expectations. Subtracting the exact forward or reverse recursion and
applying the triangle inequality shows that errors of size
$e\,\norm u$ at each step give output error at most $Be$ times the
initial direction norm. Indeed, the direction entering a step and an
error leaving that step are each amplified by at most $(2W)^{S+1}$,
and $S\,(2W)^{2S+2}\le B$. Taking $e=\epsilon L/B$ proves both mean
bounds.

Finally, $\log B\le CS\mathfrak L_*$ and $\log(1/t)\le CS\mathfrak L_*$, so
$M\le CS^2\mathfrak L_*$. A value calculation uses at most $S$ proximal or
gradient evaluations. Either derivative calculation first stores the
values and then uses at most two such evaluations per step to form the
secants. Its cost is therefore at most $CSM\le CS^3\mathfrak L_*$.
The iteration count is fixed before any input or direction is drawn,
and is the same at both endpoints of every proximal secant.
\end{proof}

For the application in this paper, the input $g$ collects the Gaussian
vectors used in one correction block. For example, a gradient evaluated
at a Gaussian perturbation of a proximal point has the form $\Phi(G,G')
 =\nabla V(\prox_{\eta V}(c+\sqrt\theta\,G)
                    +\sqrt\eta\,G')$.
It consists of two steps. The matrices in
\eqref{eq:gradient-composition-bound} give $\norm{B_2\,\begin{bmatrix}\sqrt\theta\,B_1& \sqrt\eta\,\Id\end{bmatrix}}_{\op}
 \le\sqrt{\theta+\eta}$.
Thus the lemma preserves the joint Gaussian Lipschitz bound used by the
fluctuation correction.
For a cloud block, the center $c$ is hidden, but
$c+\sqrt\theta\,G$ is the observed cloud point, or its surrogate from
\cref{lem:hidden-batch} at a derivative evaluation point, so the
computability condition of \cref{lem:gradient-graph} holds. Suppose
that the observed points and Gaussian labels have norm at most
$\exp(C\mathfrak L)$ relative to a fixed reference point, and that the
hidden center $c$ satisfies the same bound. For $\theta>0$, these
bounds imply that the hidden
standardized cloud input has norm at most
$2\exp(C\mathfrak L)/\sqrt\theta$. Thus known bounds
$W,R\le\exp(C'\mathfrak L)/\sqrt\theta$ suffice. The stated cloud
variances have $\log(1/\theta)=O(\mathfrak L)$, giving
$\mathfrak L_*=O(\mathfrak L)$.
For longer Picard calculations, the integrated
quadrature bounds \eqref{eq:Cheb-bounds} and
\eqref{eq:integrated-weights-l1} supply the corresponding coefficient bounds.

\subsection{Gradient implementation of the fluctuation correction}

Use the functions $\Phi$, $\Phi^\circ$, and $\widehat\Phi$ from
\cref{lem:gradient-graph}. The forward and reverse secant routines are the
inputs of \cref{lem:approximate-derivatives}. Their mean errors are
relative to $D\Phi^\circ$ and $(D\Phi^\circ)^\T$, while their
pathwise oddness and Lipschitz bounds hold for every direction.

We use the following integrated bias variant of
\cref{thm:robust-fluctuation}. When derivative accuracy is proved only
on prescribed input and direction regions, retain the pathwise oddness
and Lipschitz bounds, but replace \eqref{eq:robust-action-bias} by
\begin{equation}\label{eq:robust-integrated-bias}
 \sum_{P\in\{Q,P_K\}}
 \E_P\abs{\E[\widetilde W_K\mid B,Z]-w_K(B,Z)}
 \le\eps_{\rm bias}\,.
\end{equation}
The proof of \cref{thm:robust-fluctuation} then replaces
$CnL^2\varepsilon$ by $C\eps_{\rm bias}$ in
\eqref{eq:robust-fluctuation-error}. Outside the prescribed regions,
the omitted bias must be bounded in expectation using the pathwise
linear growth of the actions. The value errors in
\eqref{eq:robust-value-errors} similarly include their tails. By coupling,
a stopping event of probability $q$ and replacements of conditional Gaussian
draws by kernels with TV errors $\delta_j$ add at most
$q+\sum_j\delta_j$ to the output TV error.

\begin{proposition}[Gradient realization of the fluctuation correction]
\label{prop:gradient-fluctuation}
Let $\Phi(G)=\sum_{j=1}^JA_j\Phi_j(G_j)$ have the block form
of \cref{sec:reusable-proofs}, where each block map $\Phi_j$ satisfies
\cref{lem:gradient-graph} with at most $S$ steps and precision
parameter at most $\mathfrak L_*$; a single map is the case $J=1$.
Assume the
smallness, clipping, and finite series conditions of
\cref{thm:robust-fluctuation} at the prescribed local
TV accuracy. Then, its fluctuation correction has a gradient
implementation with that accuracy and the same fluctuation moment
scale. Each value or derivative evaluation of a single block
costs at most $CS^3\mathfrak L_*$ gradient queries.
\end{proposition}
\begin{proof}
Combine \cref{lem:gradient-graph,lem:approximate-derivatives}, using
independent seeds for the inner and outer routines and shared seeds for
the paired directions in \eqref{eq:robust-fluctuation-tilt-estimator}.
The resulting fluctuation actions are odd and $L^2$-Lipschitz, so they
have the moment bound in \cref{thm:robust-fluctuation}.
Build the approximations block by block, and use the block
construction of $\widehat\Sigma$ from \cref{sec:reusable-proofs}: a
covariance action calls only the selected block $I$, whereas a full
proposal value evaluates all $J$ blocks. Apply
\cref{lem:approximate-derivatives} to block $I$, using independent
seeds for its two derivative routines. For each retained block, put
$p_j=L_j^2\norm{A_j}_{\op}^2/L^2$, where
$L^2=\sum_{j=1}^J L_j^2\norm{A_j}_{\op}^2$.
On the prescribed input and direction regions, the bias after
reweighting is bounded by
\[
 p_j^{-1}\norm{A_j}_{\op}^2\,
       2\varepsilon L_j^2\,\norm v
 =2\varepsilon L^2\,\norm v\,.
\]
Its pathwise Lipschitz bound is
$p_j^{-1}\norm{A_j}_{\op}^2L_j^2=L^2$.
Thus block selection introduces no additional factor in either
bound. The complementary regions are handled by the integrated
bias estimate below. Choose the block value tolerances so that
their sum weighted by $\norm{A_j}_{\op}$ fits the total value error
budget in \eqref{eq:robust-value-errors}.

It remains to account for the bounded regions in
\cref{lem:gradient-graph}. Cap the number of calls and the Gaussian input
and direction radii, and let $\mathcal C$ be the event that all caps hold.
Choose the caps so that the stopping probability, the expected linear value
tails, and the expected quadratic tails of $Z$ and $Z'$ on $\mathcal C^c$
under both $Q$ and $P_K$ fit their accuracy budgets. The global Lipschitz bounds justify these
choices under $Q$. Under $P_K$, use the bounded Brownian path density,
conditional input covariance at most $\Id_m$, and conditional direction
covariance at most $\Id_n$ established in the proof of
\cref{thm:robust-fluctuation}.
On the chosen regions, \cref{lem:gradient-graph} and
\eqref{eq:robust-power-bias} control the value and derivative errors.
Thus, taking the tolerance in \cref{lem:gradient-graph} sufficiently small
gives the required bounds in \eqref{eq:robust-value-errors} and
\eqref{eq:robust-integrated-bias},
with only logarithmic dependence on the dimensions, orders, and radii
in the precision parameter.

The integrated bias version of \cref{thm:robust-fluctuation} now gives
the prescribed TV accuracy. The cost bound in
\cref{lem:gradient-graph} applies to each block evaluation;
a full proposal sums the value costs over all $J$ blocks.
\end{proof}

\subsection{Gradient implementation of the mean estimator}
\label{sec:gradient-marginal}\label{sec:gradient-relative-Stein}
We now implement the mean estimator from \cref{sec:gaussian-stein}.
Let $0\le\theta\le\eta\le1$, with $\eta>0$, and suppose that independent
samples $\cloud\sim\cN(y,\theta\Id)$ and $X\sim R_{\eta,y}$ are supplied.
The point $y$ is unknown. Our goal is to estimate
$g_\eta(y)-\E\overline g_\eta(\cloud)$ using gradients, with an arbitrarily
small prescribed bias and the moment bound of \cref{prop:marginal-score}.
We follow the two terms of \eqref{eq:marginal-cloud-identity}, using the
notation $Z_{\cloud},F_{\mathsf c},A_{\mathsf c},B_{\mathsf c},W_t,q$, and $Z_q$ introduced in
\cref{sec:gaussian-stein}.

To use the secant bounds from \cref{lem:gradient-secant}, first smooth the
test function. For $s_{\nabla}>0$, put
\[
\begin{gathered}
 F_{\mathsf c,s_{\nabla}}(u)\deq (P_{s_{\nabla}}\nabla V)
   (\prox_{\eta V}(\mathsf c)+\sqrt\eta\,u)\,, \qquad
 A_{\mathsf c,s_{\nabla}}\deq D(-\mathcal L_{\mathsf{OU}})^{-1}
 (F_{\mathsf c,s_{\nabla}}-\E_{\cN(0,\Id)}F_{\mathsf c,s_{\nabla}})\,.
\end{gathered}
\]
Apply the proof of \cref{thm:marginal-cloud} with this test function,
keeping the original RGO law and its relative score. For an independent
$G\sim\cN(0,\Id)$, it gives
\begin{equation}\label{eq:gradient-mean-identity}
 \begin{aligned}
 &\E (P_{s_{\nabla}}\nabla V)(X)
   -\E (P_{s_{\nabla}}\nabla V)(\prox_{\eta V}(\cloud)+\sqrt\eta\,G)\\
 &\qquad=-\E\Bigl[\sqrt\eta\,A_{\cloud,s_{\nabla}}(Z_{\cloud})\,
       \{\nabla V(X)-\nabla V(\prox_{\eta V}(\cloud))\}
       +\frac{\theta}{\sqrt\eta}\divg_{\cloud} A_{\cloud,s_{\nabla}}(Z_{\cloud})\Bigr]\,.
 \end{aligned}
\end{equation}
The left-hand side differs from the desired mean by at most $2s_{\nabla}\sqrt d$.

For the first term, replace the Hessian-vector product in
\eqref{eq:OU-HVP} by a gradient secant. Draw independent
$T\sim\mathsf{Exp}(1)$ and $G,\xi_{\nabla}\sim\cN(0,\Id)$, using $W_t$ from
\eqref{eq:divergence-formula} with $\mathsf c=\cloud$ and $x=X$.
For $r_{\nabla}>0$, set
\begin{equation}\label{eq:gradient-OU-action}
 \widehat A_{\cloud}^{\rm sec}(Z_{\cloud},v)
 \deq\sqrt\eta\,\mathfrak D_{r_{\nabla},s_{\nabla}}\nabla V(W_T;v,\xi_{\nabla})\,.
\end{equation}
Conditionally on $\cloud,X$ and the direction $v$, \cref{lem:gradient-secant}
and the OU representation of $A_{\cloud,s_{\nabla}}$ give
\begin{align*}
 \norm{\E[\widehat A_{\cloud}^{\rm sec}(Z_{\cloud},v)\mid \cloud,X,v]
             -A_{\cloud,s_{\nabla}}(Z_{\cloud})v}
 &\le C\sqrt\eta\,\frac{r_{\nabla}}{s_{\nabla}}\,\norm v^2\,,\qquad
 \norm{\widehat A_{\cloud}^{\rm sec}(Z_{\cloud},v)}
 \le\sqrt\eta\,\norm v\,.
\end{align*}
For every realization, the map is odd and $\sqrt\eta$-Lipschitz in $v$.
Thus, with direction
$v=\sqrt\eta\,\{\nabla V(X)-\nabla V(\prox_{\eta V}(\cloud))\}$,
the first term has the same size as in
\cref{prop:marginal-score}, and its bias tends to zero with $r_{\nabla}/s_{\nabla}$.
The auxiliary draws are independent of the direction.

For the second term, use the divergence formula from
\cref{sec:gaussian-stein} with the same smoothed test function:
\begin{equation}\label{eq:gradient-divergence-target}
 \divg_{\cloud} A_{\cloud,s_{\nabla}}(Z_{\cloud})
 =\int_0^\infty q(t)\,
 \E_G[D(P_{s_{\nabla}}\nabla V)(W_t)\,B_{\cloud}G]\,\dd t\,.
\end{equation}
Here, the expectation fixes $\cloud,X$.
This identity has already removed the unknown $y$ by integration by parts.
We approximate its two matrix-vector products by secants, using fresh
auxiliary draws for this term. Draw $T$ with density $q/Z_q$ and
independent standard Gaussians $G,\xi_{\mathrm{prox}},\xi_{\nabla}$. For $s_{\mathrm{prox}},r_{\mathrm{prox}}>0$, form
\begin{equation}\label{eq:gradient-divergence-estimator}
 \widehat{\mathcal D}_{\cloud}^{\rm sec}
 \deq Z_q\,\mathfrak D_{r_{\nabla},s_{\nabla}}\nabla V\bigl(
 W_T;\mathfrak D_{r_{\mathrm{prox}},s_{\mathrm{prox}}}\prox_{\eta V}(\cloud;G,\xi_{\mathrm{prox}}),\xi_{\nabla}\bigr)\,.
\end{equation}
Since both $\prox_{\eta V}$ and $\nabla V$ are non-expansive,
$\norm{\widehat{\mathcal D}_{\cloud}^{\rm sec}}\le Z_q\,\norm G$ for every realization.

In both terms, every occurrence of $\prox_{\eta V}$, including those defining
$W_T$, is evaluated by $x_M$ from \cref{lem:gradient-prox}, using the same
deterministic iteration count at both endpoints of every secant. The map
$x_M$ is non-expansive, so the preceding norm bounds remain valid. The remaining issue is the bias of
the divergence estimator: we must control the error after averaging over
$\cloud$, since $B_{\cloud}=D\prox_{\eta V}(\cloud)$ need not have a known modulus of
continuity.

\begin{lemma}[Averaged secant divergence]\label{lem:gradient-divergence}
Suppose $\cloud\sim\cN(y,\theta\Id)$, with $0<\theta\le\eta$, and $X$ is
independent of $\cloud$. The mean of \eqref{eq:gradient-divergence-estimator}
differs from the expectation of \eqref{eq:gradient-divergence-target} by at
most
\begin{equation}\label{eq:gradient-divergence-error}
 Cd\,\Bigl\{
 s_{\mathrm{prox}}\,(\theta^{-1/2}+\eta^{-1/2})
 +\frac{r_{\mathrm{prox}}}{s_{\mathrm{prox}}}+\frac{r_{\nabla}}{s_{\nabla}}
 +\varepsilon_{\rm prox}\,(s_{\mathrm{prox}}^{-1}+s_{\nabla}^{-1})
 \Bigr\}\,.
\end{equation}
Here, $\varepsilon_{\rm prox}$ uniformly bounds
$\E[(1+\norm\xi)\,\norm{x_M(U)-\prox_{\eta V}(U)}]$ over every proximal input
$U$ and associated smoothing Gaussian $\xi$ used in the estimator, taking
$\xi=0$ for unweighted calls. The bound is uniform in $y$ and in the law of $X$.
\end{lemma}
\begin{proof}
Condition first on $(T,\cloud,X,G,\xi_{\mathrm{prox}})$ and average $\xi_{\nabla}$. Then average $\xi_{\mathrm{prox}}$.
\Cref{lem:gradient-secant} and the pathwise bound on the inner secant show
that the mean is
\[
 Z_q\,\E[D(P_{s_{\nabla}}\nabla V)(W_T)\,
 D(P_{s_{\mathrm{prox}}}\prox_{\eta V})(\cloud)G]
\]
up to $Cd\,(r_{\nabla}/s_{\nabla}+r_{\mathrm{prox}}/s_{\mathrm{prox}})$. The gradient iteration errors add the last term
of \eqref{eq:gradient-divergence-error}, by
\eqref{eq:gradient-prox-smoothed-derivative} and
$\Lip(D(P_{s_{\nabla}}\nabla V))\le C/s_{\nabla}$.

It remains to replace
$D(P_{s_{\mathrm{prox}}}\prox_{\eta V})(\cloud)
=\E_\xi D\prox_{\eta V}(\cloud+s_{\mathrm{prox}}\xi)$
by $B_{\cloud}$ after
averaging $\cloud$. We give the weak estimate explicitly. For any matrix $B$ with
$\norm B_{\op}\le1$ and any matrix-valued $H$ with $\norm H_{\op}\le1$,
Gaussian integration by parts gives
\[
 \Lip_w\bigl(\E_G[H(w+\sigma G)\,BG]\bigr)\le C\sqrt d/\sigma
 \qquad(\sigma>0)\,.
\]
To see this, its derivative in direction $v$ is the expectation of
$\sigma^{-1}H(w+\sigma G)\,\{\ip Gv BG-Bv\}$; Cauchy--Schwarz proves the bound.
It remains true for bounded measurable $H$ by convolution.

Fix a shift $e$. In the integral over $\cloud$, make the change of variables
$\mathsf c'=\mathsf c+s_{\mathrm{prox}}e$. The Gaussian density changes in $L^1$ by at most
$Cs_{\mathrm{prox}}\,\norm e/\sqrt\theta$, and the integrand has expected norm at most
$\sqrt d$. The remaining change in the coefficient is a translation of
the mean of $W_t$ by at most $(1-\e^{-t})\,s_{\mathrm{prox}}\,\norm e$, because $\prox_{\eta V}$ is
non-expansive. Apply the preceding bound with Gaussian scale
$\sqrt{\eta\,(1-\e^{-2t})}$, $H=D(P_{s_{\nabla}}\nabla V)$, and
$B=D\prox_{\eta V}(\mathsf c')$.
After averaging $e\sim\cN(0,\Id)$, the two errors are at most
$Cds_{\mathrm{prox}}/\sqrt\theta$ and
$Cds_{\mathrm{prox}}\,(1-\e^{-t})/(\sqrt\eta\sqrt{1-\e^{-2t}})$.
Both are integrable against $q(t)\,\dd t$, with respective integrals
$Cds_{\mathrm{prox}}/\sqrt\theta$ and $Cds_{\mathrm{prox}}/\sqrt\eta$.
This proves \eqref{eq:gradient-divergence-error}.
\end{proof}

\begin{corollary}[Mean estimator using gradients]\label{cor:gradient-marginal}
The estimator in \cref{prop:marginal-score} has a gradient implementation
with arbitrarily small prescribed mean error and the same $L^p$ bound,
up to a universal factor and the prescribed error. Apart from
its RGO draws, for a prescribed mean error $\delta$ with
$\log(1/\delta)=O(\mathfrak L)$, it uses $C\mathfrak L$ gradient
queries whenever, for some point $\bar x$ with
$\norm{\nabla V(\bar x)}\le\exp(C_0\mathfrak L)$, the vectors
$\cloud-\bar x$, $X-\bar x$, and its auxiliary Gaussian vectors have norm
at most $\exp(C_0\mathfrak L)$; here $C_0$ is a fixed constant and
$C$ depends on $C_0$.
\end{corollary}
\begin{proof}
With the proximal substitutions just described, use
\eqref{eq:gradient-OU-action} with its direction
$\sqrt\eta\,\{\nabla V(X)-\nabla V(\prox_{\eta V}(\cloud))\}$ radially clipped
to a radius $R_{\rm sec}$, and subtract both that output and
$\theta\widehat{\mathcal D}_{\cloud}^{\rm sec}/\sqrt\eta$. Use independent
auxiliary draws for the two terms.
Clipping changes the mean of the first term by at most
$\eta\,\E\bigl[\norm{\nabla V(X)-\nabla V(\prox_{\eta V}(\cloud))}\,
\mathbf 1_{\{\sqrt\eta\,\norm{\nabla V(X)-\nabla V(\prox_{\eta V}(\cloud))}>R_{\rm sec}\}}\bigr]$.
Choose $R_{\rm sec}$ so this fits the prescribed budget.

By \eqref{eq:gradient-mean-identity}, it remains to approximate the two
smoothed terms. Choose $s_{\nabla}$ first to make $2s_{\nabla}\sqrt d$ small enough,
then $s_{\mathrm{prox}}$ to control the first term of
\eqref{eq:gradient-divergence-error}. Next choose $r_{\mathrm{prox}},r_{\nabla}$ to control
the secant errors in \eqref{eq:gradient-OU-action} and
\eqref{eq:gradient-divergence-estimator},
and finally the proximal value tolerance $\varepsilon_{\rm prox}$.
The proximal errors in the first term tend to zero as well, by
\eqref{eq:gradient-prox-error} and
$\Lip(D(P_{s_{\nabla}}\nabla V))\le C/s_{\nabla}$.
When $\theta=0$, omit the divergence term and its parameters.

The first term has norm at most
$\eta\,\norm{\nabla V(X)-\nabla V(\prox_{\eta V}(\cloud))}$, and the second
has norm at most $(\theta/\sqrt\eta)\,Z_q\,\norm G$.
The same moment estimates as in \cref{prop:marginal-score} therefore give
\[
 \norm{\widehat\Delta(y)}_{L^p}
 \le C\,\Bigl(\eta^{3/2}+\frac{\theta}{\sqrt\eta}\Bigr)\,
 (\sqrt d+\sqrt p)\,,
\]
up to the allocated proximal error. For an approximate RGO, also cap
the auxiliary Gaussian norms. The estimator is then bounded, so TV
approximation controls its mean error; the Gaussian and score tails
control the clipping errors and preserve the moment bound up to the
prescribed error. Keep the mean cloud and RGO cloud independent, as in
the main construction. Every evaluation uses only $\cloud,X$, gradients,
and fixed proximal iterations; it never evaluates $y$.

For the cost, first include the multiplier
$\theta/\sqrt\eta$ of the divergence estimator. For
$0<\theta\le\eta\le1$,
\[
 \frac{\theta}{\sqrt\eta}\,
    (\theta^{-1/2}+\eta^{-1/2})
 =\sqrt{\theta/\eta}+\theta/\eta\le2\,,
 \qquad \theta/\sqrt\eta\le1\,.
\]
Consequently, the weighted error in
\eqref{eq:gradient-divergence-error}, together with the first-term
and smoothing errors above, can be made at most $\delta$ using
inverse smoothing radii, secant radii, and proximal tolerances
bounded by a fixed polynomial in $d$, $1/\delta$, $R_{\rm sec}$,
and the cap radius. No inverse power of $\theta$ is needed;
when $\theta=0$, the divergence term is omitted. Under the radius
hypothesis of the corollary, the cap radius is $\exp(C_0\mathfrak L)$,
and every proximal or gradient input $u$ satisfies
$\norm{\nabla V(u)}\le\norm{\nabla V(\bar x)}+\norm{u-\bar x}
\le\exp(C\mathfrak L)$, since $\nabla V$ is $1$-Lipschitz. These
logarithms are therefore $O(\mathfrak L)$, and so are those of the
inverse tolerances when $\log(1/\delta)=O(\mathfrak L)$.
By \cref{lem:gradient-prox}, a fixed $O(\mathfrak L)$ proximal
iteration count then suffices. The estimator uses a bounded number
of secants, giving $C\mathfrak L$ gradient queries.
\end{proof}

\subsection{Cost of the local gradient operations}\label{sec:gradient-cost}
We now bound the gradient cost of the local constructions above.
The following estimate is used both for the RGO implementation and for
the full sampler.

\begin{lemma}[Logarithmic cost of the local gradient operations]
\label{lem:gradient-log-cost}
For the Gaussian cloud and improved RGO constructions in \cref{sec:work,sec:rgo}, use their
stated parameters and local accuracy budgets whose logarithms of inverse
accuracies are $O(\mathfrak L)$. Let $\mathcal C$ be the intersection, over
the ideal run, of the local cap events from the proof of
\cref{prop:gradient-fluctuation}. All inputs and directions can be capped at
radius $\exp(C\mathfrak L)$, with total stopping probability and the tail
contributions on $\mathcal C^c$ below the allocated error budget.
Input radii are measured relative to $x_{\rm ref}$ for the
main sampler, and in the normalized coordinates of \cref{sec:rgo}
for an RGO call.
On these caps, each block value or derivative action, and each mean
replica apart from its RGO draws, uses $C\mathfrak L$ gradients.
\end{lemma}
\begin{proof}
First cap the total number of local reference operations
at $\exp(C_1\mathfrak L)$. Their expected counts in
\cref{sec:work,sec:rgo} are $\exp(O(\mathfrak L))$, so Markov's
inequality makes the added stopping probability smaller than any
allocated $\exp(-C_2\mathfrak L)$ budget by increasing $C_1$.

Apply \cref{lem:fifth-run-concentration,lem:fifth-stage-moments}
and their normalized RGO versions to the fixed collection of ideal
phase inputs and stages. Outside an event of the allocated small
probability, their radii are at most $\exp(C_0\mathfrak L)$.
For these inputs, fresh Gaussian draws have uniform conditional
moment bounds. So do exact RGO draws about their proximal centers:
for $Z\sim R_{b,v}$ and $p\ge2$, strong log-concavity gives
\[
 \bigl(\E[\norm{Z-\prox_{bV}(v)}^p\mid v]\bigr)^{1/p}
 \le C\sqrt b\,(\sqrt d+\sqrt p)\,.
\]
Non-expansiveness and the fixed local compositions then control the
derived inputs. Equation \eqref{eq:fixed-random-S-properties} and
$L_{\rm fluc}^2\le1/8$ give the same control for derivative directions.
Apply these bounds conditionally at each of the deterministically
many possible operation slots, and then take a union bound. Choosing
the radius $\exp(C\mathfrak L)$ with $C$ sufficiently large gives
the claimed stopping probability. H\"older's inequality with the
same higher moments controls the linear and quadratic tail terms;
the bounds under the tilted law are those established in
\cref{prop:gradient-fluctuation}.

Each evaluation uses $S=O(1)$ maps in \cref{lem:gradient-graph}, with
$\mathfrak L_*=O(\mathfrak L)$; the $J$ terms of a proposal are counted
separately. Its cost is therefore $C\mathfrak L$ gradients.
The mean estimator uses a fixed number of secants, whose error bounds in
\cref{lem:gradient-divergence,cor:gradient-marginal} require the same
logarithmic precision. \Cref{lem:gradient-prox} gives the same cost for
these evaluations, excluding RGO draws.
\end{proof}

\subsection{Gradient implementation of the RGO}
We now adapt the RGO implementation in \cref{thm:fourth-RGO} using gradient queries. Put
\[
 \mathfrak L\deq\log\frac{d}{\eta\varepsilon}\,,\qquad
 \mathfrak L_y\deq\log\frac{d\,(1+\eta\,\norm{\nabla V(y)})}{\eta\varepsilon}\,.
\]
\begin{proposition}[Gradient RGO]\label{prop:gradient-RGO}
Let $V\in C^2(\R^d)$ satisfy $0\preceq\nabla^2V\preceq\Id$.
For $0<\eta\le1$ and $0<\varepsilon<1/4$, the construction of
\cref{sec:rgo} can sample within $\varepsilon$ in TV of $R_{\eta,y}$ using an
expected number of gradient queries bounded by
\begin{equation}\label{eq:gradient-RGO-cost}
 C\mathfrak L_y+C\mathfrak L^{21/4}\,
 \{\eta^2\,(d+\mathfrak L)\}^{1/4}\,.
\end{equation}
If $d+\mathfrak L\ge\mathfrak L^9$, the exponent $21/4$
in \eqref{eq:gradient-RGO-cost} improves to $19/4$.
\end{proposition}
\begin{proof}
For $R^V_{b,v}$ with $0<b\le1$, compute $x_M(v)$ by
\cref{lem:gradient-prox} with $\eta=b$ and set
\[
 r\deq\nabla V(x_M(v))+(x_M(v)-v)/b\,,\qquad
 \widetilde V(x)\deq V(x)-\ip r{x}\,.
\]
Then $x_M(v)=\prox_{b\widetilde V}(v)$, and the logarithmic Sobolev
inequality and Pinsker's inequality give
\begin{equation}\label{eq:gradient-tilted-RGO}
 \TV(R^{\widetilde V}_{b,v},R^V_{b,v})\le\sqrt b\,\norm r/2\,.
\end{equation}
Whenever \cref{lem:rgo-base} applies, it can use $x_M(v)$ as its exact
proximal point, with $\nabla\widetilde V=\nabla V-r$. Taking $M$ large
enough makes the additional TV error fit any prescribed local budget.

For the initial normalization, apply the preceding construction with
$(b,v)=(\eta,y)$ and choose $x_M(y)$ so that
the error in \eqref{eq:gradient-tilted-RGO} is at most $\varepsilon/4$.
This costs $C\mathfrak L_y$ gradients. Normalize the tilted target about
$x_M(y)$; its residual satisfies $\nabla F(0)=0$.
Implement the corrections of \cref{sec:rgo} by
\cref{prop:gradient-fluctuation,cor:gradient-marginal}, retaining
$\Lip(\nabla F)\le\eta/2$ and hence the mean bound
\eqref{eq:rgo-residual-mean}. Every further RGO call uses the direct
routine above.

Run \cref{thm:fourth-RGO} at accuracy $\varepsilon/4$ and allocate the
remaining $\varepsilon/2$ to replacements and stopping events using
\cref{lem:adaptive-tv-comparison}. This gives the required TV accuracy.
By \cref{lem:gradient-prox,lem:gradient-log-cost}, the gradient implementation
in normalized coordinates multiplies the query bound \eqref{eq:QE4-cost}
by $C\mathfrak L$.
Adding the normalization cost gives \eqref{eq:gradient-RGO-cost}, since
$\mathfrak L\le\mathfrak L_y$.
The large dimension improvement follows by applying the same
argument to the sharper bound in \cref{thm:fourth-RGO}.
\end{proof}

\subsection{Gradient implementation of the Gaussian cloud sampler}
We finally give the full gradient-only implementation of the sampler in
\cref{thm:core}.

\begin{theorem}[Gradient realization of the Gaussian cloud sampler]
\label{thm:gradient-realization}
Under the hypotheses of \cref{thm:core}, \cref{alg:fifth-phase} has a
gradient implementation that returns a law within $\varepsilon$ in TV of $\pi$
using at most $C\mathfrak L^{11/2}\,(d+\mathfrak L)^{1/5}$
expected gradient queries.
\end{theorem}
\begin{proof}
Use \cref{prop:gradient-fluctuation} for fluctuation corrections,
\cref{cor:gradient-marginal} for mean estimators, and
\cref{prop:gradient-RGO} for the explicit RGO draws in the hidden RGO
sampler and terminal lift. The hidden RGO likelihood already uses only
gradients and these draws.
By \cref{lem:gradient-graph}, the joint Lipschitz bound
$\sqrt{\theta+\eta}$ from \cref{lem:joint-estimator-derivatives} is preserved,
so the fluctuation parameter in \eqref{eq:fifth-alpha} is unchanged.
These implementations retain the moment bounds of \cref{sec:work} up to
constants.

Translate $x_{\rm ref}$ to the origin and impose the caps from
\cref{lem:gradient-log-cost}, returning $x_{\rm ref}$
if a cap is exceeded. By $1$-Lipschitzness of $\nabla V$ and
\eqref{eq:reference-point}, on the cap event, every RGO center $y$ satisfies
$\log(1+\eta\,\norm{\nabla V(y)})\le C\mathfrak L$.
The local variance and accuracy choices therefore give
$\mathfrak L_y=O(\mathfrak L)$ in \cref{prop:gradient-RGO}.
Together with \cref{lem:gradient-log-cost}, this multiplies the reference
query count by $C\mathfrak L$, including RGO normalizations.
The bound $CN\mathfrak L^2$ from \cref{sec:work} becomes
$CN\mathfrak L^3$, which gives the stated gradient query cost.

For the local accuracy comparison, first supply exact cloud
and RGO draws. Proposition~\ref{prop:gradient-fluctuation} converts
the value and derivative errors into TV error of the corrected
Gaussian output. For a mean correction in whitened coordinates,
let $\mu$ be the proposal mean, and write $\widehat\Delta$ for the
gradient estimator and $\Delta$ for the exact estimator, with
expectations conditional on the fixed
phase input. Corollary~\ref{cor:gradient-marginal} makes
$e_\Delta\deq\norm{\E\widehat\Delta-\E\Delta}$ arbitrarily small.
Combining \cref{prop:reusable-mean} with the Gaussian shift bound
\[
 \TV\bigl(\cN(\mu+\E\widehat\Delta,\Id),
          \cN(\mu+\E\Delta,\Id)\bigr)\le e_\Delta/2
\]
therefore gives the required local output error. The recursive TV
comparison is applied to these corrected sampling outputs.

The reference estimates in the proof of \cref{thm:core},
which do not use the present theorem, account for at most
$3\varepsilon/16$ of the TV budget. Use \cref{prop:gradient-RGO}
with the same pre-terminal and terminal RGO budgets fixed there;
these RGO errors are already included in $3\varepsilon/16$.
Allocate at most $\varepsilon/2$ to the additional value, derivative,
and mean errors, and at most $\varepsilon/4$ to the additional caps
and tail terms. Distribute these tolerances over the potential
slots counted in \cref{sec:recursive-cloud}; their inverse
accuracy logarithms remain $O(\mathfrak L)$.
The same recursive comparison, using
\cref{lem:adaptive-tv-comparison}, now gives total TV error at most
$3\varepsilon/16+\varepsilon/2+\varepsilon/4
=15\varepsilon/16<\varepsilon$.
The gradient-only proximal composition in \cref{sec:proximal-composition}
completes the implementation of \cref{thm:fifth-root}.
\end{proof}